\documentclass[reqno]{amsart}

\usepackage{a4wide}
\usepackage[utf8]{inputenc}
\usepackage[english]{babel}
\usepackage{amsmath, amsthm, amsfonts, amssymb, enumerate}
\usepackage{mathrsfs}
\usepackage[numbers]{natbib}
\usepackage{xcolor}
\usepackage{graphicx}
\usepackage{hyperref}
\usepackage{cleveref}
\usepackage{verbatim}
\usepackage{ifthen}
\usepackage{subcaption}

\theoremstyle{plain}

\newtheorem{theorem}{Theorem}[section]

\newtheorem{lemma}[theorem]{Lemma}

\newtheorem*{lemma*}{Lemma}
\newtheorem{corollary}[theorem]{Corollary}

\newtheorem*{corollary*}{Corollary}

\newtheorem*{theorem*}{Theorem}
\newtheorem{proposition}[theorem]{Proposition}

\newtheorem*{proposition*}{Proposition}
\theoremstyle{definition}
\newtheorem{definition}[theorem]{Definition}

\newtheorem{example}[theorem]{Example}

\newtheorem{examples}[theorem]{Examples}

\newtheorem{remark}[theorem]{Remark}

\newtheorem*{remark*}{Remark}
\newtheorem*{notation*}{Notation}

\numberwithin{equation}{section}

\newcommand{\id}{\operatorname{id}}

\newcommand{\N}{{\mathbb N}}
\newcommand{\R}{{\mathbb R}}
\newcommand{\Q}{{\mathbb Q}}
\newcommand{\Z}{{\mathbb Z}}

\newcommand{\diam}{\operatorname{diam}}

  \newcommand{\vertii}[1]{{\left\vert\kern-0.25ex\left\vert #1 \right\vert\kern-0.25ex\right\vert}}
 \newcommand{\vertiii}[1]{{\left\vert\kern-0.25ex\left\vert\kern-0.25ex\left\vert #1 
    \right\vert\kern-0.25ex\right\vert\kern-0.25ex\right\vert}}

\newcommand{\exclude}[1]{}

\newcounter{newcl}{}
\definecolorseries{newcol}{rgb}{step}[rgb]{.0,.5,.0}{.3,.1,.8}

\definecolor{grey}{rgb}{0.7, 0.7, 0.7}
\definecolor{darkgreen}{rgb}{0., 0.6, 0.}

\makeatletter
\def\blfootnote{\xdef\@thefnmark{}\@footnotetext}
\makeatother

\title{Measure and dimension theory of permeable sets and its applications to fractals
}

\author[G. Leobacher]{Gunther Leobacher}
\author[T. Rajala]{Tapio Rajala}
\author[A. Steinicke]{Alexander Steinicke}
\author[J. Thuswaldner]{J\"org Thuswaldner}

\address[G.L.]{Institute of Mathematics and Scientific Computing, University of Graz. Heinrichstraße 36, 8010 Graz, Austria.}
\email{gunther.leobacher@uni-graz.at}
\address[T.R.]{Department of Mathematics and Statistics, University of Jyvaskyla,
P.O. Box 35 (MaD), 40014 University of Jyvaskyla, Finland.}
\email{tapio.m.rajala@jyu.fi}
\address[A.S.]{Chair of Applied Mathematics, 
Montanuniversitaet Le\-o\-ben.
Peter-Tunner-Strasse 25/I, 8700 Leoben, Austria.}
\email{alexander.steinicke@unileoben.ac.at}
\address[J.T.]{Chair of Mathematics, Statistics, and Geometry, 
Montanuniversitaet Le\-o\-ben.
Franz-Josef-Strasse 18, 8700 Leoben, Austria. }
\email{joerg.thuswaldner@unileoben.ac.at}

\date{\today}
\thanks{The fourth author is supported by the bilateral project I~6750 granted by the Agence Nationale de la Recherche (ANR) and the Austrian Science Fund (FWF)}

\begin{document}
\resetcolorseries[7]{newcol} %

\maketitle

\begin{abstract}
We study {\it permeable} sets. These are sets $\Theta \subset \R^d$ which have the property that each two points $x,y\in \R^d$ can be connected by a short path $\gamma$ which has small (or even empty, apart from the end points of $\gamma$) intersection with $\Theta$.
We investigate relations between permeability and Lebesgue measure and establish theorems on the relation of permeability with several notions of dimension. It turns out that for most notions of dimension each subset of $\R^d$ of dimension less than $d-1$ is permeable.
We use our permeability result on the Nagata dimension to characterize permeability properties of self-similar sets with certain finiteness properties.
\end{abstract}

\section*{Introduction} 
\subsection*{Context of the paper} In this article we study the geometric concept of permeability and its variants as introduced in \cite{leoste21}. The intuition behind  these concepts is that a line segment in the Euclidean plane, being ``infinitely thin'', poses much less of a barrier than a strip of positive thickness. 
\begin{figure}[h]
     \centering
     \begin{subfigure}[b]{0.3\textwidth}
         \centering
         \includegraphics[width=\textwidth]{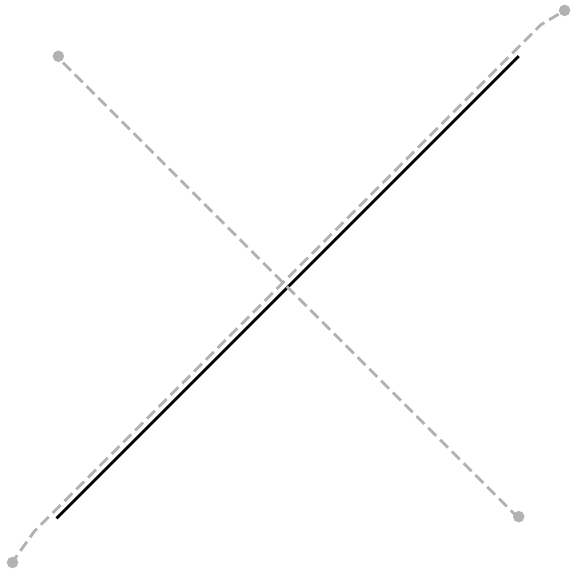}
         \caption{Line segment} %
     \end{subfigure}
     \hfill
     \begin{subfigure}[b]{0.3\textwidth}
         \centering
         \includegraphics[width=\textwidth]{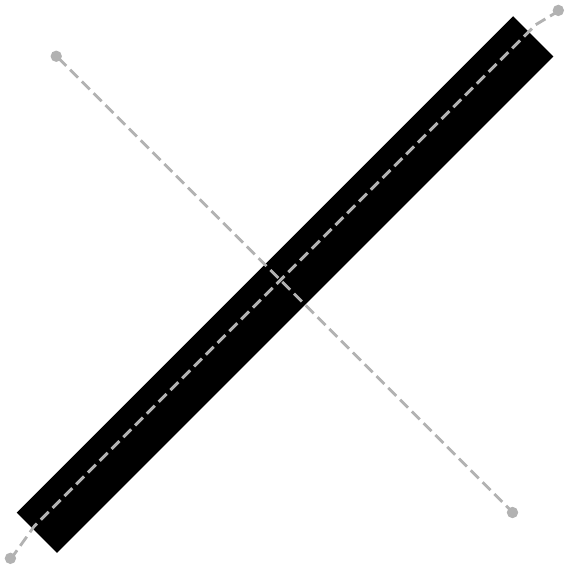}
         \caption{Strip} %
     \end{subfigure}
     \hfill
     \begin{subfigure}[b]{0.32\textwidth}
         \centering
         \includegraphics[width=\textwidth]{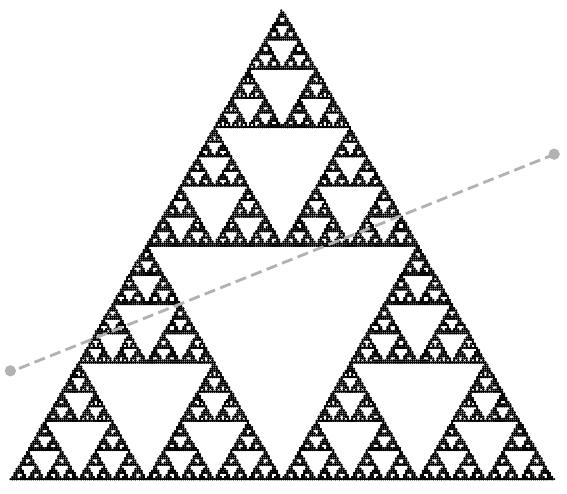}
         \caption{Sierpi\'nski triangle} %
     \end{subfigure}
 	\caption{Connecting points by short paths. (A)~If we allow for small detours, we can connect each pair of points by intersecting the line in at most one point. (B)~For the strip we get an uncountable intersection even with small detours. (C)~Can we always get countable intersections with the Sierpi\'nski triangle with small detours?}
        \label{fig:startex}
\end{figure}
Indeed, if we are willing to take small detours, we may connect each pair of points of the plane by a path that intersects the line segment in at most one point. On the other hand, for the strip of positive thickness, the intersection with a path connecting two points of the plane may remain uncountable, even if we allow for small detours.  A countable subset of the plane or even the middle third Cantor set (when isometrically embedded in the plane) can always be avoided entirely when one is willing to accept small detours from straight lines. According to Marstrand's intersection theorem \cite[Theorem~III]{Marstrand}, straight lines intersecting the Sierpi\'nski triangle often have uncountable intersection (see Figure~\ref{fig:startex}). However, as we shall see later, if we choose two points in the plane, there always exist small detours connecting these points, having only countable intersection with the Sierpi\'nski triangle.

The notions of permeability studied in the present paper provide a framework for classifying sets with respect to their intersections with ``short'' paths that ``pass through'' them. The different ``sizes'' of the intersections considered in the present paper are classified as ``empty'' (for \emph{null permeability}), ``finite'' (for \emph{finite permeability}),  or ``having countable closure'' (for \emph{permeability}). We refer to Definition~\ref{def:perm} for the formal definition of these notions.

The name ``permeability'' appears as a physical quantity in models of flow in porous media, as the {\em permeability coefficient} in  {\em Darcy's law}, {\em cf.}~\cite[Chapter 5]{bear1988dynamics}. In this context, the permeability coefficient describes the ability of a material to conduct a liquid or a gas. Our notion of permeability can be regarded as an abstract version of this physical concept.  

A concept that is closely related to null permeability is {\em tortuosity}, which describes the ability of a particle of a given size to move through a porous medium (see {\em e.g.}~\cite{ghanbarian2013} for different variants and applications of tortuosity). The notion of null permeability considered in our article is a limit case of tortuosity for infinitesimal particle size.

Another related concept is {\it percolation} ({\em cf.}~\cite{grimmett-1999}), in particular {\it Mandelbrot percolation} (also called {\it fractal percolation}, see {\it e.g.}~\cite{broman2013,buczolich2021,rams2015}), which pursues a similar goal as permeability. 
Indeed, in Mandelbrot percolation one iteratively constructs a random {\em retaining set} $\mathcal{C}\subset [0,1]^d$ and studies typical properties of $\mathcal{C}$. In particular, $\mathcal{C}$ {\em percolates} if there exists a connected subset of $\mathcal{C}$ that has nonempty intersection with the ``bottom''
$[0,1]^{d-1}\times\{0\}$ and the ``top''
$[0,1]^{d-1}\times\{1\}$ of the unit cube $[0,1]^{d}$.
While permeability also aims at connecting points, it investigates the existence of {\it short paths between arbitrary points} that do not hit any point of $[0,1]^d\setminus\mathcal{C}$ (or at most finitely or countably many), which is a somewhat different goal. Null permeability can thus be viewed as ``fast percolation everywhere''.  We provide more details on this relation in Remark~\ref{rem:percolation}.

The notion of null permeability  is also connected to the classical notion of negligibility for {\it extremal length} (see {\em e.g.}~\cite{Ahlfors1950} for the definition of extremal length and extremal distance). Recall that a set $E \subset \mathbb R^d$ is {\em negligible} for extremal lengths, if the extremal length of the collection of curves between any two disjoint continua is the same as the extremal length of the subset of the above curves that do not intersect $E$, except possibly at the endpoints. It was shown in \cite{Ahlfors1950}, that a closed set in the complex plane is negligible for extremal distances if and only if it is {\em removable} for conformal embeddings. Many results on the topic appeared in the meantime, in particular the study of negligibility for extremal distance in higher dimensions (see~\cite{Vaisala62}) and its implications  to removability ({\em cf.}~\cite{AseesSyvev74,VoGo77}) are studied. We also refer to \cite{bhatia2024metric} for a related concept that is called ``metric removability''. Very recently a modification of this condition, called negligibility for extremal distances with countable intersections, was considered in \cite{ntalampekos2023cned,ntalampekos2023cned2}. The modification is similar to passing from null permeability to permeability: the extremal length of the curve family joining the continua is required to equal the extremal length of the subcollection of curves where the curves are allowed to have countable intersection with $E$. According to \cite{ntalampekos2023cned2}, this condition for closed sets implies removability for quasiconformal homeomorphisms.

Our notion of permeability has applications in the analysis of Lipschitz maps. Given a continuous function on $\R^d$, which has bounded gradient on the complement of a small set $\Theta\subset \R^d$, 
one can ask whether it is Lipschitz continuous on the whole space.  
It is obvious that this is the case when $\Theta$ is finite or a line in the plane, and it is also obvious that it is not the case
if $\Theta$ has inner points, for example, if it is a strip of positive thickness. It is much less obvious what happens when $\Theta$ is a Lipschitz or H\"older manifold, or when it is a Sierpi\'nski carpet, a Sierpi\'nski triangle, or some other complicated set. It turns out that permeability is the proper framework for answering these more demanding questions, see \cite{leoste21}. For example, we will see in Example~\ref{exs:basic}~(9) that the Sierpi\'nski carpet is impermeable. This implies that the Devil's staircase function (also known as Cantor function), a nonconstant continuous function $f:\R\to\R$ that is constant in each complementary component of the middle third Cantor set, can be generalized to the Sierpi\'nski carpet in the sense that there is a non-constant continuous function $f:\R^2 \to \R$ that is constant on each complementary component of the Sierpinski carpet. Because the Sierpi\'nski triangle is permeable according to Example~\ref{exs:p2}, such a function does {\em not} exist for the Sierpi\'nski triangle.
 
The first application of finitely permeable sets --- without being called ``finitely permeable'' there --- was to prove Lipschitz continuity of certain functions, and can be found in the context of numerical methods for stochastic differential equations (SDEs) with discontinuous coefficients, see for instance \cite{sz2016b,MGRY2024,szoel2021,szoel2021b}. In each of these papers, classes of SDEs with discontinuous drift coefficient are studied, where the set $\Theta$  of discontinuity is a submanifold. The idea is to apply a certain transformation to such an SDE, in order to transform it to another SDE with a Lipschitz continuous drift coefficient. To this transformed SDE classical results for existence, uniqueness, and convergence of numerical methods assuming Lipschitz continuous drift coefficient of the SDE can be applied. The crucial ingredient that ensures the existence of a suitable transformation is provided in \cite[Lemma 3.6]{sz2016b}. Indeed, this result allows to conclude Lipschitz continuity of the drift coefficient on the whole of $\R^d$, from Lipschitz continuity (w.r.t.~the {\em intrinsic metric}) on $\R^d\setminus\Theta$, provided that $\Theta$ is finitely permeable. This ``transformation method'' of \cite{sz2016b} for studying SDEs with discontinuous drift coefficient stirred some interest in the community,  as it made possible a number of results on tractability of a range of numerical methods \cite{difonzo2023stochastic,MGRY2024,MGY2023,szoel2021,szoel2021b}.

\subsection*{Contributions of the present paper} 
The main contribution of this article is to develop a theory of permeable subsets of $\R^d$, 
relating permeability to topological, dimension theoretic, and measure theoretic properties
of sets. We present a variety of conditions for checking null permeability, permeability, or impermeability of sets. Our theory is applicable to a variety of natural and well-known examples, for instance, permeability can be characterized for large classes of self-similar sets, fractal curves, and related subsets in $\R^d$. 

In Section~\ref{sec:prel} we provide fundamental definitions including null permeability, finite permeability, and permeability, the central concepts of this paper. We give a list of simple examples that illustrate these concepts, and we establish basic results on permeability that will be used later on.  Moreover, we show that, under mild conditions, our notions of permeability do not depend on the norm we use in order to determine the length of a path in $\R^d$ (see \Cref{thm:pperm}). \Cref{rem:non-convex-norm} shows that the Euclidean norm furnishes the strongest notion of permeability. This result provides the justification for restricting our attention to the Euclidean norm in the sequel. Section~\ref{ssec:perm-measure} deals with permeability and impermeability criteria for sets having zero (\Cref{sec:measure-zero}) and positive (\Cref{sec:lebPos}) Lebesgue measure, respectively. We present several examples that illustrate and sharpen the results of this section. In particular, we discuss permeable and impermeable versions of so-called {\em Osgood curves}, {\em i.e.,} injective paths with positive Lebesgue measure (Examples~\ref{ex:osgood} -- \ref{ex:curve}). In \Cref{sec:dims} we investigate the relation between permeability and various notions of dimension. It turns out that for most notions of dimension a subset of $\R^d$ with dimension less than $d-1$ is null permeable. Special emphasis is put on the Hausdorff and Nagata dimension. The corresponding result for the Nagata dimension ({\it cf}.~\Cref{th:Nagata}) turns out to be useful later to establish criteria for null permeability of self-similar sets. \Cref{sec:synopsis} provides an overview of our dimension results. \Cref{sec:selfsimilarsets} is devoted to permeability properties of self-similar sets in $\R^d$ under certain finiteness properties. Here the classical {\em finite type condition} plays a crucial role. After proving that this condition is equivalent to a seemingly stronger condition (see \Cref{thm:finite-type-finite-types}) we provide quite general permeability criteria for self-similar sets. In particular, \Cref{thm:ftn-null-permeable2d} shows that self-similar sets in $\R^2$ are permeable under certain finiteness conditions, while \Cref{thm:ftn-null-permeable3d} shows that such sets are even null permeable in $\R^d$ for $d\ge 3$. We end \Cref{sec:selfsimilarsets} with a discussion of permeability properties of Bedford-McMullen carpets. This discussion exhibits impermeable subsets of $\R^d$ that are compact, topologically $0$-dimensional, and have Lebesgue measure~$0$.

\section{Preliminaries}\label{sec:prel} 
In this section we define {\it permeability}, the basic concept of our paper, and some of its variants. We provide simple examples to illustrate permeability and establish auxiliary results that will be needed later on. The section ends with a theorem that shows to what extent our notions of permeability depend on the norm we use to define a metric on $\R^d$. 

\subsection{Basic definitions}\label{sec:basic}
We recall general definitions and terminology that will be used throughout the paper. Assume that the Euclidean space $\R^d$ is equipped with some norm\footnote{A metric space would be sufficient to define permeability. However, to avoid pathologies, we restrict ourselves to this setting.} $\vertii\cdot$. For $Y\subseteq \R^d$ we denote the closure of $Y$ by $\overline{Y}$ and the interior of $Y$ by $Y^\circ$. If $x\in \R^d$, we set $d(x,Y):=\inf\{\vertii{y-x}\colon y\in Y\}$. Let $\varepsilon>0$. We use the notation
\begin{equation}\label{eq:Aeps}
\begin{aligned}
[Y]_\varepsilon &:= \{ x \in \R^d \colon \vertii{y-x} \le \varepsilon \text{ for some }y\in\overline{Y} \},\\
(Y)_\varepsilon &:= \{ x \in \R^d \colon \vertii{y-x} < \varepsilon \text{ for some }y\in Y \}
\end{aligned}
\end{equation}
for the {\em closed and open $\varepsilon$-neighborhood} of $Y$, respectively. In particular, $B_{\varepsilon}(x):=(\{x\})_\varepsilon$ is the open ball with center $x\in \R^d$ and radius $\varepsilon$.  For $X,Y \subset \R^d$, we set $d(X,Y):=\inf\{\vertii{y-x}: x\in X, y\in Y\}$. A {\em continuum} is a compact connected Hausdorff space and a {\em Peano continuum} is a locally connected metric continuum. A continuum is called \emph{nondegenerate} if it contains at least two  elements. A {\it Cantor set} is a nonempty compact totally disconnected perfect subset of $\R$. A {\em region} is a nonempty, open, and connected subset of $\R^d$. A {\em path} is a continuous map $\gamma\colon [a,b]\to \R^d$, where $a,b\in \R$ with $a<b$. We say that $\gamma$ connects 
$x\in \R^d$ and $y\in \R^d$ if $x=\gamma(a)$ and $y=\gamma(b)$.  An injective path is called an {\em arc}. A continuous injection $\gamma\colon\mathbb{S}^1\to \R^d$, where $\mathbb{S}^1$ is the unit circle, is called a {\em loop}. Often, when there is no risk of confusion, we abuse notation and identify a path, an arc, or a loop with its image, respectively. For $x,y\in \R^d$ we denote the {\em directed line segment} from $x$ to $y$ by $\overline{xy}$. If $x_0,\ldots,x_n\in \R^d$ are given, the path consisting of the concatenation of the line segments $\overline{x_{k-1}x_{k}}$ for $k\in \{1,\ldots, n\}$ is called
a {\em polygonal chain} and is denoted by $\overline{x_0\cdots x_n}$. We regard the directed line segments of a polygonal chain as its elements. For two directed line segments $s_1$ and $s_2$ we denote by $\sphericalangle(s_1,s_2) \in [0,\pi]$ the angle between the direction vectors of $s_1$ and $s_2$. 

\begin{definition}[Length of a path]\label{def:length}
If $\gamma\colon[a,b] \to \R^d$ is a path, then its {\em length} $\ell(\gamma)$ is defined as
\[
\ell(\gamma)
:=\sup\bigg\{\sum_{k=1}^n \vertii{\gamma(t_{k})-\gamma(t_{k-1})}\colon n\in\N,\,a =t_0<\cdots<t_n=b \bigg\}.
\]
\end{definition}

Let $\gamma\colon[a,b] \to \R^d$ be a path and $a=t_0<\cdots<t_n=b$. The polygonal chain consisting of the directed line segments  $\overline{\gamma(t_{k-1})\gamma(t_{k})}$, $k\in\{1,\ldots, n\}$, is called a {\em polygonal approximation of $\gamma$.} 

\subsection{Permeability}
We are now ready to provide the central definition of the present paper.
 
\begin{definition}[Permeability]\label{def:perm}
Let $\Theta\subset \R^d$.  
\begin{enumerate}
\item We call $\Theta$ {\em null permeable}, if for all $x,y\in \R^d$ and all $\delta>0$, $x$ and $y$ can be connected by a path $\gamma$ in $\R^d$  with $\ell(\gamma)\le\vertii{y-x}+\delta$ and 
$\gamma\cap \Theta\subset\{x,y\}$.
\item We call $\Theta$ {\em finitely permeable}, if for all $x,y\in \R^d$ and all $\delta>0$, $x$ and $y$ can be connected by a path $\gamma$ in $\R^d$  with $\ell(\gamma)\le\vertii{y-x}+\delta$ and 
$\gamma\cap \Theta$ finite. 
\item We call $\Theta$ {\em permeable}, if for all $x,y\in \R^d$ and all $\delta>0$, $x$ and $y$ can be connected by a path $\gamma$ in $\R^d$  with $\ell(\gamma)\le\vertii{y-x}+\delta$ and $\overline{\gamma\cap \Theta}$ countable.  
\item We call $\Theta$ {\em impermeable}, if $\Theta$ is not permeable.
\end{enumerate}
\end{definition}

Note that  every null permeable set is finitely permeable, and every finitely permeable set is  permeable. Note also that these notions of permeability depend on the norm $\vertii\cdot$ we are using (see Section~\ref{sec:norms} for more on this dependence). Null permeability is equivalent to a purely geometric version of {\em metric removability} that has been studied in \cite{rajala2019}. However, since we saw in the introduction that the term ``metric removability'' is also used in the context of conformal mappings, we will not use it in the sequel. 

\begin{examples} 
\label{exs:basic}
We start with a number of elementary examples to illustrate the permeability concepts of \Cref{def:perm}. In these examples, we equip $\R^d$ with the Euclidean metric.
\begin{enumerate}
\item For $d>1$ every finite subset of $\R^d$ is null permeable.
\item Every finite subset of $\R$ is finitely permeable.
\item Every closed and countably infinite ({\it i.e.}, scattered) subset of $\R$ is permeable (but not finitely permeable).
\item The set $\Q\subset \R$ is impermeable.
\item The sets $\{(x,0)\in \R^2\colon x\in \Q\}$ and $\{(\cos x,\sin x)\in \R^2\colon x\in \Q\cap [0,2\pi]\}$ are null permeable, their 
closures in $\R^2$ are finitely permeable.
\item The {\em Warsaw sine curve} $\{(x,\sin(x^{-1}))\colon x\in (0,\infty)\}\cup (\{0\}\times [-1,1])\subset\R^2$ is permeable, but not finitely permeable.
\item\label{ex:inner-point} If 
$\Theta\subset\R^d$ has an interior point, then $\Theta$ is impermeable (\cite[Proposition 12]{leoste21}).
\item\label{ex:extruded} If $C\subset [0,1]$ denotes the {\em middle third Cantor set} (see \cite[Chapter~2]{Edgar:04} for its first appearances in the literature), then $C\times [0,1]^{d-1}\subset\R^d$ is impermeable.
\item  \label{ex:qrimp} $([0,1]\setminus \Q)^2$ is impermeable \cite[Proposition 26]{leoste21}.
\item \label{ex:carpet1} Supersets of impermeable sets are impermeable. Thus as a consequence of item \eqref{ex:extruded}, the classical {\it Sierpi\'nski carpet} (which goes back to~\cite{Sierpinski:1916}) is impermeable.
\item  \label{ex:subset-perm} Subsets of null permeable sets are null permeable, and analogous statements hold for finite permeability and permeability. 
\item A topological submanifold of $\R^d$ of dimension smaller than $d$, which has bi-Lipschitz charts and is closed as a subset of $\R^d$, is permeable ({\em cf.}~\cite[Theorem 31]{leoste21}). In particular, if $D\subset \R^{d-1}$ is closed and convex, and $\vartheta\colon D\to \R$ is a Lipschitz function,
then its graph $\{(x,\vartheta(x))\colon x\in D\}\subset \R^d$ is permeable.
\item In \cite[Theorem 3]{BuLeSt22} an example of a H\"older-continuous function $\vartheta\colon [0,1]\to\R$ with impermeable graph (considered as a subset of $\R^2$) is constructed.
\end{enumerate}
\end{examples}

\begin{remark}[Relation to percolation]\label{rem:percolation}
Let $N> 1$ and $p\in (0,1)$. Partition the $d$-dimensional unit cube $[0,1]^d$ into $N^d$ subcubes  
of side length $\frac 1 N$. Each of these subcubes
is independently retained with probability $p$ or deleted with probability $1-p$.
Repeat this process for each of the previously retained subcubes, {\em ad infinitum}, to obtain a random set $\mathcal C$, which is called 
{\em Mandelbrot fractal percolation set}.
The investigation of permeability properties of probabilistic constructions like Mandelbrot percolation offers interesting facts and questions.
\begin{enumerate}
\item If the inequality $p\leq\frac{1}{N}$ holds, then for any $\delta>0$ any two points $x,y\in[0,1]^d$ can be connected by a path $\gamma$ of length $\|x-y\|+\delta$ that does not intersect $\mathcal{C}\setminus \{x,y\}$ ({\it i.e.}, $\mathcal{C}$ is null permeable). This is a consequence of $\mathcal{C}$ being totally disconnected (see \cite{FalconerGrimmett92}), having finite $(d-1)$-dimensional Hausdorff measure (see {\it e.g.}~\cite[Section 1.a]{Chayes88}, where this is shown for $d=2$), and \cite[Theorem 1.3]{rajala2019}.
\item  
There exists $p_c\in (0,1)$ such that for $p\in [p_c, 1)$ there are nondegenerate continua contained in $\mathcal C$ (see {\it e.g.}\ \cite{FalconerGrimmett92}), so for $d=2$, $\mathcal{C}$ 
is not null permeable. We do not know whether $\mathcal{C}$ can be permeable.
\item It is unknown what kind of behavior w.r.t.\ permeability can occur for $p\in \big(\tfrac 1 N, p_c\big)$.
\end{enumerate}
\end{remark}

\subsection{Auxiliary results on permeability}\label{sec:auxiliary-results}
The next two lemmata state that for a set $\Theta\subset \R^d$ with empty interior 
null permeability as well as permeability already follows if one only considers short paths connecting points in the complement of $\Theta$. 

\begin{lemma}\label{lemma:outside-theta} Let $\Theta\subset \R^d$ have empty interior. 
Assume that for all $x,y\in \R^d\setminus \Theta$ and all $\delta>0$ there exists a path $\gamma$ in $\R^d$
 connecting $x$ and $y$ with $\ell(\gamma)< \vertii{y-x}+\delta$ and $\gamma\cap \Theta=\emptyset$. Then $\Theta$ is null  permeable.
\end{lemma}

\begin{proof}
Let $x,y\in \R^d$ and let $\delta>0$. 
Since $\Theta$ has no interior points, there exist sequences $(x_n)_{n\ge 1}$ and $(y_n)_{n\ge 1}$
of points in $\R^d\setminus \Theta$ such that
$\vertii{x_n-x}<\frac{\delta}{24\cdot 2^{n}}$ and $\vertii{y_n-y}<\frac{\delta}{24\cdot 2^{n}}$ for all $n$. By assumption,  
there exist paths $\gamma_n$ from $x_{n+1}$ to $x_n$ with 
\begin{align*}
\ell(\gamma_n)<\vertii{x_{n+1}-x_n}+\frac{\delta}{6\cdot 2^n}\leq \vertii{x_{n+1}-x}+\vertii{x_n-x}+\frac{\delta}{6\cdot 2^n}<\frac{\delta}{4\cdot 2^n}
\end{align*} and $\gamma_n\cap \Theta=\emptyset$, as well as paths
 $\eta_n$ from $y_n$ to  $y_{n+1}$ with
$\ell(\eta_n)<\vertii{y_{n+1}-y_n}+\frac{\delta}{6\cdot 2^n}<\frac{\delta}{4\cdot 2^n}$ and $\eta_n\cap \Theta=\emptyset$.
Moreover, there exists a path $\gamma_0$ from $x_1$ to $y_1$ with 
$\ell(\gamma_0)<\vertii{y_{1}-x_1}+\frac{\delta}{6}$ and $\gamma_0\cap \Theta=\emptyset$.
Now let  $\gamma$ be the closure of the concatenation of the paths
$\ldots,\gamma_2,\gamma_1,\gamma_0,\eta_1,\eta_2,\ldots$ By taking the closure of this concatenation, the only points that are added are $x$ and $y$, thus $\gamma$ is a path. For its length we gain
\begin{align*}
\ell(\gamma)&=\ell(\gamma_0)+\sum_{n=1}^\infty (\ell(\gamma_i)+\ell(\eta_i))\leq \vertii{y_1-x_1}+\frac{\delta}{6}+\sum_{n=1}^\infty\Big(\frac{\delta}{4\cdot 2^n}+\frac{\delta}{4\cdot 2^n}\Big)\\
&\leq \vertii{x-x_1}+\vertii{y-x}+\vertii{y_1-y}+\frac{2\delta}{3}\leq \frac{\delta}{48}+\vertii{y-x}+\frac{\delta}{48}+\frac{2\delta}{3}\leq\vertii{y-x}+\delta. \qedhere
\end{align*}
\end{proof}

\begin{lemma}\label{lemma:outside-theta-closed} 
Let $\Theta\subset \R^d$ have empty interior. 
Assume that  for all $x,y\in \R^d\setminus \Theta$ and all $\delta>0$ there exists a path $\gamma$ 
 connecting $x$ and $y$ with $\ell(\gamma)< \vertii{y-x}+\delta$ and $\overline{\gamma\cap \Theta} $ countable.
Then $\Theta$ is  permeable.
\end{lemma}

\begin{proof}
Let $x,y\in \R^d$ and let $\delta>0$.
We choose sequences $(x_n)_{n\ge 1}$ and $(y_n)_{n\ge 1}$
as in the proof of \Cref{lemma:outside-theta}. Again  
there are paths $\gamma_n$ from $x_{n+1}$ to $x_n$ with 
$\ell(\gamma_n)<\vertii{x_{n+1}-x_n}+\frac{\delta}{6\cdot 2^n}<\frac{\delta}{4\cdot 2^n}$ and $\overline{\gamma_n\cap \Theta}$ countable, and paths 
 $\eta_n$ from $y_n$ to  $y_{n+1}$ with
$\ell(\eta_n)<\vertii{y_{n+1}-y_n}+\frac{\delta}{6\cdot 2^n}<\frac{\delta}{4\cdot 2^n}$ and $\overline{\eta_n\cap \Theta}$ countable.
Moreover, there exists a path $\gamma_0$ from $x_1$ to $y_1$ with 
$\ell(\gamma_0)<\vertii{y_1-x_1}+\frac{\delta}{6}$ and $\overline{\gamma_0\cap \Theta}$  countable.
Now let  $\gamma$ be the closure of the concatenation of the paths
$\ldots,\gamma_2,\gamma_1,\gamma_0,\eta_1,\eta_2,\ldots$,
and, as in the proof of \Cref{lemma:outside-theta}, we get that the closure adds only $x$ and $y$ to the concatenation. Hence, $\gamma$ is a path with
\(
\ell(\gamma)<\vertii{y-x}+\delta.
\) 
We need to show that, by taking the closure of the concatenation, countability of $\overline{\gamma\cap\Theta}$ is maintained. However, this follows because
\begin{equation}\label{eq:glue-paths-closed}
\overline{\gamma\cap\Theta} =  \overline{(\gamma_0\cap \Theta}) \cup \bigcup_{n=1}^\infty \overline{(\gamma_n\cup \eta_n)\cap \Theta} \cup(\{x,y\} \cap \Theta),
\end{equation}
and the right hand side is a union of countably many countable sets. 
\end{proof}

\subsection{Invariance under change of norms}\label{sec:norms}
In this section, for $1\le p \le \infty$ we will denote the $p$-norm by $\vertii{\cdot}_p$. If we concatenate  axis-parallel line segments to connect points and apply \Cref{lemma:outside-theta} it is easy to see that the set $(\R\setminus\Q)^2$ is permeable in $(\R^2,\vertii{\cdot}_1)$. However,  $(\R\setminus\Q)^2$ is impermeable in $(\R^2,\vertii{\cdot}_2)$ by Example~\ref{exs:basic}~\eqref{ex:qrimp}. This raises the question to what extent permeability depends on the norm used on $\R^d$. It turns out that norms admitting many geodesics (like $\vertii{\cdot}_\infty$ and $\vertii{\cdot}_1$) correspond to weaker forms of permeability while norms having a strictly convex closed unit ball all lead to the same class of permeable sets. Analogous assertions hold for finite permeability and null permeability. 
Indeed, we get the following results.

\begin{theorem}\label{thm:pperm}
Let $\vertii\cdot$ be any norm on $\R^d$ such that the boundary of its unit ball is strictly convex.
Then $\Theta\subset \R^d$ is permeable in $(\R^d,\vertii\cdot)$ if and only if it is permeable in $(\R^d,\vertii\cdot_2)$. The same equivalence is true for finite permeability and null permeability.
\end{theorem}

\begin{theorem}\label{rem:non-convex-norm}
If $\Theta\subset \R^d$ is permeable in $(\R^d,\vertii\cdot_2)$ then it is permeable in $(\R^d,\vertii\cdot)$ for any norm~$\vertii{\cdot}$. The same statement is true for finite permeability and null permeability.
\end{theorem}

The following lemma forms the main part of the proofs of \Cref{thm:pperm} and \Cref{rem:non-convex-norm}. In its statement recall the definitions from \Cref{sec:basic} and the fact that the line segments of a polygonal approximation of a path are regarded as its elements.

\begin{lemma}\label{lem:pperm}
We denote by $\hat \ell$ the length of a path in $\R^d$ w.r.t.\ some norm $\vertii\cdot$. Fix $z\in \R^d$. For every $n\in\N$ let $\gamma_n$ be a path in $\R^d$ connecting $0$ and $z$. Consider the following assertions. 
\begin{enumerate}
\item[(1)] For each $\varepsilon>0$ there is $N\in\N$ such that for all $n\ge N$ and all polygonal approximations $Z$ of $\gamma_n$ we have
\begin{equation*}
\sum_{s\in Z \colon \sphericalangle(s,\overline{0z}) > \varepsilon}  \hat\ell(s) <  \varepsilon.
\end{equation*}
\item[(2)] $\lim_{n\to \infty}\hat \ell(\gamma_n) = \vertii z $.
\end{enumerate}
Then
(1) $\Rightarrow$ (2) holds for any norm and (2) $\Rightarrow$ (1) holds for each norm with strictly convex closed unit ball. 
\end{lemma}

\begin{proof}
Let $\vertii{\cdot}$ be any norm on $\R^d$. W.l.o.g.\ we may assume that $\vertii z =1$. Let us denote the open unit ball of  $\vertii\cdot$ by $U$. Using the norm axioms one easily checks that $\overline{U}$ is convex. Thus there exists an affine hyperplane $E_z$ with  $z \in \partial U \cap E_z$ and $E_z\cap U=\emptyset$ ($E_z$ is a {\em supporting hyperplane} of $\overline{U}$ in $z$). Denote  by $\pi_{z}$ the projection  parallel to $E_{z}$ onto  $\R z$, and by $v$ the unit normal vector of $E_z$ satisfying $\langle x,v\rangle>0$ for each $x\in E_z$. By definition, this projection does not increase the norm.

We first prove that (1) $\Rightarrow$ (2) for arbitrary norms. Let $\alpha:=\tfrac{\pi}{2}-\sphericalangle(\overline{0v},\overline{0z})$ and consider the map 
\[
f:\,\left\{ x \in E_z \colon \sphericalangle(\overline{0x},\overline{0z}) \le \frac{\alpha}{2}\right\}\setminus\{z\}\to \R
\,;\quad x\mapsto \frac{\vertii x - 1}{\sphericalangle(\overline{0x},\overline{0z})}.
\] 
Because $\overline{U}$ is nondegenerate and convex,  the singularity of $f$ at $z$ is removable, and we can extend $f$ continuously to the whole filled ellipse
$\left\{ x \in E_z \colon \sphericalangle(\overline{0x},\overline{0z}) \le \frac{\alpha}{2}\right\}$. Denoting this extension again by $f$,  we 
define a constant  
\begin{equation}\label{eq:ccc}
 c := \sup\left\{f(x)\colon x \in E_z \text{ with }\sphericalangle(\overline{0x},\overline{0z}) \le \frac{\alpha}{2}\right\},
\end{equation}
which is finite.
With this definition of $c$, we have for every $x \in E_z$ with $\sphericalangle(\overline{0x},\overline{0z}) \le \varepsilon \le \frac{\alpha}{2}$ that
\begin{equation}\label{eq:pperm1_a_pre}
 \vertii x \le 1 + c\varepsilon.
\end{equation}
 Then \eqref{eq:pperm1_a_pre} implies that for each $x\in \R^d$ with $\sphericalangle(\overline{0x},\overline{0z}) \le \varepsilon\le \frac{\alpha}{2}$ we have
\begin{equation}\label{eq:pperm1_a}
\hat\ell(\overline{0x}) = \vertii{x} \le (1+c\varepsilon) \vertii{\pi_z(x)} =  (1+c\varepsilon)\hat\ell(\pi_z(\overline{0x})).
\end{equation}
For each $\varepsilon>0$ small enough there is  $n\ge N$ such that for each $Z$ that approximates $\gamma_n$ we have 
\begin{equation}\label{eq:lhatlx1}
\hat\ell(Z)  \le \varepsilon + \sum_{s\in Z \colon \sphericalangle(s,\overline{0z}) \le \varepsilon}  \hat\ell(s)
\le  \varepsilon + (1+c\varepsilon)\sum_{s\in Z \colon \sphericalangle(s,\overline{0z}) \le \varepsilon}  \hat\ell(\pi_{z}(s)),
\end{equation}
where we used \eqref{eq:pperm1_a} in the second inequality.
Because by assumption we have
\[
\sum_{s\in Z \colon \sphericalangle(s,\overline{0z}) \le \varepsilon}  \hat\ell(\pi_{z}(s)) \le \sum_{s\in Z \colon \sphericalangle(s,\overline{0z}) \le \pi/2}  \hat\ell(\pi_{z}(s)) =
 1 +  \sum_{s\in Z \colon \sphericalangle(s,\overline{0z}) > \pi/2}  \hat\ell(\pi_{z}(s)) \le 1 + \varepsilon,
\]
we gain from \eqref{eq:lhatlx1} that
\[
\hat\ell(Z) \le  \varepsilon + (1+c\varepsilon)(1+\varepsilon)
< 1+ (2c+2) \varepsilon.
\]
Since $Z$ was arbitrary, this yields $\hat\ell(\gamma_n)\le 1+(2c+2)\varepsilon $ for $n\ge N$ and, hence, $\lim_{n\to \infty}\hat\ell(\gamma_n) = 1$.

To establish the implication (2) $\Rightarrow$ (1), we may assume that the closed unit ball $\overline{U}$ of $\vertii{\cdot}$ is strictly convex. We define a function $\delta \colon \mathbb R_+\to \mathbb R$ as
\[
 \delta(\varepsilon) := \inf\{\vertii x - 1\colon x \in E_z \text{ with } \sphericalangle(\overline{0x},\overline{0z}) > \varepsilon\}.
\]
Because  $\overline{U}$ is strictly convex, we have $ \delta(\varepsilon) > 0$ for all $\varepsilon \in \R_+$. Suppose that there is $\varepsilon>0$ such that for arbitrarily large $n$ there is an approximation $Z$ of $\gamma_n$  satisfying
\[
\sum_{s\in Z \colon \sphericalangle(s,\overline{0z}) > \varepsilon}  \hat\ell(s) =\varepsilon'  \ge \varepsilon.
\]
The projections $\pi_{z}(s)$ of the line segments $s\in Z$ with $\sphericalangle(s,\overline{0z}) > \varepsilon$ cover at most a $\vertii\cdot$-length $\frac{\varepsilon'}{1+\delta(\varepsilon)}$ of $\overline{0z}$. Thus the projections of the line segments $s\in Z$ with $\sphericalangle(\overline{0x},\overline{0z}) \le  \varepsilon$ have to cover a $\vertii\cdot$-length at least $1 -\frac{\varepsilon'}{1+\delta(\varepsilon)}$ of $\overline{0z}$. Note that $\ell(s)\ge \ell(\pi_z(s))$. We therefore get
\begin{align*}
\hat\ell(\gamma_n)\ge \hat\ell(Z)
&=
\sum_{s\in Z \colon \sphericalangle(s,\overline{0z}) > \varepsilon}  \hat\ell(s) + \sum_{s\in Z \colon \sphericalangle(s,\overline{0z}) \le \varepsilon}  \hat\ell(s)\\
&\ge
\varepsilon' + \sum_{s\in Z \colon \sphericalangle(s,\overline{0z}) \le \varepsilon}  \hat\ell(\pi_{z}(s))
\ge
\varepsilon' +1 -\frac{\varepsilon'}
{1+\delta(\varepsilon)} \ge 1 +\frac{\varepsilon\delta(\varepsilon)}
{1+\delta(\varepsilon)} .
\end{align*}
and, hence, $\hat\ell(\gamma_n)$ cannot converge to $1$.
\end{proof}

\Cref{lem:pperm} is sharp in the following sense: If $\vertii{\cdot}$ is a norm whose closed unit ball is not strictly convex, then there exist points $z_1,z_2\in \R^d$ with $\vertii{z_1}=\vertii{z_2}=1$, $z_1\ne z_2$, and $\vertii{\tfrac{1}{2}(z_1+z_2)}=1$. For $z=\tfrac{1}{2}(z_1+z_2)$, it is not hard to find $\varepsilon>0$ and  a polygonal chain $Z$ with $\sum_{s\in Z}\hat\ell(s)=\vertii{z}$ and $\sum_{s\in Z \colon \sphericalangle(s,\overline{0z}) > \varepsilon}  \hat\ell(s) \ge  \varepsilon$.

We immediately gain the following corollary, which is of interest in its own right.

\begin{corollary}\label{cor:pperm}
Let $\vertii\cdot$ be any norm on $\R^d$ whose unit ball is strictly convex. Further let $x,y\in \R^d$ and let $\gamma_n$ be a sequence of paths from $x$ to $y$.  Let $\hat \ell$ be the length of a path w.r.t.\ $\vertii\cdot$ and let $\ell$ be the length of a path w.r.t.\ $\vertii\cdot_2$. Then $\lim_{n\to \infty }\hat \ell(\gamma_n)=\vertii{y-x}$ if and only if $\lim_{n\to \infty }\ell(\gamma_n)=\vertii{y-x}_2$.
\end{corollary}

\begin{proof}
By the equivalence of all norms on $\R^d$, the condition in Lemma~\ref{lem:pperm} does not depend on the norm. This implies the result.
\end{proof}

\begin{proof}[Proof of Theorem~\ref{thm:pperm}]
Let $\hat \ell$  and $\ell$ be as in \Cref{cor:pperm}. Now $\Theta$ is permeable in $(\R^d,\vertii\cdot)$ if and only if for each $x,y\in \R^d$ there is a sequence $(\gamma_n)_{n\in \N}$ of paths with $\overline{\gamma_n\cap \Theta}$ countable and $\lim_{n\to \infty}\hat \ell(\gamma_n) =\vertii{x-y}$. By \Cref{cor:pperm} this is equivalent to the fact that for each $x,y\in \R^d$ there is a sequence $(\gamma_n)_{n\in \N}$ of paths with $\overline{\gamma_n\cap \Theta}$ countable and $\lim_{n\to\infty}\ell(\gamma_n) = \vertii{x-y}_2$. But this is equivalent to permeability of $\Theta$ in $(\R^d,\vertii\cdot_2)$.

The same argument goes through for finite permeability and null permeability.
\end{proof}

\begin{proof}[Proof of Theorem~\ref{rem:non-convex-norm}]
This runs along the same lines as the proof of Theorem~\ref{thm:pperm} because the implication (1) $\Rightarrow$ (2) in Lemma~\ref{lem:pperm} is true for any norm.
\end{proof}

Because of Theorems~\ref{thm:pperm} and~\ref{rem:non-convex-norm}, from now onwards, we will exclusively work in the space $(\R^d,\vertii{\cdot}_2)$. For this reason, in the sequel we will just write $\vertii\cdot$ for the $2$-norm.

\section{Permeability and Lebesgue measure}\label{ssec:perm-measure}
In this section we investigate how properties of the Lebesgue measure of a set relate to our notions of permeability. As we will see in Section~\ref{sec:measure-zero}, the fact that a set $\Theta \subset \R^d$ has Lebesgue measure zero allows for easier ways to infer that $\Theta$ is permeable. On the other hand, Section~\ref{sec:lebPos} provides a criterion for impermeability under the assumption that $\Theta$ has positive Lebesgue measure. Nevertheless, we will see from various examples that the Lebesgue measure cannot be used to characterize permeability and its variants.

\subsection{Sets of measure zero}\label{sec:measure-zero}
In this section we will deal with permeability results for subsets of $\R^d$ with zero Lebesgue measure. 
We start with an auxiliary result.

\begin{lemma}[{\cite[Lemma 30]{leoste21}}]\label{lemma:fubini-argument}
Let $\Theta\subset \R^d$ be a Lebesgue nullset. Then for all $x,y\in \R^d$ and all $\varepsilon>0$ there exists a polygonal chain $\gamma\colon[a,b]\to\R^d$ connecting $x$ and $y$, parametrized by arc length\footnote{The assertion that $\gamma$ is parametrized by arc length is not present in \cite[Lemma 30]{leoste21}, but is evident from its proof.}, such that $\ell(\gamma)<\vertii{y-x} +\varepsilon$ and $\{t\in[a,b]:\gamma(t)\in\Theta\}$ is a Lebesgue nullset in $[a,b]$.
\end{lemma}

The next definition is a special case of \cite[2.15]{heinonen1998}.

\begin{definition}[Quasi convexity]\label{defi:C-quasi-convex}
A subset $A\subset \R^d$ is called {\em $C$-quasi convex} if there exists $C>0$ such that for all $x,y\in A$ and $\varepsilon>0$ 
there exists a path $\gamma\subset A$ connecting $x$ and $y$ with $\ell(\gamma) \le C \vertii{y-x}+\varepsilon$.  
\end{definition}

Our first result shows that $C$-quasi convexity of its complement implies null permeability of a Lebesgue nullset.

\begin{proposition}\label{prop:c-quasiconvex=>nullperm}
Let $\Theta\subset \R^d$ and $C\in [1,\infty)$. If $\Theta$ is a Lebesgue nullset and $\R^d\setminus \Theta$ is $C$-quasi convex  then $\Theta$ is  null permeable. 
\end{proposition}

\begin{proof}
Let $x,y\in\R^d\setminus\Theta$ with $x\ne y$ and $\delta>0$. Lemma \ref{lemma:fubini-argument} provides a path 
$\gamma\colon [0,\ell(\gamma)]\to \R^d$ from $x$ to $y$, parametrized by arc length, such that 
$A=\{t\in [0,\ell(\gamma)]\colon \gamma(t)\in \Theta\}$ is 
a Lebesgue nullset in $[0,\ell(\gamma)]$ and such that $\ell(\gamma)<\vertii{x-y}+\frac{\delta}{3}$. 
Thus there exists a countable collection of disjoint intervals $\big((a_n,b_n)\big)_{n\in \N}$ such that $A\subset \bigcup_{n=1}^\infty (a_n,b_n)$ and
$\sum_{k=1}^\infty (b_n-a_n)<\frac{\delta}{3C}$. By our assumption $\R^d\setminus \Theta$ is $C$-quasi convex. Thus for each $n\in\N$ there exists a path $\kappa_n\colon [a_n,b_n]\to\R^d\setminus \Theta$ connecting $\gamma(a_n)$ and $\gamma(b_n)$ with
\[
 \ell(\kappa_n)\le C\vertii{\gamma(b_n)-\gamma(a_n)} + \frac{\delta}{3\cdot 2^n}
 \le C\ell(\gamma|_{[a_n,b_n]}) + \frac{\delta}{3\cdot 2^n}.
\]
Then 
\begin{equation*}
\kappa(t):=\begin{cases}
\gamma(t)&\text{ if }t\notin \bigcup_{n=1}^\infty [a_n,b_n],\\
\kappa_n(t)&\text{ if }t\in  [a_n,b_n]
\end{cases}
\end{equation*}
defines a path $\kappa \colon[0,\ell(\gamma)]\to \R^d$ disjoint from $\Theta$ satisfying
\begin{align*}
\ell(\kappa)&\le \ell(\gamma)+\sum_{n=1}^\infty \ell(\kappa_n) \le\ell(\gamma)+ \sum_{n=1}^\infty \Big(C\vertii{\gamma(b_n)-\gamma(a_n)}+ \frac{\delta}{3\cdot 2^n}\Big)\\
&\le \ell(\gamma) + \frac{\delta}{3}+\sum_{n=1}^\infty C\ell(\gamma|_{[a_n,b_n]})
\le \vertii{y-x}+\frac{2\delta}{3}+\sum_{n=1}^\infty C (b_n-a_n)<\vertii{y-x}+\delta. 
\end{align*}
Now the assertion follows from \Cref{lemma:outside-theta}.
\end{proof}

\begin{corollary}\label{cor:product-of-dense-sets} Let $d\ge 2$ and $1\le k\le d-1$.
Let $A\subset \R^k$ be a Lebesgue nullset and let $B\subset \R^{d-k}$ be a set with dense complement in $\R^{d-1}$. Then $A\times B\subset \R^d$ is null permeable. 

\end{corollary}

\begin{proof}
We claim that the complement of $A\times B$ is $\sqrt{2}$-quasi convex. 
To show this, note that, as a Lebesgue nullset, $A$ has dense complement. 
Recall that $(A\times B)^c=(A^c\times B)\cup (A\times B^c) \cup (A^c\times B^c)$.
We first consider the case $(x_1,x_2)\in (A^c\times B)$ and $(y_1,y_2)\in (A\times B^c)$. 
Let $\varepsilon>0$. Since $A^c$ and $B^c$ are dense in $\R^k$ and $\R^{d-k}$, respectively, there exist $\xi_2\in B^c$ and  $\eta_1\in A^c$ with $\vertii{\xi_2-x_2}<\varepsilon$ and $\vertii{\eta_1-y_1}<\varepsilon$. Now the polygonal chain $\gamma:=\overline{(x_1,x_2)(x_1,\xi_2)(\eta_1,\xi_2)(\eta_1,y_2)(y_1,y_2)}$ is a path in $(A\times B)^c$ with 
\[
\ell(\gamma)<\varepsilon+\vertii{\eta_1-x_1}+\vertii{y_2-\xi_2}+\varepsilon<\vertii{y_1-x_1}+\vertii{y_2-x_2}+4 \varepsilon\le \sqrt{2} \vertii{y-x}+4\varepsilon. 
\]
Since the other cases can be shown in a similar way, the claim is proved.
 
Because $A$ is a Lebesgue nullset, the same is true for $A\times B$. Thus it follows from \Cref{prop:c-quasiconvex=>nullperm} that $A\times B$ is null permeable.
\end{proof}

The next corollary  is a variant of the above corollary.

\begin{corollary}\label{cor:projPerm}
Let $\Theta\subset\R^d$ be a closed Lebesgue nullset such that each projection to a coordinate $(d-1)$-plane is nowhere dense therein. Then $\Theta$ is null permeable.
\end{corollary}
\begin{proof}
Since all $(d-1)$-dimensional coordinate projections of $\Theta$ are nowhere dense, $\R^d\setminus\Theta$ is $C$-quasi convex for some $C>0$ by \cite[Theorem A]{hakobyan2008}. Thus the corollary follows from \Cref{prop:c-quasiconvex=>nullperm}.
\end{proof}

\begin{example}[The Menger sponge is null permeable]
The {\it Menger sponge} (see~\cite[Chapter~6]{Edgar:04} for its origins) is obtained by removing the set 
$$
\Big(\Big(\frac{1}{3},\frac{2}{3}\Big)\times \Big(\frac{1}{3},\frac{2}{3}\Big)\times [0,1]\Big)\cup \Big(\Big(\frac{1}{3},\frac{2}{3}\Big)\times [0,1]\times \Big(\frac{1}{3},\frac{2}{3}\Big)\Big)\cup \Big([0,1]\times \Big(\frac{1}{3},\frac{2}{3}\Big)\times \Big(\frac{1}{3},\frac{2}{3}\Big)\Big)
$$ 
from $[0,1]^3$ and repeating this operation on the remaining 20 cubes of side length $\frac{1}{3}$ iteratively. Indeed, the Menger sponge is the attractor of a self-similar iterated function system (see Section~\ref{sec:ifsdef} for a definition of these objects).
The projection of the Menger sponge to each coordinate plane is equal to the Sierpi\'nski carpet. Because the Sierpi\'nski carpet is nowhere dense, null permeability of the Menger sponge follows from \Cref{cor:projPerm}.
\end{example}

Note that this argument does not work for the Sierpi\'nski tetrahedron (if one uses the tetrahedron's base plane the projection thereon equals a filled triangle). For the treatment of this set other methods can be used, see Examples~\ref{ex:SierpinskiTetrahedron} and \ref{ex:SierpTetr2}.

We continue with a version of Proposition~\ref{prop:c-quasiconvex=>nullperm} for permeability. 

\begin{proposition}\label{prop:c-finperm=>perm}
Let $\Theta\subset \R^d$   and let $C\in [1,\infty)$. 
Assume that $\Theta$ is {\em $C$-quasi permeable}, {\em i.e.}, for all $x,y\in \R^d$ and all $\delta>0$, $x$ and $y$ can be connected by a path $\gamma$ in $\R^d$  with $\ell(\gamma)\le C \vertii{y-x}+\delta$ and $\overline{\gamma\cap \Theta}$ countable.
If $\Theta$ is contained in a closed Lebesgue nullset, then $\Theta$ is permeable. 
\end{proposition}

\begin{proof}
W.l.o.g., we may assume that $\Theta$ is a closed Lebesgue nullset. Assume that $x$ and $y$ are points in $\R^d$ and $\delta>0$. We proceed just as in the proof of Proposition~\ref{prop:c-quasiconvex=>nullperm}. Using Lemma~\ref{lemma:fubini-argument}, we obtain a path $\gamma\colon[0,\ell(\gamma)]\to\R^d$, connecting $x$ and $y$, parametrized by arc length and with $\ell(\gamma)<\vertii{x-y}+\frac{\delta}{3}$ and $A=\gamma^{-1}(\Theta)$. As $\Theta$ is closed, $A\subset [0,\ell(\gamma)]$ is compact. Thus for some $K>1$ there exist disjoint intervals $((a_n,b_n))_{1\le k\le K}$ such that $A\subseteq \bigcup_{n=1}^K(a_n,b_n)$ and $\sum_{n=1}^K(b_n-a_n)<\frac{\delta}{3C}$. By assumption, for each $n\in\{1,\ldots, K\}$ there exists a path $\kappa_n$ connecting $\gamma(a_n)$ and $\gamma(b_n)$ with $\overline{\kappa_n\cap \Theta}$ countable and $\ell(\kappa_n)<C\vertii{\gamma(b_n)-\gamma(a_n)}+\frac{\delta}{3K}$. Replacing $\gamma$ on the pieces $[a_n,b_n]$ by $\kappa_n$ for each $n\in\{1,\ldots, K\}$, we obtain a path $\kappa$ that connects $x$ and $y$ and has $\ell(\kappa)<\vertii{y-x}+\delta$. Since $\overline{\kappa([a_n,b_n])\cap\Theta}$ is countable for each $n\in\{1,\ldots, K\}$, $\overline{\kappa([0,1])\cap\Theta}$ is countable as well and permeability of $\Theta$ follows.
\end{proof}

Example~\ref{ex:closedneeded} shows that one cannot remove the closedness condition in the statement of Proposition~\ref{prop:c-finperm=>perm}. In Example~\ref{exs:basic}~\eqref{ex:carpet1} we showed that the Sierpi\'nski carpet, a Lebesgue nullset, is impermeable. In \Cref{sec:dims} we will see more examples of impermeable Lebesgue nullsets, and \Cref{sec:bedford-mcmullen} contains examples of particularly ``small'' impermeable sets.

\subsection{Sets of positive measure}\label{sec:lebPos}
At the end of the \Cref{sec:measure-zero} we mentioned that Lebesgue nullsets may well be impermeable. We start this section with examples that point in the opposite direction in providing null permeable sets with large Lebesgue measure.

\begin{example}[Null permeable sets with large Lebesgue measure]\label{rem:LebPerm}
Let $d\ge 2$. We provide closed null permeable sets $A_n \subset [0,1]^d$, whose measures tend to $1$ for $n\to\infty$ and a null permeable set $A\subset [0,1]^d$ with Lebesgue measure $1$ (see also \cite[Proposition~3.5]{rajala2019}). Let $(q_i)_{i \in \mathbb N}$ be a dense sequence in $\R^d$ and define $A = [0,1]^d \setminus \bigcup_{0\le i<j} \overline{q_iq_j}$ (recall that $\overline{xy}$ denotes the line segment from $x$ to $y$). In the same spirit, for each $n\in \N$ define, using the notation \eqref{eq:Aeps}, the closed set $A_n = [0,1]^d \setminus \bigcup_{0\le i<j} (\overline{q_iq_j})_{2^{-i-j-n}}$. By \Cref{lemma:outside-theta}, $A$ and $A_n$ are null permeable. Obviously, the Lebesgue measures of $A_n$ tend to $1$ for $n\to\infty$, and $A$ has Lebesgue measure~1. \end{example}

See also Example~\ref{ex:curve} for a permeable curve with positive Lebesgue measure.\smallskip

We continue with an impermeability result for sets with positive Lebesgue measure. Recall that $x\in \R^d$ is a {\em Lebesgue point} of a measurable set $A\subset \R^d$, if and only if 
\[
\lim_{r\to 0}\frac{\lambda(B_r(x)\setminus A)}{\lambda(B_r(x))}=0.
\]
Moreover, for a set $A\subset \R^d$ we denote  the $s$-dimensional Hausdorff measure by  $\mathcal H^s(A)$.

\begin{proposition}\label{prop:impermeableintersection}
Let $d,k \in \mathbb N$, let $A_1,\ldots,A_k\subset \R$ have positive Lebesgue measure, and let $f_1,\ldots,f_k$ be isometries of $\mathbb R^d$. If the set $A=\bigcap_{i=1}^k f_i(A_i\times \mathbb R^{d-1})$ has positive Lebesgue measure, then it is impermeable in $\R^d$.
\end{proposition}
\begin{proof} 
 The case $d = 1$ is trivial, so we assume $d \ge 2$. By considering subsets of $A_i$, $i\in\{1,\ldots, k\}$, we may assume that all the $A_i$ are compact. Let $x$ be a Lebesgue point of $A$. Applying a shift by $-x$ to all $f_i$ and noting that neither permeability, impermeability nor the measure of $A$ changes through translations, it suffices to consider $x=0$. Further, using translates $B_i$ of $A_i$, one can find isometries $\tilde{f}_i$ such that $\tilde{f}_i(0)=0$ and $f_i(A_i\times\R^{d-1})=\tilde{f}_i(B_i\times\R^{d-1})$ for all $i\in \{1,\dotsc,k\}$. Thus we may assume that $f_i(0)=0$.
 
 We want to find a  point $y$ close to $0$ so that $0$ and $y$ cannot be joined by a short path whose intersection with $A$ is countable. Let us first identify a suitable direction of the point $y$ from $0$. We do this by writing $v_i := f_i((1,0,\dots,0))$ for all $i$ and selecting a direction $v \in \mathbb S^1$ satisfying 
$
c:=\min\{|\langle v, v_1\rangle|, \dots, |\langle v,v_k\rangle|\}>0.
$
 Since each $A_i$ is compact, the complement of $f_i(A_i\times \mathbb R^{d-1})$ is the disjoint union  of countably many components of the form $f_i(U\times \mathbb  R^{d-1})$, where $U$ is an open interval. We define $w(f_i(U\times \mathbb  R^{d-1})) := \diam(U)$. For a given 
 $i\in \{1,\ldots,k\}$ and $\delta>0$ let $\{T^i_{\delta,j}\}_{j \in \mathbb N}$ be the collection of all complementary components of $f_i(A_i\times \mathbb R^{d-1})$ that intersect $B_\delta(0)$. Since $0$ is a Lebesgue  point of $A$, (the scalar) $0$ is also a Lebesgue  point of $A_i$, so we have, 
 setting
 \[
 a_\delta:=\max\bigg\{\sum_jw(T^i_{\delta,j}):i \in \{1,\dots,k\}\bigg\},
 \]
 that
 $\limsup_{\delta \to 0}\frac{a_\delta}{\delta} = 0$.
 Thus we may choose $\delta>0$ small enough so that  
\begin{equation}\label{c-delta-choice}
a_\delta< 3\delta q \quad\text{with}\quad q:=\frac{(1-\sqrt{1-c^2})}{6k} \in (0,1)
\end{equation} 
and we define $y_\delta :=  \delta v$.
Assume that $A$ is permeable. Then there exists a path $\gamma\colon [0,1]\to \R^d$ connecting $0$ and $y_\delta$ with $\gamma\cap A$ countable and
\begin{eqnarray}\label{eq:hausdorff-gamma1}
\delta \le \ell(\gamma)<\vertii{y_\delta} (1+q)=\delta (1+q).
\end{eqnarray} 
We may assume w.l.o.g.\ that $\gamma$ is an arc (see {\it e.g.}  \cite[Theorem 4.4.7]{AmbTil04}, or \cite[Proposition 3.4]{AlbOtt17}). Since \cite[Proposition 3.5]{AlbOtt17} implies that $ \mathcal H^1(\gamma) =\ell(\gamma)$ we gain from \eqref{eq:hausdorff-gamma1} that
 \begin{eqnarray}\label{eq:hausdorff-gamma}
\delta \le \mathcal H^1(\gamma)=\ell(\gamma)<\vertii{y_\delta} (1+q)=\delta (1+q).
\end{eqnarray}
 Because 
 \[
 \delta\le\mathcal H^1(\gamma)=\mathcal H^1(\gamma\setminus A)=\mathcal H^1(\bigcup_{i=1}^k(\gamma\setminus f_i(A_i\times \mathbb R^{d-1})))\le \sum_{i=1}^k \mathcal  H^1(\gamma \setminus f_i(A_i \times \mathbb R^{d-1})), 
 \]
 for some $i \in \{1,\dots,k\}$ we have
 \begin{eqnarray}\label{eqn:len-gamma-ge1k}
 \mathcal  H^1(\gamma \setminus f_i(A_i \times \mathbb R^{d-1})) \ge \frac{\delta}{k}.
 \end{eqnarray}
  Keep this $i$ fixed from now on.
 Let us abbreviate $c_i := |\langle v,v_i\rangle| \ge c >0$.
By changing $v$ to $-v$ if necessary, we may assume that $\langle v, v_i \rangle = c_i$. By choosing an absolutely continuous reparametrization of $\gamma$ (which is always possible since $\gamma$ is of finite length, see \cite[Theorem 4.4]{AlbOtt17})
 we gain
 \[
 \begin{split}
   \langle \dot{\gamma}(t) , v \rangle &= \langle \dot{\gamma}(t) , \langle v, v_i\rangle v_i \rangle + \langle \dot{\gamma}(t) , v - \langle v, v_i\rangle v_i \rangle
\\
& \le c_i\langle \dot{\gamma}(t) , v_i \rangle + \vertii{\dot{\gamma}(t)}\sqrt{1-c_i^2}
\\
&   \le \langle \dot{\gamma}(t) , v_i \rangle + \vertii{\dot{\gamma}(t)}\sqrt{1-c^2}
 \end{split}
 \]
for a.a.\ parameter values $t$, where we used $\vertii{v-\langle v,v_i\rangle v_i}=\sqrt{1-c_i^2}$. Hence, 
\begin{equation}\label{eq:calc-len-ydelta}
\begin{aligned}
  \vertii{y_\delta} & = \int_0^1 \langle \dot{\gamma}(t) , v \rangle\,dt
  = \int_0^1 \langle \dot{\gamma}(t) , v \rangle\chi_{\mathbb R^d \setminus f_i(A_i \times \mathbb R^{d-1})}(\gamma(t))\,dt + \int_0^1 \langle \dot{\gamma}(t) , v \rangle\chi_{f_i(A_i \times \mathbb R^{d-1})}(\gamma(t))\,dt\\
  & \le
  \int_0^1 \big(\langle \dot{\gamma}(t) , v_i \rangle + \vertii{\dot{\gamma}(t)}\sqrt{1-c^2} \big)\chi_{\mathbb R^d \setminus f_i(A_i \times \mathbb R^{d-1})}(\gamma(t))\,dt+ \int_0^1 \vertii{\dot{\gamma}(t)}\chi_{f_i(A_i \times \mathbb R^{d-1})}(\gamma(t))\,dt\\
  & \le \int_0^1 \langle \dot{\gamma}(t) , v_i \rangle\chi_{\mathbb R^d \setminus f_i(A_i \times \mathbb R^{d-1})}(\gamma(t))\,dt + \sqrt{1-c^2}\mathcal H^1(\gamma \setminus f_i(A_i \times \mathbb R^{d-1})) \\
  &\quad+ \mathcal H^1(\gamma \cap f_i(A_i \times \mathbb R^{d-1}))\\
  &\le \int_0^1 \langle \dot{\gamma}(t) , v_i \rangle\chi_{\mathbb R^d \setminus f_i(A_i \times \mathbb R^{d-1})}(\gamma(t)) \chi_{(-\delta,\delta)}(\langle \gamma(t) , v_i \rangle) \,dt+\int_0^1 \vertii{\dot{\gamma}(t)}\chi_{\R\setminus (-\delta,\delta)}(\langle \gamma(t) , v_i \rangle) \,dt\\
  &\quad + \sqrt{1-c^2}\mathcal H^1(\gamma \setminus f_i(A_i \times \mathbb R^{d-1})) + \mathcal H^1(\gamma \cap f_i(A_i \times \mathbb R^{d-1})).
\end{aligned}
\end{equation}
Now note that for $z\in \mathbb R^d$ we may write 
\[z=\langle z,v_i\rangle v_i +z-\langle z,v_i\rangle v_i
=\langle z,v_i\rangle f_i((1,0,\ldots,0))+f_i((0,x_2,\ldots,x_d))=f_i\big((\langle z,v_i\rangle,x_2,\ldots,x_d)\big)\] for some 
$x_2,\ldots,x_d\in \R$. Thus 
 $z\in f_i(A_i \times \mathbb R^{d-1})$ iff $\langle z,v_i\rangle \in A_i$,
 or, equivalently,
 $z\in \mathbb R^d \setminus f_i(A_i \times \mathbb R^{d-1})$ iff $\langle z,v_i\rangle \in \mathbb R \setminus A_i$.
 Defining $g(t):=\langle \gamma(t) , v_i \rangle$ and noting that $g(1)=\langle \delta v,v_i\rangle=\delta c_i>0$ we get, using the generalized change of variable theorem \cite[Theorem 3]{Leader2003} in the penultimate equality,
 \begin{align*}
 \lefteqn{\int_0^1 \langle \dot{\gamma}(t) , v_i \rangle\chi_{\mathbb R^d \setminus f_i(A_i \times \mathbb R^{d-1})}(\gamma(t)) \chi_{(-\delta,\delta)}(\langle \gamma(t) , v_i \rangle) \,dt}\\
 &=\int_0^1 \langle \dot{\gamma}(t) , v_i \rangle\chi_{\R\setminus A_i}\big(\langle  \gamma(t),v_i\rangle\big) \chi_{(-\delta,\delta)}(\langle \gamma(t) , v_i \rangle) \,dt\\
 &=\int_0^1 g'(t) \chi_{(\R\setminus A_i)\cap(-\delta,\delta)}\big(g(t)\big)  \,dt=\int_{g(0)}^{g(1)}  \chi_{(\R\setminus A_i)\cap(-\delta,\delta)}\big(s\big)  \,ds\le a_\delta.
\end{align*}
Combining this with \eqref{eq:calc-len-ydelta} finally yields
 \begin{align*}
 \|y_\delta\|
  & \le a_\delta +(\ell(\gamma)-\delta)+ \sqrt{1-c^2}\mathcal H^1(\gamma \setminus f_i(A_i \times \mathbb R^{d-1})) + \mathcal H^1(\gamma \cap f_i(A_i \times \mathbb R^{d-1}))
 \end{align*}
 We proceed, using  \eqref{c-delta-choice}, \eqref{eq:hausdorff-gamma}, and \eqref{eqn:len-gamma-ge1k}, to get
  \begin{align*}
    \vertii{y_\delta}& \le  a_\delta + \delta q+\sqrt{1-c^2}\mathcal H^1(\gamma \setminus f_i(A_i \times \mathbb R^{d-1})) + \big(\mathcal{H}^1(\gamma)-\mathcal H^1(\gamma \setminus f_i(A_i \times \mathbb R^{d-1}))\big)\\
    &\le 3q\delta + \delta q+(\sqrt{1-c^2}-1)\frac{\delta}{k}+\mathcal{H}^1(\gamma)= \mathcal{H}^1(\gamma) - 2q\delta
    \le \mathcal{H}^1(\gamma) \Big(1 - \frac{2q}{1+q}\ \Big) \\
   &=  \frac{1-q}{1+q}\mathcal{H}^1(\gamma).
 \end{align*}
Thus $\frac{1+q}{1-q}\vertii{y_\delta}  \le \mathcal H^1(\gamma)$, which contradicts \eqref{eq:hausdorff-gamma}. This shows that $A$ is impermeable.
\end{proof}

As an immediate consequence of Proposition \ref{prop:impermeableintersection} we have the following result.

\begin{corollary}\label{cor:impermeableproduct}
 Let $A_1,\ldots,A_d\subset \R$ be sets with positive Lebesgue measure. Then $\prod_{i=1}^d A_i$ is impermeable in $\R^d$.
\end{corollary}

\begin{example}\label{ex:closedneeded}
With help of Corollary \ref{cor:impermeableproduct} we will exhibit a $\sqrt{2}$-quasi permeable set (as defined in \Cref{prop:c-finperm=>perm}) $\Theta \subset\mathbb R^2$ of Lebesgue measure zero that is impermeable. Note that, by \Cref{prop:c-finperm=>perm}, $\Theta$ cannot be closed. The set $\Theta$
will be a countable union of boundaries of squares --- and hence of Hausdorff dimension 1 --- that are defined by using approximations of the  Smith-Volterra-Cantor set.
Let us start with its definition. Set $F_{0}:=[0,1]$. If for some $n\in \N$ the set $F_{n-1}=\bigcup _{k=1}^{2^{n-1}}[a_{k},b_{k}]$ with $0=a_{1}<b_{1}<a_{2}<b_{2}<\cdots <a_{2^{n-1}}<b_{2^{n-1}}=1$ has already been constructed, we inductively define
\[
{ F_{n}:=\bigcup _{k=1}^{2^{n-1}}\bigg(\Big[a_{k},{\frac {a_{k}+b_{k}}{2}}-{\frac {1}{2^{2n+1}}}\Big]\cup \Big[{\frac {a_{k}+b_{k}}{2}}+{\frac {1}{2^{2n+1}}},b_{k}\Big]\bigg)}.
\]
 (Note that all of the $2^n$ intervals constituting $F_n$ have the same length, denoted by $\lambda_n$.)
Then $F:=\bigcap_{n\ge0} F_n$ is the {\em Smith-Volterra-Cantor set}, a Cantor set of positive Lebesgue measure. Our set $\Theta$ is now defined in terms of the approximations $F_n$ by setting
\[
\Theta := \bigcup_{n \ge 0} \partial(F_n \times F_n)
\]
(see Figure~\ref{fig:SVC}).
\begin{figure}[h]
\includegraphics[width=0.25\textwidth]{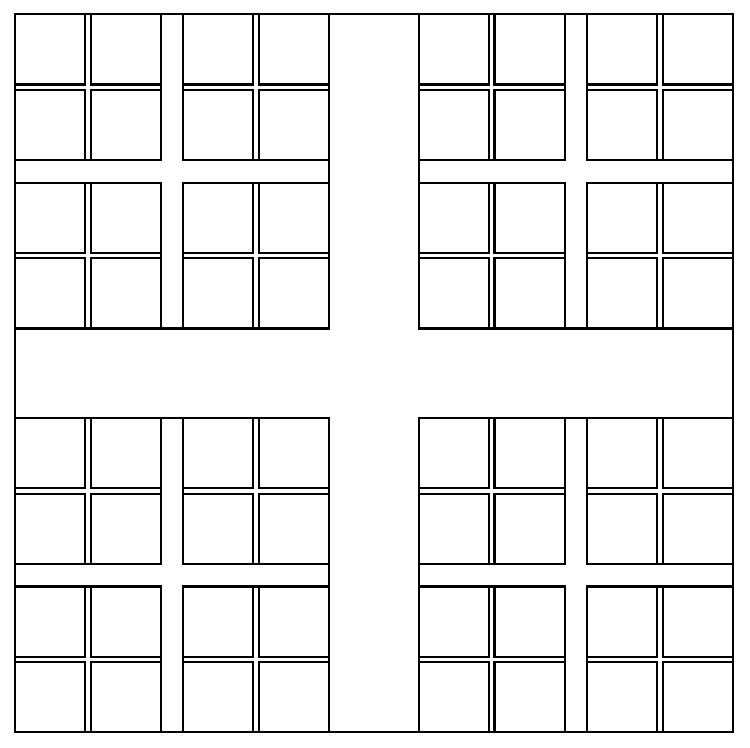} \hskip 1cm
\includegraphics[width=0.25\textwidth]{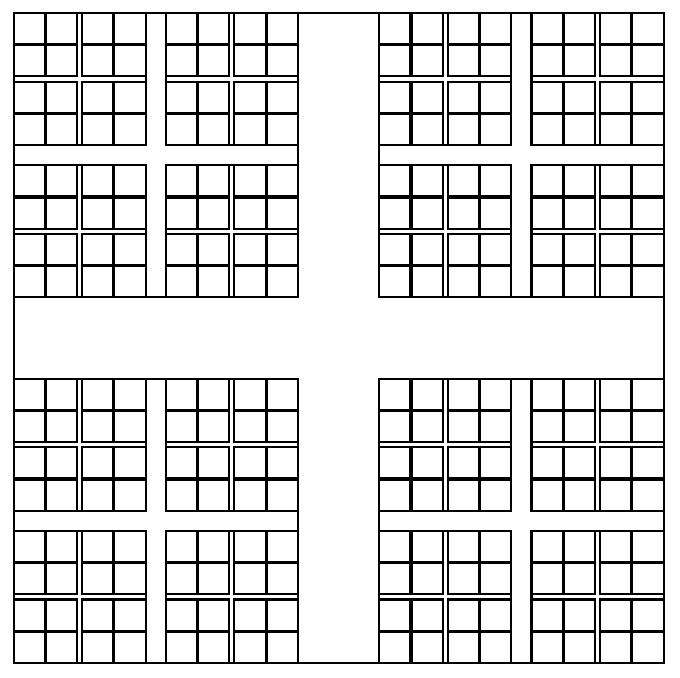}
\caption{Two approximations of the set $\Theta$. The width of the gaps between the squares tends to zero much faster than the width of the squares. We construct a path $\gamma$ connecting two points $(x_1,y_1),(x_2,y_2)\in \R^2$ that consists of axis-parallel line segments inside the gaps. This path $\gamma$ has countable intersection with $\Theta$.}\label{fig:SVC}
\end{figure}

Let us next prove that $\Theta$ is $\sqrt{2}$-quasi permeable.
To that end, fix $(x_1,y_1),(x_2,y_2) \in \mathbb R^2$ and let $\varepsilon > 0$.
By symmetry we may assume w.l.o.g.\ that $x_2\le x_1$ and $y_2\le y_1$.
Let $(c_0,d_0) \in \mathbb R^2$ be such that
$
 \vertii{(c_0,d_0) - (x_i,y_i)} \le \frac12\vertii{(x_1,y_1)-(x_2,y_2)}+\varepsilon
$
for $i \in\{1,2\}$ and so that $c_0 \notin F_{k_0}\cup\{x_1,x_2\}$ and $d_0\notin F_{k_0}\cup\{y_1,y_2\}$ for some $k_0 \in \mathbb N$. Now, we construct a path $\gamma_1$ connecting $(c_0,d_0)$ and $(x_1,y_1)$ (the path $\gamma_2$ connecting  $(c_0,d_0)$ and $(x_2,y_2)$ can be constructed analogously, and we omit the details). We assume  that $c_0 < x_1$ and $d_0 < y_1$ (the other constellations can be treated analogously).
For every $i \ge 1$ we iteratively select  real numbers $c_i \in (c_{i-1},x_1)$ and $d_i \in (d_{i-1},y_1)$ so that $|c_i - x_1| \le 2^{-i}\vertii{(x_1,y_1)-(x_2,y_2)}$, $|d_i-y_1| \le 2^{-i}\vertii{(x_1,y_1)-(x_2,y_2)}$ and $c_i,d_i \notin F_{k_i}$ for some $k_i \in \mathbb N$. Then we let $\gamma_1$ be the path obtained by concatenation of the axis-parallel line segments $\overline{(c_0,d_0)(c_1,d_0)}, \overline{(c_1,d_0)(c_1,d_1)}, \overline{(c_1,d_1)(c_2,d_1)}, \ldots$ By construction, each of these line-segments intersects $\Theta$ finitely many times. Thus, the only possible accumulation point of $\Theta \cap \gamma_1$ is at $(x_1,y_1)$. Constructing $\gamma_2$ along the same lines, we can concatenate $\gamma_1$ and $\gamma_2$ to a path $\gamma$ connecting $(x_1,y_1)$ and $(x_2,y_2)$ with $\ell(\gamma) \le \sqrt{2}\big(\vertii{(x_1,y_1) - (x_2,y_2)} + 2\varepsilon\big)$. 
Since $\varepsilon>0$ was arbitrary, $\Theta$ is $\sqrt{2}$-quasi permeable.

Finally, we show that $\Theta$ is impermeable. To this end we consider an arc $\gamma \colon[0,1] \to \mathbb R^2$. We claim that
\begin{equation}\label{eq:no_escape}
 \overline{\gamma\cap \Theta} = \gamma\cap \overline{\Theta}.
\end{equation}
Since $F$ has positive measure and $F\times F \subset \overline{\Theta}$, the impermeability of $\Theta$ then follows from \eqref{eq:no_escape} and Corollary \ref{cor:impermeableproduct}.
We have to establish \eqref{eq:no_escape}. Since $\gamma\cap \Theta \subseteq \overline{\gamma\cap \Theta}\subseteq \gamma\cap \overline{\Theta}$ trivially holds, it remains to show that $\gamma\cap (\overline{\Theta}\setminus \Theta) \subseteq  \overline{\gamma\cap \Theta}$. To see this, take $x \in \gamma\cap (\overline{\Theta}\setminus \Theta)$ and $\varepsilon > 0$. Let $i\in \mathbb N$ be large enough so that $\lambda_i < \min\{\diam(\gamma),\varepsilon\}$. By the construction of $\Theta$ there exists a square $Q\subset F_i\times F_i$ of side-length $\lambda_i$ so that $x \in Q$ and $\partial Q \subset \Theta$. By the choice of $i$, we have that $\gamma \cap \partial Q \ne \emptyset$.
Since such $i$ and $Q$ can be found for all $\varepsilon > 0$, we conclude that
$x \in \overline{\gamma\cap \Theta}$. Thus, \eqref{eq:no_escape} holds.
\end{example}

The remaining examples in this section concern arcs in $\mathbb{R}^2$ with positive Lebesgue measure, {\it i.e.}, so-called {\em Osgood curves}.

\begin{example}\label{ex:osgood}
\begin{figure}[h]
\includegraphics[width=0.333\textwidth]{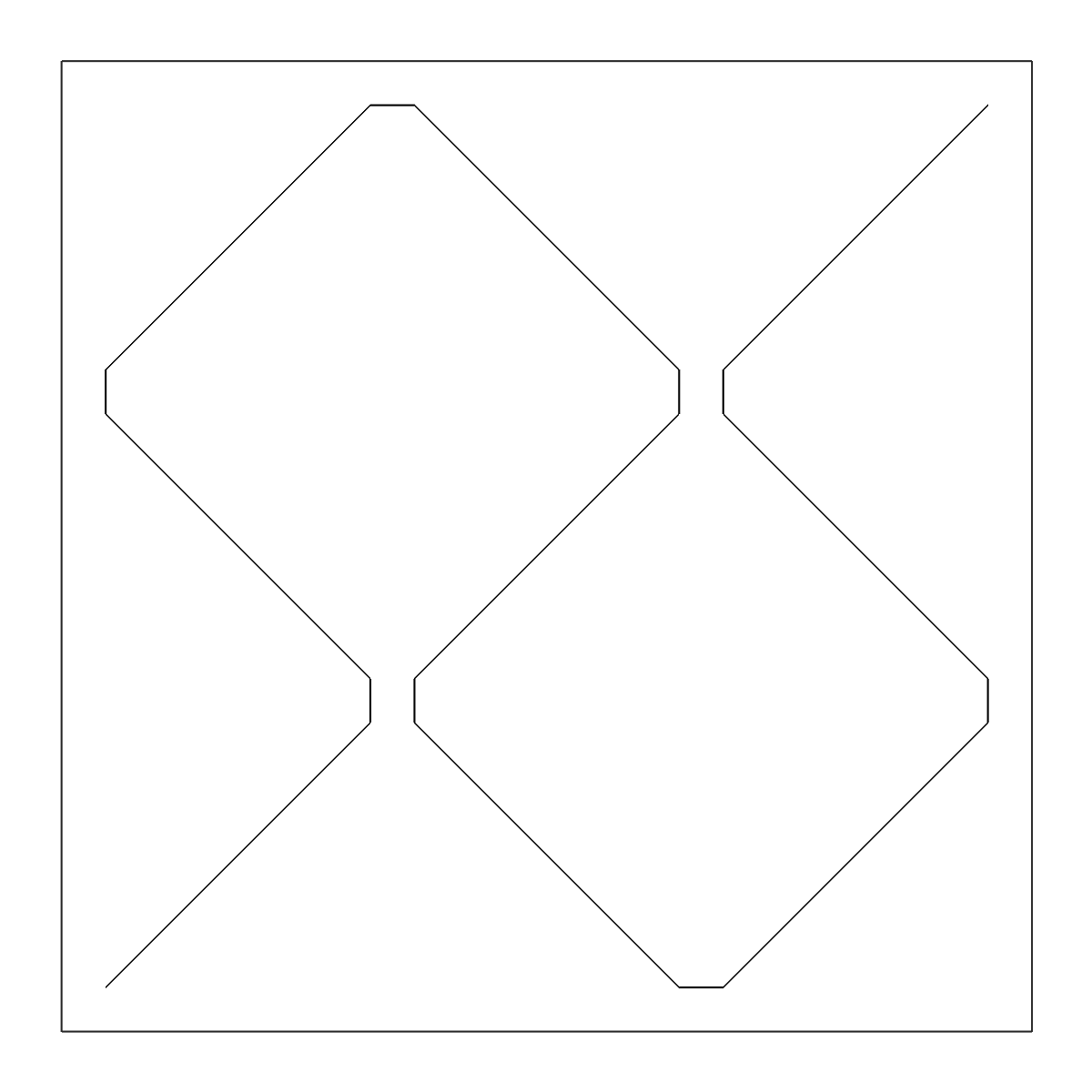}
\includegraphics[width=0.333\textwidth]{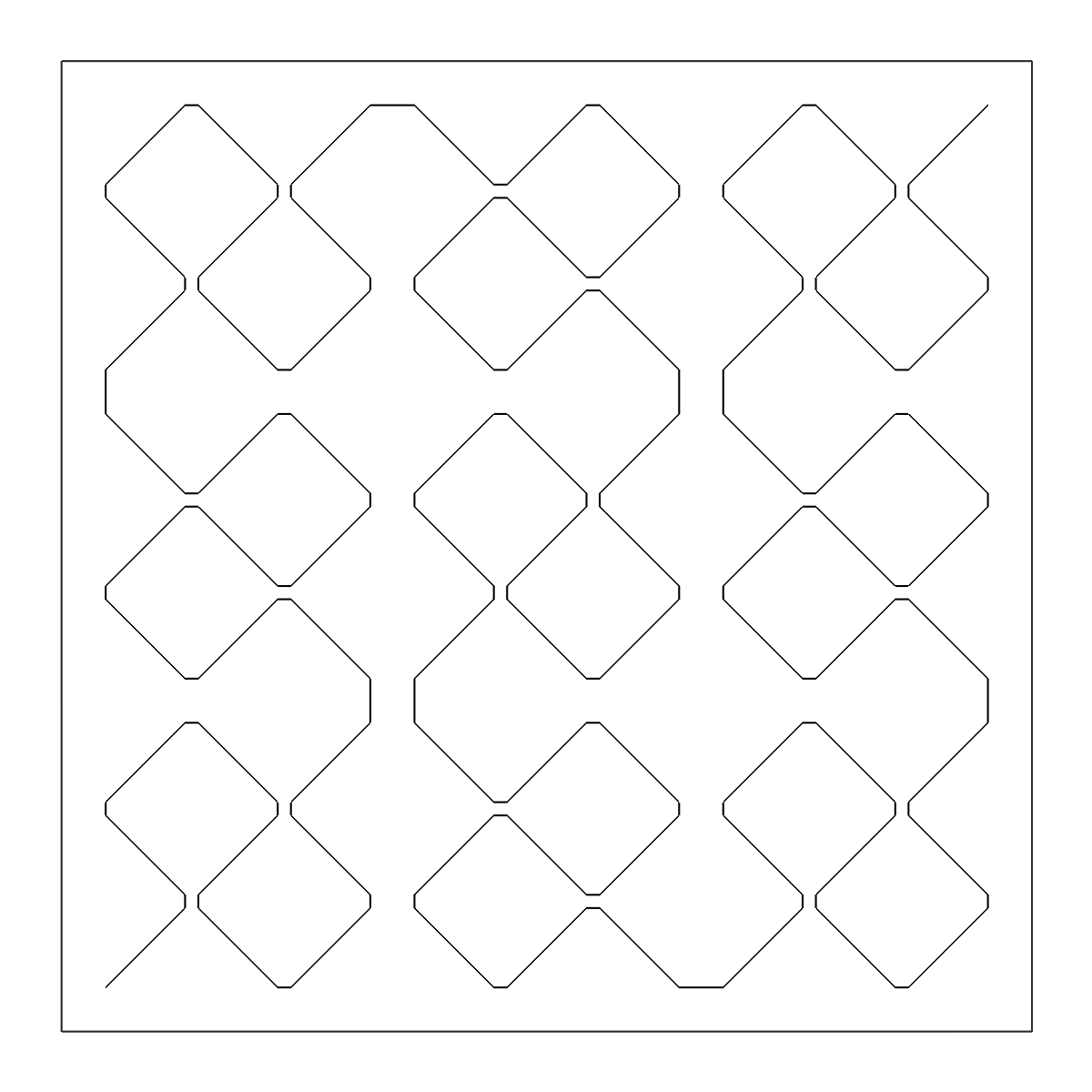}
\includegraphics[width=0.333\textwidth]{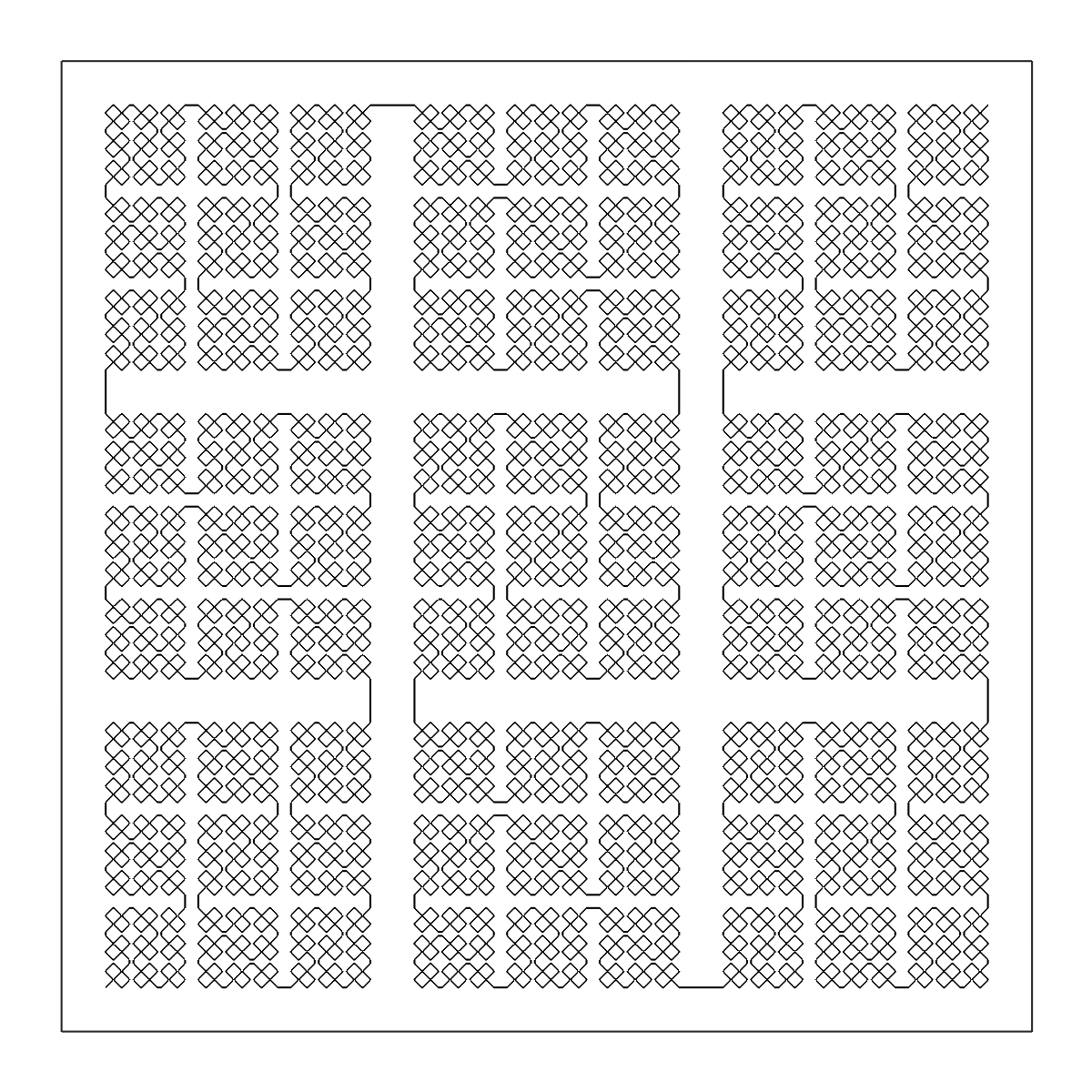}
\caption{One, two, and ten iterations of Osgood's construction}\label{fig:osgood}
\end{figure}
The classical Osgood curve $\gamma\colon[0,1]\to [0,1]^2$ goes back to \cite{OsgoodAJC}. We give a short reminder of its construction and refer to \cite[Section 8.2]{Sagan:94} for more details: Start with the path in Figure \ref{fig:osgood}. Then, iteratively substitute each of the diagonal segments of the path by a scaled-down version of that path, modified in a way that the relative size of the gaps between the segments parallel to a coordinate axis converges to zero exponentially faster than the diameter of the path.
Indeed, not only has $\gamma$ positive measure, it even contains the Cartesian product of a  Cantor set of positive measure with itself, which is impermeable by \Cref{cor:impermeableproduct}. Thus also $\gamma$ is impermeable. 
\end{example}

\begin{figure}[h]
\includegraphics[width=0.5\textwidth]{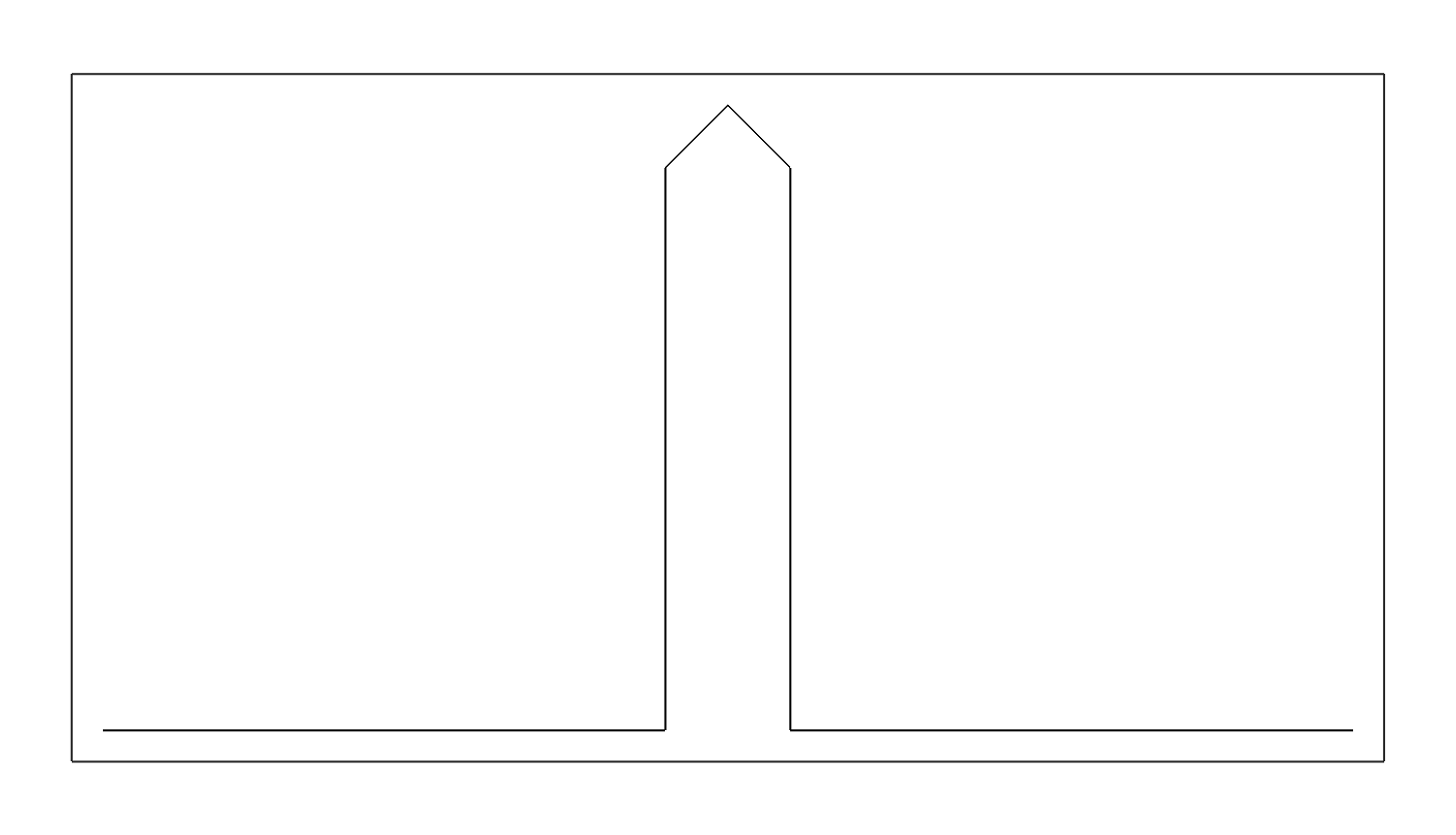}%
\includegraphics[width=0.5\textwidth]{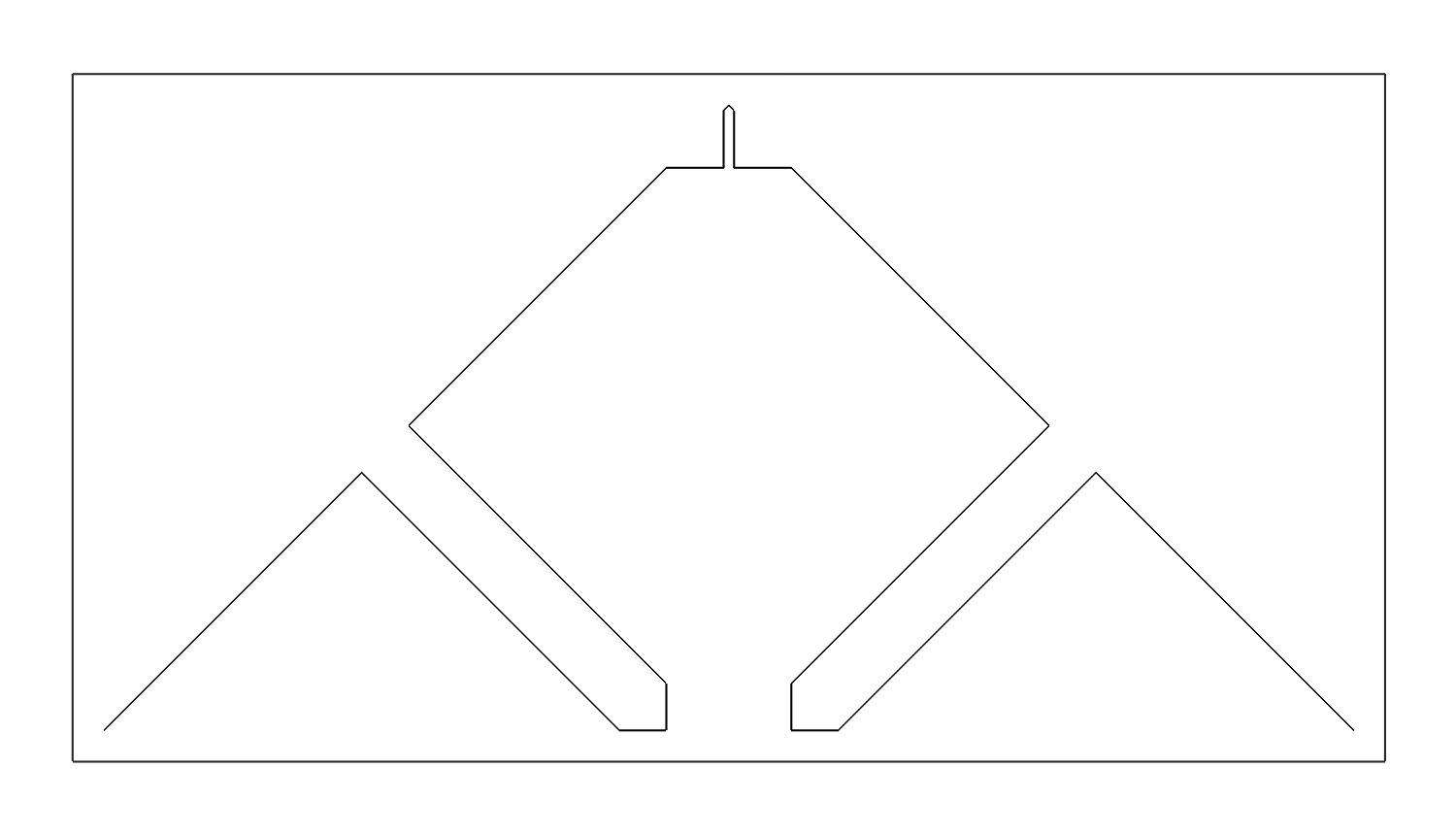} \\
\includegraphics[width=0.5\textwidth]{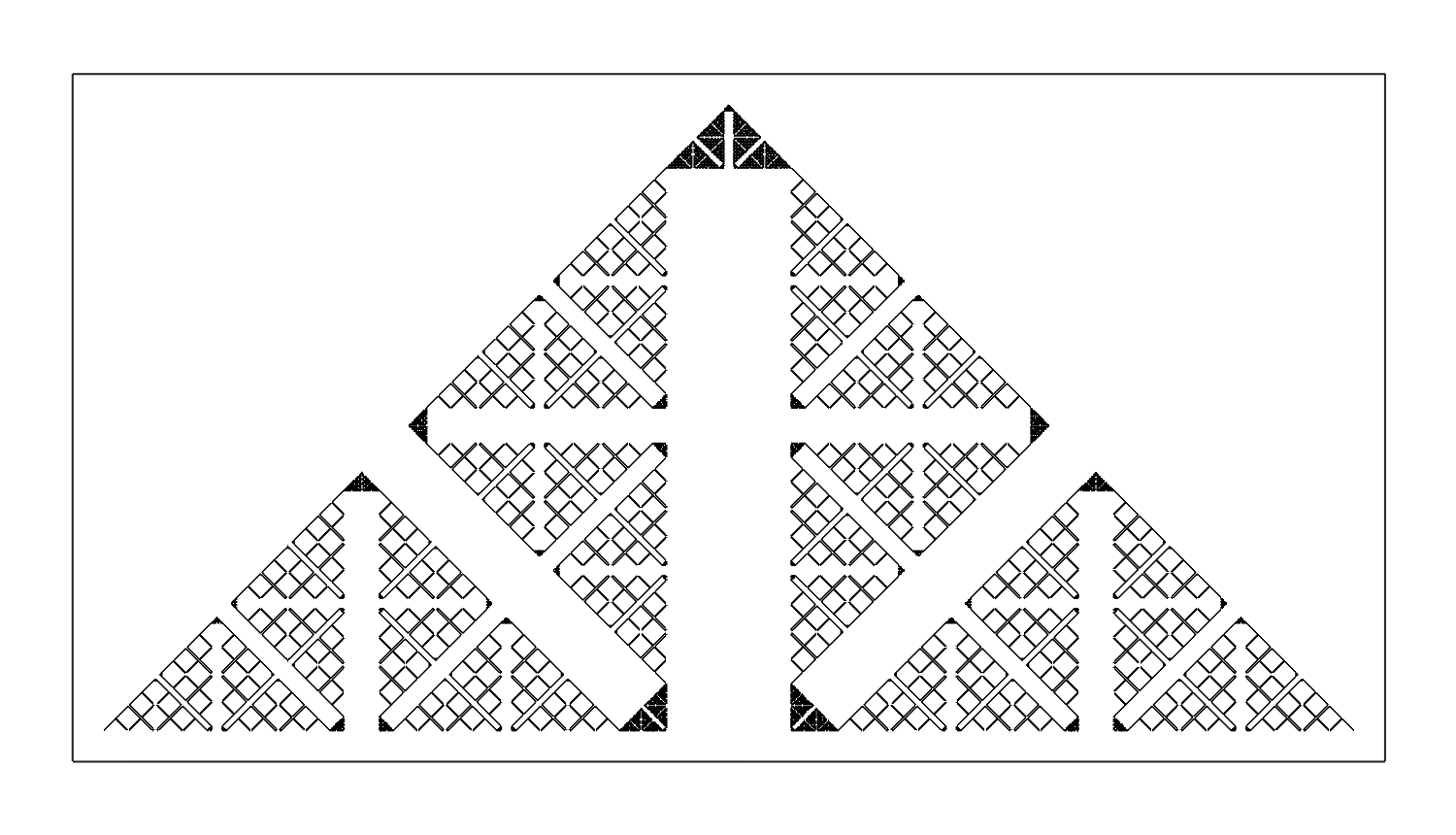}%
\caption{One, two, and ten iterations of Sierpiński's construction\label{fig:osgoodSierp}}
\end{figure}
\begin{example}\label{ex:curveSier}
Sierpiński's construction (see again \cite[Section 8.2]{Sagan:94}) of an Osgood curve is depicted in Figure \ref{fig:osgoodSierp}. It contains the intersection of four sets $f_i(A_i\times \mathbb R)$, $i\in \{1,\ldots,4\}$, where $f_1,f_2,f_3,f_4$ are rotations and $A_1,A_2,A_3,A_4$ are Cantor sets of positive measure and is thus impermeable by \Cref{prop:impermeableintersection}.
\end{example}

\begin{example}\label{ex:curve}
\begin{figure}[h]
\includegraphics[width=0.56\textwidth]{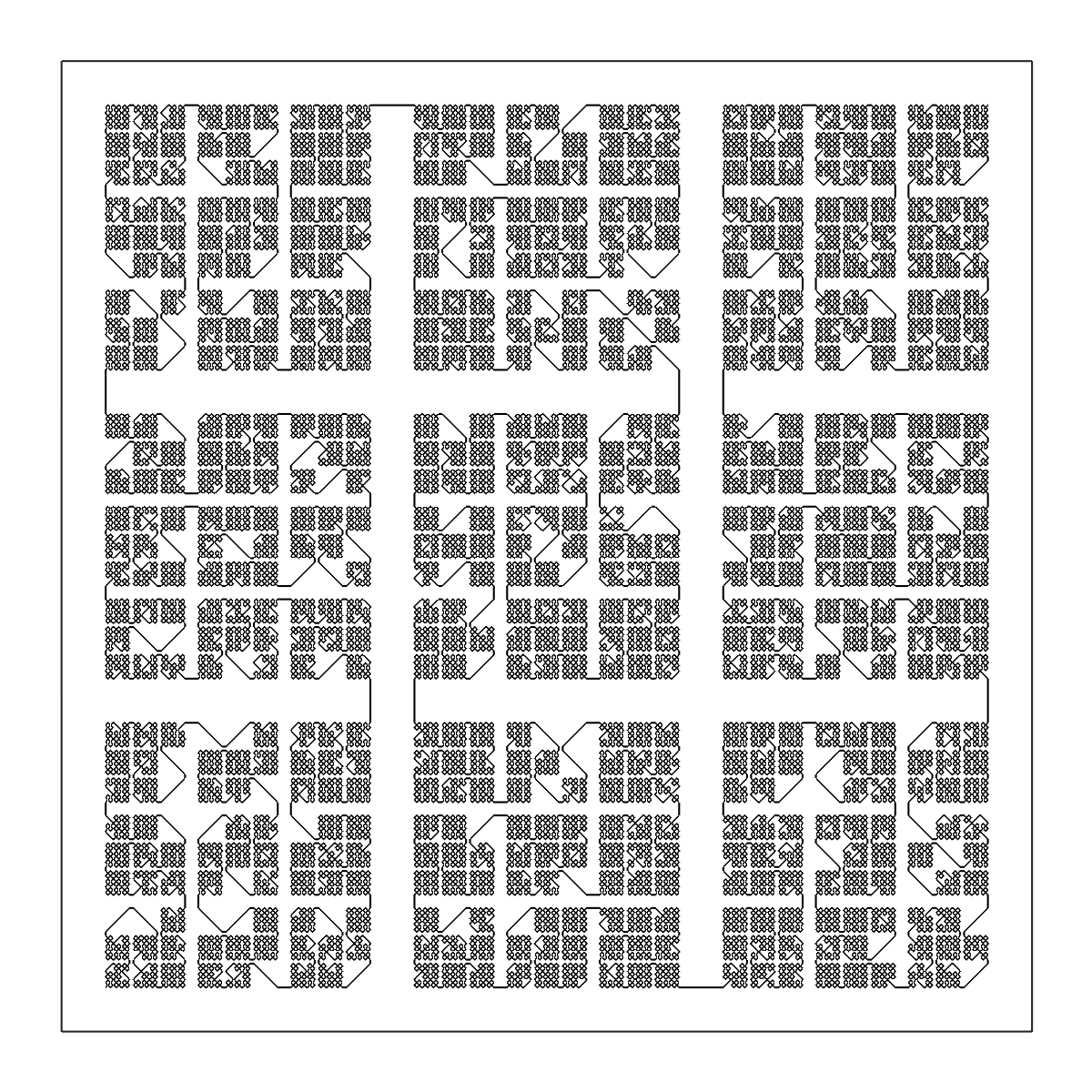}
\caption{Five iterations of a permeable Osgood curve}\label{fig:posgood}
\end{figure}
We next construct an Osgood curve which is permeable. To this end, 
let $C$ be the Osgood curve from \Cref{ex:osgood}. Let further 
$(A_m)_{m\ge 1}$ be a sequence of closed permeable subsets of $[0,1]^2$ with $\lim_{m\to \infty}\lambda(A_m)=1$. Such a sequence exists by \Cref{rem:LebPerm}. Additionally, assume that the dense sequence $(q_i)_{i\in \N}$, chosen in \Cref{rem:LebPerm}, has the property that for each line segment $\overline {q_iq_j}$ the angle included with the $x$-axis is not a multiple of $\frac{\pi }{4}$. We then have  $\lim_{m\to \infty}\lambda(C\cap A_m)=\lambda(C)$, so there exists $m_0$ such that $\lambda(C\cap A_{m_0})>0$. We now change Osgood's construction as follows: To every diagonal line segment in the $n$-th iteration of Osgood's construction there corresponds an axis-parallel square of which the endpoints of the line segments are the square's corners. Let $\mathcal Q_n$ denote
the collection of those squares and let $\mathcal Q:=\bigcup \mathcal Q_n$. 
We only refine in those squares which contain elements of $C\cap A_{m_0}$ (see \Cref{fig:posgood}). We claim that the resulting curve $\Theta$ is permeable. Suppose we have two distinct points $x,y\in \{q_i\colon i\in \N\}$. Then $\overline{xy}\cap A_{m_0}=\emptyset$. Suppose $\overline{xy}\cap \Theta$ is infinite. Then $\overline{xy}\cap \Theta$ has an accumulation point $z$. Thus there is a sequence $(Q_k)_{k\in \N}$,  $Q_k\in \mathcal Q$,  of axis-parallel squares which contain elements of $C\cap A_{m_0}$ and such that $Q_k\subset B_{\frac{1}{k}}(z)$. But then $z\in C\cap A_{m_0}$, since $C\cap A_{m_0}$ is closed, contradicting $\overline{xy}\cap A_{m_0}=\emptyset$. Now we can conclude permeability of $\Theta$ by \Cref{lemma:outside-theta-closed} (because $\Theta$ is closed, it suffices to consider points $x$ and $y$ from a dense subset of $\R^2\setminus \Theta$).
\end{example}

\section{Permeability and dimension}\label{sec:dims}
In this section we investigate the relationship between permeability (and its variants) and various notions of dimension for subsets of $\R^d$ (for the relationship between percolation and dimension we refer to \cite{broman2013}). We will get the same picture for most of the notions of dimension that we study: If the dimension is contained in the interval $[d-1,d]$, no general statement about permeability can be made. However, sets with dimension less than $d-1$ are null permeable. Only topological dimension shows a different behavior. 
Among the dimensions considered in this article, the Nagata dimension is the most challenging one. However, our result on Nagata dimension, namely \Cref{th:Nagata}, turns out to be quite useful. Indeed, we can apply this result in Section~\ref{sec:3difs} in order to show null permeability for a large class of self-similar sets in $\R^d$ with $d\ge 3$.

\subsection{Hausdorff dimension}\label{sec:dimH}
Let $d\ge 2$. We denote the Hausdorff dimension of a subset $A\subset \R^d$ by $\dim_H(A)$. We start with the following result.

\begin{proposition}\label{prop:dimhindep}
Let $d\ge 2$. 
\begin{itemize}
\item[(1)] For each $s\in [0,d]$ there exists a null permeable set $\Theta\subset \R^d$ with $\dim_H(\Theta)=s$.
\item[(2)] For each $s\in [d-1,d]$ there exists an impermeable set $\Theta\subset \R^d$ with $\dim_H(\Theta)=s$.
\end{itemize}
\end{proposition}

Note that the situation is different for $d=1$. Any set $\Theta\subset\R$ with $\dim_H(\Theta)>0$ is impermeable.

\begin{proof}
To prove (1) first observe that by \Cref{rem:LebPerm} there are null permeable subsets of $\mathbb{R}^d$ with positive Lebesgue measure. Thus a null permeable set can have Hausdorff dimension $s=d$. Let now $s\in (0,d)$ be arbitrary. For $a:=\frac sd \in(0,1)$ let $C_a$ be a self-similar Cantor set with $\dim_H(C_a)=a$.  Then $\dim_H((C_a)^d)=ad=s$  (here we use that, if the box counting dimension of $Y$ coincides with its Hausdorff dimension, then $\dim_H(X\times Y)=\dim_H(X)+\dim_H(Y)$, see \cite[Corollary 7.4]{Falconer2003}). The set $(C_a)^d$ is null permeable for each $a\in(0,1)$ by \Cref{prop:c-quasiconvex=>nullperm}. Because for $d\ge 2$ the singleton $\{0\}$ is null permeable with $\dim_H(\{0\})=0$, assertion (1) follows also for~$s=0$. 

To show (2), observe that for an uncountable set $P\subset [0,1]$ satisfying $\dim_H(P)=0$, we obtain the (obviously) impermeable set $[0,1]^{d-1}\times P$ with $\dim_H([0,1]^{d-1}\times P)=d-1$ (see again  \cite[Corollary 7.4]{Falconer2003}). Let $a\in(0,1]$ be arbitrary. With a self-similar Cantor set $P=C_a$ satisfying $\dim_H(C_a)=a\in (0,1]$ we get the impermeable set  $[0,1]^{d-1}\times C_a$ satisfying $\dim_H([0,1]^{d-1}\times C_a)=a+d-1\in (d-1,d]$. Since $a\in(0,1]$ was arbitrary, this proves (2).
\end{proof}

\Cref{prop:dimhindep} indicates that permeability and Hausdorff dimension are quite ``independent''. However, we are able to give some results showing that low Hausdorff dimension implies permeability and even null permeability. We refer to Kal\-my\-kov {\it et al.}~\cite[Theorem~1.3]{rajala2019} for a related (but different) result. In the sequel we use $U^\bot$ to denote the orthogonal complement of a linear subspace $U$ of $\mathbb{R}^d$. 

\begin{theorem}\label{th:Haus}
Let $d\ge 1$. If $\Theta\subset \R^d$ satisfies $\mathcal{H}^{d-1}(\Theta)=0$, then $\Theta$ is null permeable. 
\end{theorem}

\begin{proof}
Since the case $d=1$ is trivial, we may assume that $d\ge 2$. Let $x,y\in \R^d$ and $\delta>0$. Let $B$ be the $(d-1)$-dimensional ball  in $\frac{x+y}{2}+(\R(y-x))^\bot$ with center $\frac{x+y}{2}$ and radius $\delta$. Consider the map 
\begin{equation}\label{eq:cHdf}
c\colon [-1,1]\times B\to\R^d, (s,z)\mapsto \frac{1}{2}(x+y)+s\frac{y-x}{2}+(1-|s|)z.
\end{equation}
This parametrizes the double cone $\mathcal{C}:=c([-1,1]\times B)$ over $B+\frac{1}{2}(x+y)$ with apices $x$ and $y$ (see \Cref{fig:double-cone}). 
Note that the restriction of $c$ to $(-1,1)\times B$,  which takes away the apices $\{x,y\}$, is bijective. For $s\in (-1,1)$ let $p_{s}$ be the map 
$$p_{s}\colon c([\min(s,0),\max(s,0)]\times B)\to c(\{0\}\times B),c(t,z)\mapsto c(0,z).$$
\begin{figure}[ht]
\begin{center}
\begingroup%
  \makeatletter%
  \providecommand\color[2][]{%
    \errmessage{(Inkscape) Color is used for the text in Inkscape, but the package 'color.sty' is not loaded}%
    \renewcommand\color[2][]{}%
  }%
  \providecommand\transparent[1]{%
    \errmessage{(Inkscape) Transparency is used (non-zero) for the text in Inkscape, but the package 'transparent.sty' is not loaded}%
    \renewcommand\transparent[1]{}%
  }%
  \providecommand\rotatebox[2]{#2}%
  \newcommand*\fsize{\dimexpr\f@size pt\relax}%
  \newcommand*\lineheight[1]{\fontsize{\fsize}{#1\fsize}\selectfont}%
  \ifx\svgwidth\undefined%
    \setlength{\unitlength}{250.22255485bp}%
    \ifx\svgscale\undefined%
      \relax%
    \else%
      \setlength{\unitlength}{\unitlength * \real{\svgscale}}%
    \fi%
  \else%
    \setlength{\unitlength}{\svgwidth}%
  \fi%
  \global\let\svgwidth\undefined%
  \global\let\svgscale\undefined%
  \makeatother%
  \begin{picture}(1,0.36567569)%
    \lineheight{1}%
    \setlength\tabcolsep{0pt}%
    \put(0,0){\includegraphics[width=\unitlength,page=1]{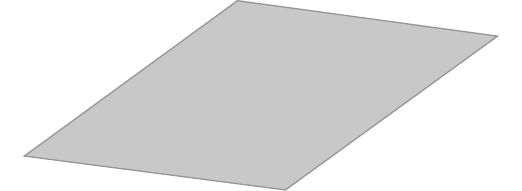}}%
    \put(0.59539201,0.14376863){\makebox(0,0)[t]{\lineheight{1.25}\smash{\begin{tabular}[t]{c}$\tfrac{1}{2}(x+y)$\end{tabular}}}}%
    \put(0,0){\includegraphics[width=\unitlength,page=2]{double-cone2.pdf}}%
    \put(0.03065245,0.03289642){\makebox(0,0)[t]{\lineheight{1.25}\smash{\begin{tabular}[t]{c}$x$\end{tabular}}}}%
    \put(0.9692403,0.26779979){\makebox(0,0)[t]{\lineheight{1.25}\smash{\begin{tabular}[t]{c}$y$\end{tabular}}}}%
    \put(0.49204479,0.29164938){\makebox(0,0)[t]{\lineheight{1.25}\smash{\begin{tabular}[t]{c}$z$\end{tabular}}}}%
    \put(0.36197908,0.11188143){\makebox(0,0)[t]{\lineheight{1.25}\smash{\begin{tabular}[t]{c}$s$\end{tabular}}}}%
    \put(0.3824004,0.19378374){\makebox(0,0)[t]{\lineheight{1.25}\smash{\begin{tabular}[t]{c}$c(s,z)$\end{tabular}}}}%
  \end{picture}%
\endgroup%
 \end{center}
\caption{Illustration of the double cone $\mathcal C$ and its parametrization $c$.}\label{fig:double-cone}
\end{figure}
Note that for each $s\in (-1,1)$ the map $p_{s}$ is Lipschitz. Take a strictly increasing sequence $(s_n)_{n\geq 0}$ with $s_n\in(0,1)$ and $\lim_{n\to\infty}s_n=1$. Then, because $p_{\pm s_n}$ is Lipschitz for each $n\in\N$ and $\mathcal{H}^{d-1}(\Theta)=0$,
\[
\Theta_{\infty}  :=\bigcup_{n=1}^\infty p_{-s_n}(\Theta\cap c([-s_n,0]\times B)) 
\cup \bigcup_{n=1}^\infty p_{s_n}(\Theta\cap c([0,s_n]\times B))
\]
is a countable union of subsets of $c(\{0\}\times B)=\frac{1}{2}(x+y)+B$, each of them having zero $\mathcal{H}^{d-1}$-measure by \cite[Proposition 2.2 and (2.9)]{Falconer2003}. Therefore, $\mathcal{H}^{d-1}(\Theta_\infty)=0$ and, hence, $\Theta_\infty$ cannot contain a $(d-1)$-dimensional ball in $\frac{1}{2}(x+y)+B$. Thus there is $z^*\in (\frac{1}{2}(x+y)+B)\setminus \Theta_{\infty}$, such that $\gamma:=\bigcup_{n\geq 1} p_{-s_n}^{-1}(z^*)\cup\bigcup_{n\geq 1} p_{-s_n}^{-1}(z^*)$ satisfies $\gamma\cap \Theta =\emptyset$. By construction, $\overline{\gamma}:=\gamma\cup\{x,y\}$ is a path connecting $x$ and $y$ that satisfies $\overline{\gamma}\cap\Theta\subseteq\{x,y\}$, whose length is bounded by $(\vertii{x-y}^2+4\delta^2)^{\frac12}$. Since $x,y\in \R^d$ and $\delta>0$ were arbitrary, this establishes null permeability of $\Theta$.
\end{proof}

From \Cref{th:Haus} we obtain the following corollary. 
\begin{corollary}
For $d\geq 2$, a countable subset of $\R^d$ is null permeable.
\end{corollary}

\begin{example}[Brownian motion]\label{ex:BM}
With probability one, a path $B$ of a $d$-dimensional {\it Brownian motion} in $\R^d$ for $d\geq 2$ satisfies $\dim_H(B)=2$ and $\mathcal{H}^2(B)=0$ (this is the content of \cite{Taylor53}, for $d=2$ it goes back to \cite{Levy40}; for a modern survey and more detailed information about Hausdorff measures of random fractals and graphs of Markov processes see \cite{Xiao04}). Therefore, if $d\geq 3$, a path of a $d$-dimensional Brownian motion is almost surely null permeable in $\R^d$.
\end{example}

In the proof of the next result we need the following special case of a result from Mattila~ \cite{Mattila:1995} (recall that the 0-dimensional Hausdorff measure is the counting measure). 

\begin{lemma}[{see \cite[Theorem~10.10]{Mattila:1995}}]\label{lem:Mattila1010}
Let $\Theta \subset \mathbb{R}^d$ be a Borel set with $\mathcal{H}^{d-1}(\Theta) < \infty$. Then for each line $g \subset \mathbb{R}^d$  passing through the origin we have
$
\#((g+a) \cap \Theta) < \infty 
$
for almost all $a \in g^\bot$.
\end{lemma}

Note that \cite[Theorem~10.10]{Mattila:1995} contains the assumption $\mathcal{H}^{d-1}(\Theta) > 0$. However, it is obvious (and mentioned in \cite[beginning of Chapter~10]{Mattila:1995}) that this assumption can be dropped for the part of the result that is formulated in our lemma.

\begin{theorem}\label{th:HausMeasure}
Let $d\ge 1$. If $\Theta\subset \R^d$ is a Borel set with $\mathcal{H}^{d-1}(\Theta)<\infty$, then $\Theta$ is permeable. 
\end{theorem}

\begin{proof}  
Since the case $d=1$ is trivial, we may assume that $d\ge 2$. Let $x,y\in \R^d$ and $\delta>0$. We use the notation from the proof of Theorem~\ref{th:Haus}. Consider the cone $c([-1,1]\times B)$.  For each $s\in[0,1)$, the restriction $c_s\colon[-s,s]\times B\to c([-s,s]\times B)$ of the map $c$ from \eqref{eq:cHdf} is a bi-Lipschitz map, so $\mathcal{H}^{d-1}\big(\Theta\cap c([-s,s]\times B)\big)<\mathcal{H}^{d-1}(\Theta)<\infty$ implies that $\mathcal{H}^{d-1}\big(c_s^{-1}(\Theta\cap c([-s,s]\times B))\big)<\infty$ (see again \cite[Proposition 2.2 and (2.9)]{Falconer2003}). By Lemma \ref{lem:Mattila1010}, for each such $s$ there is a subset $D_s\subset B$ with $\mathcal{H}^{d-1}(D_s)=\mathcal{H}^{d-1}(B)$ such that for each $z\in D_s$ we have that
$(-s,s)\times\{z\}$ has finite intersection with $c_s^{-1}(\Theta\cap c([-s,s]\times B))$. 
Taking the image under $c_s$, we conclude that $c([-s,s]\times \{z\})\cap\Theta$ is finite for each $z\in D_s$. Take a strictly increasing sequence $(s_n)_{n\geq 1}$ with $s_n\in(0,1)$ and $\lim_{n\to\infty}s_n=1$. Then, $D_{\infty}:=\bigcap_{n\geq 1}{D_{s_n}}$ has the same $(d-1)$-dimensional Hausdorff measure as $B$ and for each $z\in D_{\infty}$, $c([-1,1]\times \{z\})$ is a polygonal chain of length at most $(\vertii{x-y}^2+4\delta^2)^{\frac12}$ connecting $x$ and $y$, such that $c([-s_n,s_n]\times \{z\}) \cap \Theta$ is finite for each $n\in\N$ . Thus $c([-1,1]\times \{z\}) \cap \Theta$ can have accumulation points only in $\{x,y\}$ and, hence, $\overline{c([-1,1]\times \{z\}) \cap \Theta}$ is countable. Since $x,y\in \R^d$ and $\delta>0$ were arbitrary, this implies that $\Theta$ is permeable.
\end{proof}

\begin{example}\label{ex:SierpinskiTetrahedron}
The {\it Sierpi\'nski tetrahedron} $T \subset \R^3$ has Hausdorff dimension $2$ with positive and finite Hausdorff measure. This can easily be derived from the observations in \cite[p.~335]{JonesCampa93}. Just note that the one-to-one mapping mentioned there is Lipschitz to get positivity of $\mathcal{H}^2$; finiteness of $\mathcal{H}^2$ --- which is important for us --- is obtained by choosing a suitable sequence of $s$-covers (see \Cref{def:covering}) with $s\to0$. Thus $T$ is permeable by Theorem~\ref{th:HausMeasure}.
\end{example}

In view of Theorem \ref{th:HausMeasure}, one may ask whether a totally disconnected set $\Theta\subset\R^d$ with $\mathcal{H}^{d-1}(\Theta)<\infty$ is always null permeable. This has been answered affirmatively for closed sets in \cite[Theorem~1.3]{rajala2019}. We mention that the assumption of total disconnectedness cannot be dropped as a line segment in $\mathbb{R}^2$ is obviously not null permeable.

\subsection{Box counting and packing dimension}
Denote by $\dim_B(\Theta)$ the box counting dimension of a subset  $\Theta\in \R^d$.

\begin{proposition}\label{cor:Box}
Let $d \ge 2$. 
\begin{itemize}
\item[(1)] For each $s\in [0,d]$ there exists a null permeable set $\Theta\subset \R^d$ with $\dim_B(\Theta)=s$.
\item[(2)] For each $s\in [d-1,d]$ there exists an impermeable set $\Theta\subset \R^d$ with $\dim_B(\Theta)=s$.
\item[(3)] Each $\Theta\subset \R^d$ with $\dim_B(\Theta)< d-1$  is null permeable.
\end{itemize}
The same assertions hold for the packing dimension. 
\end{proposition}

\begin{proof}
We start with (1). The case $s=0$ is trivial. For the case $s\in(0,d)$, we may choose a self-similar Cantor set $C_a\subset \R$ with $\dim_B(C_a)=a<1$ and, hence, $\dim_B((C_a)^d)=ad=s$. By \Cref{prop:c-quasiconvex=>nullperm}, $(C_a)^d$ is null permeable. For the case $s=d$ we use again the set $A$ from \Cref{rem:LebPerm} and the inequality $\dim_B(A)\ge \dim_H(A)$.

Assertion (2) follows for $s\in (d-1,d]$ by using the same sets as the ones we used in the case of Hausdorff dimension in the proof of \Cref{prop:dimhindep}.  For $s=d-1$ we construct a Cantor set $C$ of box counting dimension $0$ (take a Cantor middle set with proportion of the ``middle parts'' increasing
sufficiently fast, \cite[Section 3.2]{Falconer:97}) and obtain the impermeable set $\Theta=C\times [0,1]^{d-1}$. 

It remains to show (3). Since $\dim_B(\overline{\Theta})=\dim_B(\Theta)<d-1$ we may assume that $\Theta$ is closed. Because $\dim_B(\overline{\Theta})\ge \dim_H(\overline{\Theta})$ (see {\it e.g.}~\cite[Inequality (3.17)]{Falconer2003}), the result follows from Theorem~\ref{th:Haus}. 
 
By~\cite[Inequality (3.29)]{Falconer2003}, the same assertions are true for the packing dimension by combining our results on the Hausdorff dimension from Section~\ref{sec:dimH} with the above assertions (1), (2), and~(3).
\end{proof}

\subsection{Assouad dimension}\label{sec:Adim}
Recently, also the Assouad dimension $\dim_A$ has been studied extensively (see {\em e.g.}~Fraser~\cite{fraser:21} or Robinson~\cite[Chapter~9]{Robinson:11}). In the following we show that w.r.t.\ permeability it behaves similarly as the dimensions treated above. Indeed, we have the following analog of \Cref{cor:Box}.

\begin{proposition}\label{cor:Assu}
Let $d\ge 2$. 
\begin{itemize}
\item[(1)] For each $s\in [0,d]$ there exists a null permeable set $\Theta\subset \R^d$ with $\dim_A(\Theta)=s$.
\item[(2)] For each $s\in [d-1,d]$ there exists an impermeable set $\Theta\subset \R^d$ with $\dim_A(\Theta)=s$.
\item[(3)] Each $\Theta\subset \R^d$ with $\dim_A(\Theta)< d-1$  is null permeable.
\end{itemize}
\end{proposition}

\begin{proof}
We start with (1). Since for $d\ge 2$ a singleton is null permeable with Assouad dimension zero, the assertion holds for $s=0$.
We know from \Cref{prop:dimhindep}~(1) that there are null permeable sets with Hausdorff dimension $s=d$ in $\mathbb{R}^d$. Since $\dim_H(\Theta) \le \dim_A(\Theta)$ ({\em cf}.~\cite[Lemma~2.4.3]{fraser:21}) these sets also have Assouad dimension $d$. Moreover, let $C_a$ be a self-similar Cantor set of Hausdorff dimension $a\in(0,1)$. Because $(C_a)^d$ is a null-permeable self-similar IFS attractor satisfying the open set condition, we have $\dim_H((C_a)^d) =\dim_A((C_a)^d)$ (see~\cite{Olsen:11}). Thus the case $0<s<d$ follows from \Cref{prop:dimhindep}~(1) as well. 

To prove (2) observe that there exists an uncountable set $B$ of Assouad dimension zero ({\em e.g.}~a Cantor set as constructed in \cite[Section 3.2]{Falconer:97}). Hence, $\dim_A([0,1]^{d-1} \times B)=d-1$ by \cite[Lemma~9.7]{Robinson:11}. With an analogous construction with a self-similar Cantor set $P=C_a$ of Assouad dimension $a\in (0,1]$ we can find impermeable subsets of $\R^d$ of any Assouad dimension in $s\in[d-1,d]$. 

Assertion (3) follows from Theorem~\ref{th:Haus}, because $\dim_H(\Theta) \le \dim_A(\Theta)$. 
\end{proof}

\subsection{Nagata dimension}
The integer-valued Nagata dimension (see {\it e.g.}~\cite{Assouad82}; this kind of dimension is also called Assouad-Nagata dimension) provides conditions for extensions of Lipschitz functions (see for instance \cite{LangSchlichenmaier05}) and is therefore of particular interest. Since we will study it in some detail, we recall its definition for the case where the ambient space is $\R^d$. This requires some preparations. 

\begin{definition}[$s$-cover and $r$-multiplicity]\label{def:covering}
An \emph{$s$-cover} $\mathcal{U}$ of a set $\Theta\subset \R^d$ is a collection of subsets of $\R^d$ such that $\Theta\subset \bigcup_{U\in\mathcal{U}}U$ and $\mathrm{diam}(U)\leq s$ for all $U\in \mathcal{U}$. The \emph{$r$-multiplicity} of $\mathcal{U}$ is given by the smallest $k\in \mathbb{N}$ such that each subset of $\Theta$ with diameter at most $r$ has a nonempty intersection with at most $k$ elements of $\mathcal{U}$. 
\end{definition}

\begin{definition}[$s$-separated sets]
We say that the sets $A,B\subset \R^d$ are {\em $s$-separated}, for some constant $s\ge 0$, if 
$d(A,B)  \ge  s$. A collection $\mathcal{U}$ of subsets of $\R^d$ is called {\em $s$-separated} if its elements are pairwise $s$-separated.
\end{definition}

Using this terminology, according to \cite[Proposition 2.5]{LangSchlichenmaier05}, the Nagata dimension of a set $\Theta\subseteq \R^d$  can be defined in two equivalent ways as follows. 

\begin{definition}[Nagata dimension; {{\em cf.}~\cite[Proposition 2.5]{LangSchlichenmaier05}}]\label{def:nagata} Let $\Theta\subset \R^d$. The following are equivalent definitions of the  \emph{Nagata dimension}\footnote{We only need item (2). We included item (1) just because it is the classical definition of Nagata dimension.}.
\begin{enumerate}
\item The Nagata dimension $\dim_N(\Theta)$ is defined as the smallest $n\in\mathbb{N}$ for which there exists $C > 0$ such that for all $r > 0$ the set $\Theta$ has a $Cr$-cover with $r$-multiplicity at most $n + 1$. 
\item The Nagata dimension  $\dim_N(\Theta)$ is the minimum of all $n\in\mathbb{N}$ with the following property: There exists $c>0$ such that for all $s>0$ the set $\Theta$ admits an $s$-cover $\mathcal{U}=\bigcup_{i=0}^n \mathcal{U}_i$, where $\mathcal{U}_i$ is a collection of $cs$-separated sets for each $i\in \{0,\ldots, n\}$. In this case we say that $\Theta$ has Nagata dimension $\dim_N(\Theta)$ {\em with constant~$c$}. 
\end{enumerate}
\end{definition}

For the relation between the constants $C$ and $c$ as well as for other equivalent characterizations of Nagata dimension we refer to~\cite[Proposition 2.5]{LangSchlichenmaier05} and its proof. It is evident from the definition that each $\Theta\subset \R^d$ with Nagata dimension $\dim_N(\Theta)$ has Nagata dimension $\dim_N(\Theta)$  with constant $c$ for some $c\in(0,1]$. 

\begin{proposition}\label{pro:nagataClosure}
Let $\Theta\subset \R^d$ be given. Then
$
\dim_N (\Theta) = \dim_N (\overline{\Theta}). 
$
\end{proposition}

\begin{proof}
Assume that for some $c>0$, $s>0$ and $n\in \mathbb{N}$ the set $\Theta$ admits an $s$-cover $\mathcal{U}=\bigcup_{i=0}^n \mathcal{U}_i$, where $\mathcal{U}_i$ is $cs$-separated. As $\mathcal{U}_i$ is $cs$-separated for each $i\in \{0,\ldots, n\}$, also 
$
\overline{\mathcal{U}_i} := \{\overline{U}_i \colon U_i \in \mathcal{U}_i\}
$
is $cs$-separated  for each $i\in \{0,\ldots, n\}$. 
Since $\mathcal{U}$ covers $\Theta$, and every $\mathcal{U}_i$ is $cs$-separated (and thus locally finite), we have
\[
\overline{\Theta} \subset  \overline{\bigcup_{U\in \mathcal{U}} U}=\bigcup_{i=0}^n \overline{\bigcup_{U\in \mathcal{U}_i} U}=\bigcup_{i=0}^n \bigcup_{U\in \mathcal{U}_i}\overline{ U}= \bigcup_{U\in \mathcal{U}} \overline{U}.
\]
Thus $\overline{\Theta}$ admits the $s$-cover $\overline{\mathcal{U}}:=\bigcup_{i=0}^n \overline{\mathcal{U}_i}$, where $\overline{\mathcal{U}_i}$ is $cs$-separated for each $i\in \{0,\ldots, n\}$ and, hence, \Cref{def:nagata}~(2) yields $\dim_N( \overline{\Theta}) \le  \dim_N(\Theta)$. Since the reverse inequality follows from the monotonicity of the Nagata dimension, the result is proved.
\end{proof}

The following result is an immediate consequence of the definition of the Nagata dimension.

\begin{lemma}\label{lem:NagataSimilar} 
Assume that $\Theta\subset \R^d$ satisfies $\dim_N (\Theta)=k$ with constant $c$. If $f\colon \R^d \to \R^d$ is a similarity transformation then $\dim_N (f(\Theta))=\dim_N (\Theta)$ with constant $c$.
\end{lemma}

Also for the Nagata dimension we get an analog of \Cref{prop:dimhindep}.

\begin{proposition}\label{prop:dimnindep}
Let $d\ge 2$. 
\begin{itemize}
\item[(1)] For each $s\in \{0,\ldots,d\}$ there exists a null permeable set $\Theta\subset \R^d$ with $\dim_N(\Theta)=s$.
\item[(2)] For each $s\in \{d-1,d\}$ there exists an impermeable set $\Theta\subset \R^d$ with $\dim_N(\Theta)=s$.
\end{itemize}
\end{proposition}

\begin{proof}
To prove (1) we note that $\Z^s \times \{0\}^{d-s}$ is null permeable and has Nagata dimension $s$ for each $s\in \{0,\ldots, d\}$. 

To prove (2) take the uncountable set $B$ with Assouad dimension $0$ mentioned in the proof of \Cref{cor:Assu} and consider the impermeable set $\Theta=B\times [0,1]^{d-1}$. This set has $\dim_N(\Theta)=d-1$ because $\dim_N(B)\leq \dim_A(B)=0$ (the inequality is due to \cite[Theorem 1.1]{LeDonneRajala15}) and
\[
d-1=\dim_N(\{0\}\times [0,1]^{d-1}) \le \dim_N(B\times [0,1]^{d-1})=\dim_N(B)+\dim_N([0,1]^{d-1}) \le d-1
\] 
(see \cite[Theorem 2.6]{LangSchlichenmaier05} for the last inequality). This proves the case $s=d-1$. Since $\R^d$ is an impermeable set with $\dim_N(\R^d)=d$, (2) follows.
\end{proof}

The main goal of the present section is to establish the following result.

\begin{theorem}\label{th:Nagata}
Let $d\ge 2$. If $\Theta \subset \mathbb{R}^d$ satisfies $\dim_N(\Theta) \le d-2$ then $\Theta$ is null permeable.
\end{theorem}

\Cref{th:Nagata} can be viewed as a quantified version of a classical result by Mazurkiewicz on the topological dimension of cuts of the Euclidean space (see~\cite[\S59, II, Theorem~1]{Kuratowski68}).

The proofs of the analogs of \Cref{th:Nagata} for the Hausdorff dimension in \Cref{sec:dimH} are relatively easy. This is due to the fact that the Hausdorff dimension (and the Hausdorff measures) behaves well under Lipschitz maps and, {\it a fortiori}, under projections. In particular, the Hausdorff dimension cannot increase under a Lipschitz map. Since this is no longer true for the Nagata dimension, the proof of \Cref{th:Nagata} is much harder and relies on several auxiliary results.

\begin{proposition}\label{pro:LebesgueNagata}
If a measurable set $\Theta\subset \mathbb{R}^d$ has a Lebesgue point then  $\dim_N (\Theta) = d$.
\end{proposition}

Before we turn to the proof of this result we recall the definition of a porous set. Recall that $B_r(x)$ denotes the open ball with radius $r$ around $x$. 

\begin{definition}[Porosity; see {\cite[Section 2]{zajicek1976}}]\label{def:porosity}
 Let $\Theta$ be a subset of $\R^d$ and $q>0$.
\begin{enumerate}
\item
A point $x$ is called a $q$-\emph{porosity point} of $\Theta$, if for each $r>0$, there exists $y\in \R^d$ with $B_{qr}(y)\subset B_{r}(x)\setminus \Theta$.
\item $\Theta$ is called $q$-\emph{porous} if every point of $\Theta$ is a $q$-porosity point of $\Theta$.
\end{enumerate}
\end{definition}

From this definition it is immediate that a $q$-porous set $\Theta\subset \mathbb{R}^d$  does not have a Lebesgue point, neither does its closure $\overline{\Theta}$. Thus the $d$-dimensional Lebesgue measure of $\overline{\Theta}$ is zero. Proposition~\ref{pro:LebesgueNagata} now follows by combining \cite[Theorem~5.2]{Luukkainen:98} with \cite[Theorem~1.1]{LeDonneRajala15} (see also~\cite[Remark~6.14]{David21}). However, this proof runs via the Assouad dimension. For the sake of self-containment, we therefore provide a short and direct proof of Proposition~\ref{pro:LebesgueNagata}. Indeed, we will prove \Cref{pro:LebesgueNagata} by showing that each $\Theta \subset \mathbb{R}^d$ with $\dim_N(\Theta)\le d-1$ is $q$-porous for some $q>0$ (see \Cref{lem:ballsize}). To establish this lemma, we need the following classical result from dimension theory.

\begin{lemma}[{\cite[\S28, II, Theorem 6]{Kuratowski66}}]\label{lem:KurDim1}
Let $S=[p_0,\ldots,p_d]$ be a nondegenerate closed simplex with closed faces $S_i=[p_0,\ldots,p_{i-1},p_{i+1},\ldots,p_d]$, $i \in\{0,\ldots,d\}$. If $A_0,\ldots, A_d$ are $d+1$ closed sets  with $S=A_0\cup\cdots\cup A_d$ and $A_i \cap S_i=\emptyset$ for each $i \in\{0,\ldots,d\}$, then $A_0\cap\cdots\cap A_d\ne\emptyset$.
\end{lemma}

Proposition~\ref{pro:LebesgueNagata} is an immediate consequence of the following lemma. The uniformity of the porosity constant $q$ asserted in this lemma will be needed later (see Lemma~\ref{lem:InductionStartNagata1}).

\begin{lemma}\label{lem:ballsize}
Let $d\in \mathbb{N}$ and $c > 0$. Then there exists  $q>0$ such that every set 
$\Theta\subset\mathbb{R}^d$ with $\dim_N (\Theta) \le d-1$ with constant $c$ is $q$-porous. 
\end{lemma}

\begin{proof}
Let  $e_0:=0$ and let $e_i$ be the $i$-th standard basis vector, $i\in \{1,\ldots,d\}$.
Consider the closed $d$-simplex $S=\frac{1}{2}[e_0,e_1,\ldots,e_d] \subset B_1(0)$ with faces $S_j=\frac{1}{2}[e_0,\ldots,e_{j-1},e_{j+1},\ldots,e_d]$ ($0\le j \le d$). We fix $\varepsilon>0$  so small that $[S]_{c\varepsilon} \subset B_1(0)$ and that each set $X\subset \R^d$ with $\mathrm{diam}(X)\le \varepsilon$ satisfies $d(X,S_j)\ge c\varepsilon$  for at least one $j\in\{0,\ldots,d\}$. Now set~$q=\frac{c\varepsilon}3$.

Consider $\Theta\subset\mathbb{R}^d$ with $\dim_N (\Theta) \le d-1$ with constant $c$, $x\in \R^d$, and  $r>0$. We need to find $y\in \R^d$ with $B_{qr}(y)\subset B_r(x)\setminus \Theta$. In view of Lemma~\ref{lem:NagataSimilar}, by applying a suitable similarity transformation to $\Theta$, we may assume w.l.o.g.\ that $x=0$ and $r=1$, {\it i.e.}, we need to find $y\in \R^d$ with $B_{q}(y)\subset B_1(0)\setminus \Theta$.  By Definition~\ref{def:nagata}, the set $\Theta
$ admits an $\varepsilon$-cover $\mathcal{U}=\bigcup_{i=0}^{d-1} \mathcal{U}_i$, where $\mathcal{U}_i$ is $c\varepsilon$-separated for each $i\in \{0,\ldots, d-1\}$.  Set
\[
\mathcal{V}_j := \{U\in  \mathcal{U} \colon U \text{ is a distance at least $c\varepsilon$ away from }S_j \}\qquad (j\in\{0,\ldots,d\}),
\]
$A_0 := \bigcup_{U\in \mathcal{V}_0} [U]_{c\varepsilon/3}$, and for $j\in\{1,\ldots,d\}$
\[
A_j := \bigcup_{U\in \mathcal{V}_j \setminus(\mathcal{V}_0\cup\cdots\cup \mathcal{V}_{j-1})} [U]_{c\varepsilon/3}.
\]
Note that each $\mathcal {V}_j$ is locally finite and, hence, each $A_j$ is closed. By the choice of $\varepsilon$, $\mathcal{U}=\bigcup_{j=0}^{d} \mathcal{V}_j$. 
Thus $A_0\cup\cdots\cup A_d \supseteq [\Theta]_{c\varepsilon/3}$. Suppose that  $[\Theta]_{c\varepsilon/3} \supseteq S$. Then $A_0\cap S,\ldots, A_d\cap S$ satisfy the conditions of Lemma~\ref{lem:KurDim1}. Thus $A_0\cap \cdots\cap A_d \ne\emptyset$. This implies that there exist pairwise distinct $U_0,\ldots, U_d \in \mathcal{U}$ such that $[U_0]_{c\varepsilon/3}\cap\cdots\cap [U_d]_{c\varepsilon/3} \ne\emptyset$. By the pigeonhole principle at least two of the sets $U_0,\ldots, U_d$ belong to the same $\mathcal U_i$, and this $\mathcal U_i$ therefore cannot be $c\varepsilon$-separated. This contradiction yields $[\Theta]_{c\varepsilon/3} \not\supseteq S$. This implies that there is $y\in S$ such that the ball $B_{q}(y)$ of radius $q=\frac{c\varepsilon}3$ satisfies $B_{q}(y)\subset [S]_{q}\subset [S]_{c\varepsilon} \subset B_1(0)$ and $B_{q}(y) \cap \Theta =\emptyset$. This proves the lemma.
\end{proof}

We recall one more result from dimension theory. In its statement, we use the following terminology. For sets $X,Y,W,Z\subset\R^d$ with $X \subset W$ and $Y\subset W$ we say that $Z$ {\it separates} $W$ between $X$ and $Y$, if there exist $M$ and $N$ such that $W\setminus Z=M\cup N$ with $(\overline{M}\cap N)\cup(M \cap\overline{N})=\emptyset$, $X\subset M$, and $Y\subset N$ (sloppily speaking, $X$ and $Y$ lie in different connected components of $W\setminus Z$).

\begin{lemma}[{\cite[\S28, II, Theorem 8]{Kuratowski66}}]\label{lem:KurDim2}
Let $S=[p_0,\ldots,p_d]$ be a nondegenerate closed simplex with closed faces $S_i=[p_0,\ldots,p_{i-1},p_{i+1},\ldots,p_d]$, $i \in\{0,\ldots,d\}$. If $A_1,\ldots, A_d$ are $d$ open sets in  $S$ with $A_i \cap S_i=\emptyset$ for each $i \in\{1,\ldots,d\}$, and such that $A_1\cup\cdots\cup A_d$ separates $S$ between $\{p_0\}$ and $S_0$, then $A_1\cap\cdots\cap A_d\ne\emptyset$.
\end{lemma}

Let $d\in \mathbb{N}$, $c>0$, and let $q$ be the porosity constant from \Cref{lem:ballsize}.  
Choose a nondegenerate closed $(d-1)$-simplex $E_{0}=[p_1,\ldots,p_d] \subset \mathbb{R}^{d-1}\times \{0\}$ with barycenter $b_0=0$ and diameter bounded by $q$ . Set $p_0=(0,\ldots,0,4)$ and define 
a nondegenerate closed simplex 
$
E :=  [p_0,\ldots,p_d].
$
For $i\in\{0,\ldots,d\}$ let $E_i=[p_0,\ldots,p_{i-1},p_{i+1},\ldots,p_d]$ be a closed face of $E$
(note that for $i=0$ this coincides with the definition of $E_0$ above). Let
\[
T:= E \setminus (B_{q}(p_0) \cup B_{q}(b_0))
\]
(see Figure~\ref{fig:T}).
\begin{figure}[h]
     \centering
     \begin{subfigure}[b]{0.38\textwidth}
         \centering
        \includegraphics[height=0.38\textheight]{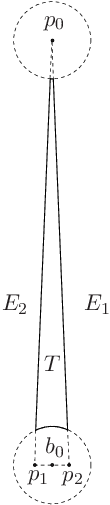}         
        \caption{} 
     \end{subfigure}
     \hfil
     \begin{subfigure}[b]{0.38\textwidth}
         \centering
         \includegraphics[height=0.38\textheight]{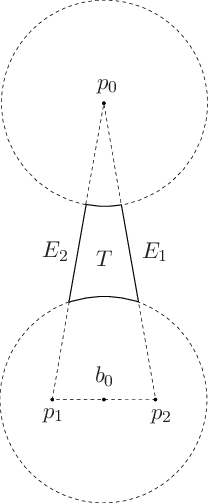}
         \caption{} 
     \end{subfigure}
 	\caption{Two examples of $T$. In (A) the radius $q$ of the dotted circles is
small, in (B) it is larger, which is not scaled properly, but better suited for the
visualization of the cover $\mathcal{U}$ in Figure~\ref{fig:U}.}
\label{fig:T}
\end{figure}
Throughout this section fix $\varrho>0$ in a way that each set $B$ of diameter bounded by
$(1+c)\varrho$ intersecting $T$ satisfies $B\cap(\{p_0\} \cup E_0)=\emptyset$ and
$B\cap E_i=\emptyset $ for some $i\in\{1,\ldots,d\}$.  Then, also throughout
this section, fix a finite cover $\mathcal{U}$ of $T$ by
open balls of diameter $\frac{c\varrho}4$. 
An example of a cover $\mathcal{U}$ is shown in Figure~\ref{fig:U}. There  we illustrate that $\varrho$ is chosen small enough.

\begin{figure}[h]
\includegraphics[height=0.28\textheight]{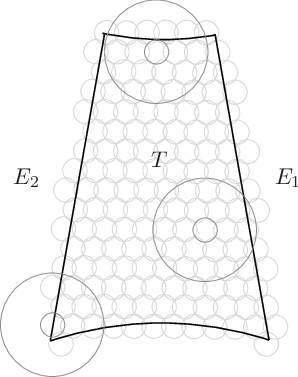}
\caption{The set $T$ from Figure~\ref{fig:T}~(B) enlarged and drawn together with the elements of a cover $\mathcal{U}$. The radius $\varrho$ is chosen so that no enlarged ball intersects both $E_1$ and $E_2$. Three of the enlarged balls are also drawn in the figure.
\label{fig:U}}
\end{figure}

\begin{lemma}\label{lem:Ucover}
Let $X_1, \ldots, X_d$ be unions of elements of $\mathcal U$ and set $\mathcal{P}=\{X_1,\ldots, X_d\}$. If there is $\delta >0$ with $ [X_i]_\delta  \cap E_i = \emptyset$ for each  $i\in\{1,\ldots, d\}$ and $\bigcap_{i=1}^d [X_i]_\delta = \emptyset$, then there is a path $\alpha_\mathcal{P}$ of finite length $\ell(\alpha_\mathcal{P})$ in $E\setminus \overline{X_1\cup\cdots\cup X_d}$ connecting $p_0$ and $b_0$. 
\end{lemma}

\begin{proof}
Because each set $X_i$, $i\in\{1,\ldots,d\}$, is a union of elements of $\mathcal {U}$,  by the conditions on $\mathcal {U}$ we have $\overline{X_1\cup \cdots \cup X_d}\cap(\{p_0\} \cup E_0)=\emptyset$. 
Thus we may choose $\delta>0$ small enough so that $[X_1\cup \cdots \cup X_d]_{\delta}\cap(\{p_0\} \cup E_0)=\emptyset$. Then there is a nondegenerate simplex $E'=[p'_0,\ldots,p'_d]\subset E^\circ$ with faces $E'_i$, satisfying $\vertii{p_i'-p_i}<\frac\delta2$ ($i\in\{0,\ldots,d\}$), such that $E_i' \cap [X_i]_{\delta/2} = \emptyset$ for each $i\in\{1,\ldots,d\}$, $[X_1\cup \cdots \cup X_d]_{\delta/2}\cap(\{p'_0\} \cup E'_0)=\emptyset$, and such that there exists a  line segment $\beta_1 \subset E\setminus (X_1\cup\cdots\cup X_d)$ connecting $p_0$ and $p_0'$  and a  line segment $\beta_2 \subset E\setminus (X_1\cup\cdots\cup X_d)$ connecting the barycenter $b_0'$ of $E_0'$ and $b_0$. Thus it suffices to construct a path $\alpha'_{\mathcal{P}}$ of finite length $\ell(\alpha'_\mathcal{P})$ in $E\setminus \overline{X_1\cup\cdots\cup X_d}$ connecting $p'_0$ and $b'_0$.

Since $(X_1)_{\delta/2} \cap\cdots\cap (X_d)_{\delta/2}=\emptyset$, Lemma~\ref{lem:KurDim2} implies that $E' \cap ((X_1)_{\delta/2} \cup\cdots\cup (X_d)_{\delta/2})$ does not separate $E'$ between $\{p'_0\}$ and $E'_0$. Thus $E' \cap \overline{X_1\cup\cdots\cup X_d}$ does not separate $E'$ between $\{p'_0\}$ and $E'_0$. Because $E'_0\cap \overline{X_1\cup\cdots\cup X_d} =\emptyset$, the set $E' \cap \overline{X_1\cup\cdots\cup X_d}$ does not separate $E'$ between $\{p'_0\}$ and $\{b'_0\}$  and, {\em a fortiori}, it does not separate $E^\circ$ between $\{p'_0\}$ and $\{b'_0\}$. Now \cite[\S49, IV, Theorem~1]{Kuratowski68} implies that there is a region $R \subset E^\circ\setminus \overline{X_1\cup\cdots\cup X_d}$ (open in $E^\circ$ and, hence, open in $\mathbb{R}^d$) with $\{p'_0,b'_0\}\in R$. Thus $p'_0$ and $b'_0$ can be connected by a path  $\alpha'_\mathcal{P} \subset R$ of finite length (even by a polygonal chain, see~\cite[\S59, I, Theorem~1]{Kuratowski68}). Obviously, $\alpha'_\mathcal{P}$ is contained in $E\setminus \overline{X_1\cup\cdots\cup X_d}$. Thus $\alpha_\mathcal{P} = \beta_1 \cup \alpha'_\mathcal{P}\cup \beta_2$ is the desired path.
\end{proof}

Note that because $\mathcal{U}$ is finite, there are only finitely many choices of collections $\mathcal{P}=\{X_1,\ldots, X_d\}$ that satisfy the assumptions of Lemma~\ref{lem:Ucover}. Let $\mathcal{K}$ be the finite set  of all these collections. Note that $\mathcal K\ne \emptyset$  because we may choose $X_1=\ldots=X_d=\emptyset$. For all that follows we fix a path  $\alpha_\mathcal{P}$ as in Lemma~\ref{lem:Ucover} for each $\mathcal{P} \in \mathcal{K}$. Then
\[
\ell_{\max}:=  \max\{\ell(\alpha_\mathcal{P})\colon \mathcal{P} \in \mathcal{K}\}
\]
is finite. The finiteness of this maximum is important for bounding the lengths of the paths $\alpha$ uniformly in the following lemma.

\begin{lemma}\label{lem:InductionStartNagata1}
Suppose $\Theta \subset E \subset \mathbb{R}^d$ has $\dim_N (\Theta)\le d-2$ with constant $c$, and let $q$ be as in Lemma~\ref{lem:ballsize}. Then there exists a path $\alpha\subset E$ with $\ell(\alpha) \le \ell_{\max}$ connecting $p_0$ with $b_0$ such that
$\Theta \cap \alpha \subset B_{q}(p_0) \cup B_{q}(b_0)$.
\end{lemma}

The assertion of Lemma~\ref{lem:InductionStartNagata1}, is illustrated in Figure~\ref{fig:alpha}.

\begin{proof}
Let $\mathcal U$ and $\varrho$ be as described before \Cref{lem:Ucover} . Because $\Theta$ satisfies $\dim_N (\Theta)\le d-2$ with constant $c$, the set $\Theta \setminus \big(B_{q}(p_0) \cup B_{q}(b_0)\big)$ admits a $\varrho$-cover $\mathcal{C}=\bigcup_{i=0}^{d-2} \mathcal{C}_i$, where $\mathcal{C}_i$ is $c\varrho$-separated for each $i\in \{0,\ldots, d-2\}$.
For each $C \in \mathcal{C}$ write
\[
 U_C := \bigcup_{\substack{U \in \mathcal U\\U \cap C \ne \emptyset}}U.
\]
For $i\in \{0,\ldots,d-2\}$ define
$
\mathcal{V}'_i := \{U_C  \colon  C\in \mathcal{C}_i\},
$
which is $\frac{c\varrho}2$-separated for each $i\in \{0,\ldots, d-2\}$ and set $\mathcal{V}:=\bigcup_{i=0}^{d-2} \mathcal{V}'_i$. Let
\[
\mathcal{V}_1 := \{V\in  \mathcal{V}\colon [V]_{c\varrho/4} \cap E_1 = \emptyset \}
\]
and
for $j\in\{2,\ldots,d\}$ let
\[
\mathcal{V}_j := \{V\in  \mathcal{V}\colon [V]_{c\varrho/4} \cap E_j = \emptyset \}
\setminus (\mathcal{V}_1 \cup\cdots\cup \mathcal{V}_{j-1}).
\]
By construction, each $V\in \mathcal{V}$ has diameter bounded by $(1+\frac c2)\varrho$ and, hence, satisfies $[V]_{c\varrho/8}  \cap E_j = \emptyset$ for some $j\in\{1,\ldots,d\}$. Thus $\mathcal{V}=\mathcal{V}_1 \cup \cdots \cup \mathcal{V}_d$ and, hence, $\mathcal{V}_1 \cup \cdots \cup \mathcal{V}_d$ is a cover of $\Theta \setminus \big(B_{q}(p_0) \cup B_{q}(b_0)\big)$. 
Define
\[
X_j:=\bigcup_{V\in\mathcal{V}_j}V, \qquad j\in \{1,\ldots, d\},
\]
and set $\delta=\frac{c\varrho}8$. Assume that there is $x \in \bigcap_{j=1}^d [X_j]_\delta$. Then there exist pairwise distinct $V_1,\ldots, V_d \in \mathcal V$ so that $x \in [V_j]_\delta$ for all $j\in\{1,\ldots,d\}$. Since there are $d$ of these sets, two of them are from the same collection $\mathcal V_i'$ for some $i\in\{0,\ldots,d-2\},$ a contradiction, because $\mathcal V_i'$ is $4\delta$-separated. Thus, $\bigcap_{j=1}^d [X_j]_\delta = \emptyset$ and so the sets $X_i$
satisfy the conditions of Lemma~\ref{lem:Ucover}. This lemma exhibits a path  $\alpha_\mathcal{P}$, $\mathcal{P}=\{X_1,\ldots, X_d\}$, with $\ell(\alpha_\mathcal{P}) \le \ell_{\max}$ that is contained in $E\setminus \overline{X_1\cup\cdots\cup X_d}$ and connects $p_0$ and $b_0$. Because $\overline{X_1\cup\cdots\cup X_d} \supset \Theta \setminus \big(B_{q}(p_0) \cup B_{q}(b_0)\big)$, the path  $\alpha:=\alpha_\mathcal{P}$ has the properties asserted in the statement of the lemma.
\end{proof}

\begin{figure}[h]
\includegraphics[width=0.8\textwidth]{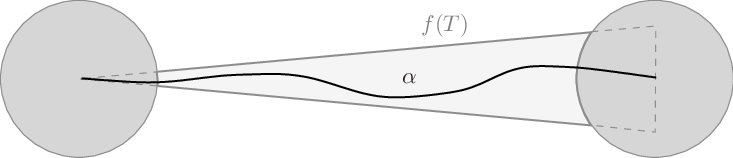}
\caption{Together with Lemma~\ref{lem:NagataSimilar}, Lemma~\ref{lem:InductionStartNagata1} implies that for each similarity transformation $f\colon\R^d\to\R^d$ we can find a path $\alpha\subset f(E)$ that does not intersect $\Theta$ outside the shaded disks $f(B_{q}(p_0) \cup B_{q}(b_0))$, {\it i.e.}, for $\alpha$ we have $(\alpha \cap f(T)) \cap \Theta=\emptyset$. \label{fig:alpha}}
\end{figure}

Set $a_1= (-1,0,\ldots,0)$ and $a_2=(1,0,\ldots,0)$. 
Let $\Theta\subset \mathbb{R}^d$ have $\dim_N(\Theta) \le d-2$ with constant $c$. By Lemma~\ref{lem:ballsize}, $\Theta$ is $q$-porous with $q$ depending only on $d$ and $c$, so for every $i\in\{1,2\}$ there exists a sequence $(s_{i,n})_{n\ge 0}$ with $\vertii{s_{i,n} - a_i} \le 2^{-n}$ such that  
\begin{equation}\label{eq:disjointAdisks}
B_{2^{-n}q}(s_{i,n}) \cap \Theta = \emptyset \qquad(n\ge 0).  
\end{equation}
This implies that
\begin{equation}\label{eq:bSequences}
\vertii{s_{1,0} - s_{2,0}} \le 4 \qquad\text{and}\qquad
\vertii{s_{i,n} - s_{i,n+1}} \le 2^{1-n} \quad (i\in\{1,2\},\, n\ge 0).
\end{equation}
We illustrate this sequence in Figure~\ref{fig:arc1}.\footnote{If $\Theta$ is disjoint from $\{a_1,a_2\}$ (which we will assume later) and closed, we could define $(s_{i,n})_{n\ge 0}$ in a way that there exists $n_0$ such that $s_{i,n}=a_i$ holds for $n\ge n_0>0$ as in the figure ($i\in\{1,2\}$). However, this observation is not crucial for the current proof.}
\begin{figure}[h]
\includegraphics[width=0.9\textwidth]{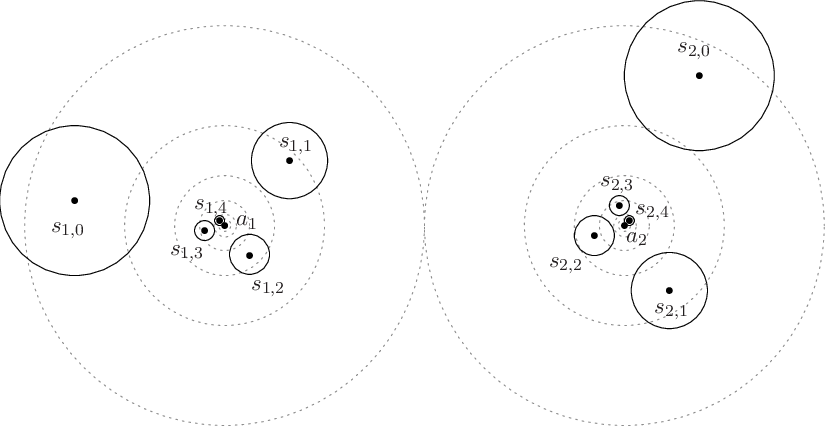}
\caption{The sequences $(s_{i,n})_{n\ge 1}$ together with the circles $B_{2^{-n}q}(s_{i,n})$ for $i\in\{1,2\}$ and small $n$.
\label{fig:arc1}}
\end{figure}

For a  similarity transformation $f\colon \R^d\to \R^d$ we call the unique number $r>0$ satisfying $\vertii{f(x)-f(y)} = r\vertii{x-y}$ for all $x,y\in \R^d$ the {\em similarity ratio} of $f$ and denote it by $\operatorname{sr}(f)$.

\begin{lemma}\label{lem:subpaths}
Let $\Theta\subset \mathbb{R}^d$ with $\dim_N(\Theta) \le d-2$ with constant $c$ and for $i\in\{1,2\}$ define the sequences $(s_{i,n})_{n\ge 0}$ as above. Then the following assertions hold.
\begin{itemize}
\item[(1)] For each $i\in \{1,2\}$ and each $n \ge 0$ there is a path $\alpha_{i,n} \subset \mathbb{R}^d\setminus \Theta$ connecting $s_{i,n}$ and $s_{i,n+1}$ with $\ell(\alpha_{i,n}) \le 2^{-n-1}\ell_{\max}$.
\item[(2)] There exists a path $\beta\subset \mathbb{R}^d\setminus \Theta$ connecting $s_{1,0}$ and $s_{2,0}$ with $\ell(\beta) \le \ell_{\max}$.
\end{itemize}
\end{lemma}

\begin{proof}
We start with the proof of (1). Fix $i\in\{1,2\}$ and $n\ge 0$. We may assume that $s_{i,n}\ne s_{i,{n+1}}$ because otherwise there is nothing to prove. Fix a similarity transformation $f_{i,n}\colon\mathbb{R}^d \to \mathbb{R}^d$ that maps $p_0$ to $s_{i,n}$ and $b_0$ to $s_{i,{n+1}}$. By Lemma~\ref{lem:NagataSimilar}, $f_{i,n}^{-1}(\Theta)$ satisfies $\dim_N (f_{i,n}^{-1}(\Theta)) \le d-2$ with constant $c$. Moreover, \eqref{eq:disjointAdisks} and \eqref{eq:bSequences} imply that $\operatorname{sr}(f_{i,n}) \le 2^{-n-1}$ and
\begin{equation}\label{eq:fminus1empty}
f_{i,n}^{-1}(\Theta) \cap B_q(p_0)=\emptyset \quad\hbox{and}\quad f_{i,n}^{-1}(\Theta) \cap B_q(b_0)=\emptyset.
\end{equation}
Thus the set $f_{i,n}^{-1}(\Theta)\cap E$ satisfies the conditions of Lemma~\ref{lem:InductionStartNagata1} and we obtain a path $\alpha\subset E$ connecting $p_0$ and $b_0$ of length $\ell(\alpha) \le \ell_{\max}$ with $f_{i,n}^{-1}(\Theta) \cap \alpha \subset B_{q}(p_0) \cup B_{q}(b_0)$. From \eqref{eq:fminus1empty} we even obtain that $f_{i,n}^{-1}(\Theta)\cap \alpha = \emptyset$. Thus $\alpha_{i,n}=f_{i,n}(\alpha)$ is a path with the desired properties.
 
To prove (2) we may assume that $s_{1,0}\ne s_{2,0}$ because otherwise there is nothing to prove. Choose a similarity transformation $f\colon\mathbb{R}^d \to \mathbb{R}^d$ mapping $p_0$ to $s_{1,0}$ and $b_0$ to $s_{2,0}$ and observe that the similarity ratio of $f$ is bounded by $1$ because of~\eqref{eq:bSequences}. Then proceed as in the proof of~(1).
\end{proof}

\begin{lemma}\label{lem:avoidA1}
Let $\Theta\subset \mathbb{R}^d \setminus\{(\pm 1,0,\ldots,0)\}$ with $\dim_N(\Theta) \le d-2$ with constant $c$. Then there exists a path $\alpha \subset \mathbb{R}^d\setminus \Theta$,  connecting $a_1=(-1,0,\ldots,0)$ and $a_2=(1,0,\ldots,0)$, with length $\ell(\alpha)\le 3\ell_{\max}$.
\end{lemma}
The construction of the path  $\alpha$ is illustrated in Figure~\ref{fig:arc2}.
\begin{figure}[h]
\includegraphics[width=0.9\textwidth]{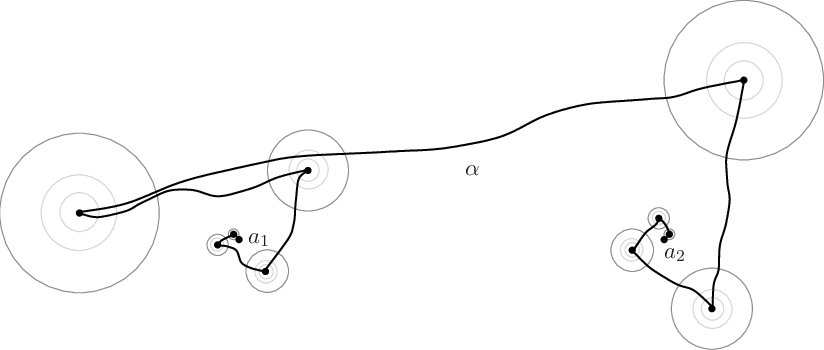}
\caption{
An illustration of the path $\alpha$ constructed in Lemma~\ref{lem:avoidA1}. The subpaths are given by Lemma~\ref{lem:subpaths}. They avoid the set $\Theta$ near the endpoints due to porosity, while further away from the endpoints the avoidance follows from Lemma~\ref{lem:InductionStartNagata1}.
\label{fig:arc2}}
\end{figure}

\begin{proof}
This is an immediate consequence of Lemma~\ref{lem:subpaths}. Just concatenate the paths 
\[
\ldots, \alpha_{1,3},\alpha_{1,2},\alpha_{1,1},\alpha_{1,0},\beta,\alpha_{2,0},\alpha_{2,1},\alpha_{2,2},\alpha_{2,3},\ldots
\] 
(with the  orientations appropriately adapted)
and take the closure to obtain the desired path~$\alpha$.
\end{proof}

We can now finish the proof of our main theorem.

\begin{proof}[Proof of \Cref{th:Nagata}]
Let $x,y\in \R^d\setminus \Theta$. If $x=y$ there is nothing to prove. Thus assume $x\ne y$. There exists a 
similarity transform $f$ which maps $x$ to $(1,0,\ldots,0)$ and $y$ to $(-1,0,\ldots,0)$. Hence,
$\operatorname{sr}(f)=\frac{2}{\|y-x\|}$.
By \Cref{lem:NagataSimilar} and \Cref{lem:avoidA1} there exists a path $\alpha\subset \R^d\setminus f(\Theta)$ connecting $(1,0,\ldots,0)$ and $(-1,0,\ldots,0)$ with $\ell(\alpha)\le 3 \ell_{\max}$. Therefore
$f^{-1}\circ \alpha$ is a path in $\R^d\setminus \Theta$ connecting $x$ and $y$ with $\ell( f^{-1}\circ \alpha)
=\operatorname{sr}(f^{-1})\ell(\alpha)=\frac{\|y-x\|}{2}3\ell_{\max}$. 
Thus $\R^d\setminus \Theta$ is $\tfrac{3}{2}\ell_{\max}$-quasi convex. Since $\Theta$ is a Lebesgue 
nullset by \Cref{pro:LebesgueNagata}, the result now follows from \Cref{prop:c-quasiconvex=>nullperm}.
\end{proof}

\subsection{Lipschitz dimension}
Let $d\ge 2$. In \cite{David21}, the integer valued \emph{Lipschitz dimension} $\dim_{\rm Lip}$ of Cheeger-Kleiner is treated. In \cite[Corollary 3.5]{David21} it is shown that $\dim_N(\Theta) \leq\dim_{\rm Lip}(\Theta)$ holds for each $\Theta\subset \R^d$. This immediately yields that if for $\Theta\subset\R^d$ we have $\dim_{\rm Lip}(\Theta)<d-1$, then $\Theta$ is null permeable. An example for an impermeable set with $\dim_{\rm Lip}(\Theta)=d-1$ is furnished by $\Theta=C\times[0,1]^{d-1}$ with $C$ being the middle third Cantor set. Indeed, from \cite[Proposition 3.6]{David21} it follows that $\dim_{\rm Lip}(C)=0$ and it is easy to see by definition (\cite[Definition 1.3]{David21}) that $\dim_{\rm Lip}([0,1]^{d-1})=d-1$. By \cite[Proposition 4.2]{David21}, we have that $\dim_{\rm Lip}(C\times[0,1]^{d-1})\leq d-1$ and since $\dim_N(C\times[0,1]^{d-1})=d-1$, we conclude that $\dim_{\rm Lip}(C\times[0,1]^{d-1})= d-1$. An example for a permeable set with Lipschitz dimension $d$ is given by $\Z^d$.

\subsection{Topological dimension}
The {topological dimension} of a subset of $\R^d$ is defined for instance in \cite[\S25, I]{Kuratowski66}, and we denote it by $\dim_T$. Concerning permeability, it behaves differently from the other notions of dimension. Sets of dimension 0 may be impermeable, take for example the product $F^d\subset\R^d$, where $F$ is the Smith-Volterra-Cantor set of \Cref{ex:closedneeded}. On the other hand, a hyperplane $H$ in $\R^d$ is a finitely permeable set with $\dim_T(H)=d-1$, and any linear subspace of topological dimension $d-2$ is null permeable. However, a set $\Theta\subset\R^d$ with $\dim_T(\Theta)=d$ has nonempty interior by the Menger-Urysohn Theorem (see {\it e.g.}~\cite[Theorem~1.8.10]{Engelking}) and, hence, it is impermeable. Summing up, it is easy to get the following result.

\begin{proposition}\label{cor:Topo}
Let $d \ge 1$. 
\begin{itemize}
\item[(1)] For each $s\in \{0,\ldots,d\}$ there exists an impermeable set $\Theta\subset \R^d$ with $\dim_T(\Theta)=s$.
\item[(2)] For each $s\in \{0,\ldots,d-1\}$ there exists a finitely permeable set $\Theta\subset \R^d$ with $\dim_T(\Theta)=s$.
\item[(3)] For each $s\in \{0,\ldots,d-2\}$ there exists a null permeable set $\Theta\subset \R^d$ with $\dim_T(\Theta)=s$. 
\item[(4)] Each $\Theta\subset \R^d$ with $\dim_T(\Theta)= d$  is impermeable.
\end{itemize}
\end{proposition}

For $d=2$, item (3) is best possible, {\it i.e.}, there is no null permeable set of topological dimension $d-1=1$. Indeed, let $\Theta\subset \R^2$ be null permeable, $x\in \Theta$, and $r>0$. There exist points $w_1,w_2,w_3\in \R^2$ such that  $\overline{w_1w_2w_3w_1}\subset B_{r}(x)$, and such that $\operatorname{wind}(\overline{w_1w_2w_3w_1},x)=1$, where $\operatorname{wind}(\kappa,z)$ denotes the winding number of a closed path $\kappa\subset \R^2$ relative to $z\in \R^2$. 
Since $\Theta$ has no interior points, for every $i\in\{1,\dotsc,4\}$ we can shift the 
corner point $w_i$ by a sufficiently small distance, to find points $v_i$ in the complement of $\Theta$, such that still $\overline{v_1v_2v_3v_1}\subset B_{r}(x)$ and $\operatorname{wind}(\overline{ v_1v_2v_3v_1},x)=1$. 
Since $\Theta$ is null permeable we can find paths $\gamma_1,\gamma_2,\gamma_3$ of finite length connecting $v_1$ and $v_2$, $v_2$ and $v_3$, and $v_3$ and $v_1$, respectively.  If we require the paths $\gamma_1,\gamma_2,\gamma_3$ to be sufficiently short, as we may by null permeability of $\Theta$, the   concatenation $\gamma$ of $\gamma_1,\gamma_2,\gamma_3$ satisfies $\gamma\subset B_{r}(x)$, and $\operatorname{wind}(\gamma,x)=1$. By \cite[\S 61, II, Theorem 5]{Kuratowski68}, we may assume that $\gamma$ is a loop. By the Jordan Curve Theorem there exist disjoint open sets $A_1,A_2$ with $\R^2=A_1\cup A_2\cup \gamma$ and $x\in A_1$. By construction, $x\in A_1\subset B_{r}(x)$. Since $\gamma\cap \Theta=\emptyset$, $A_1$ is a neighborhood of $x$ whose boundary has empty intersection with $\Theta$.

We have shown that, for every $x\in \Theta$ and $r>0$, there exists an open set $A_1\subset B_r(x)$ such that $A_1\cap \Theta$ is a neighborhood of $x$ w.r.t.~to the subspace topology on $\Theta$ having empty boundary relative to this topology. Since $r$ --- and thus also $A_1\cap\Theta$ --- can be made arbitrarily small, $\dim_T(\Theta)=0$ by definition.

If $d\ge 3$ and $\Theta\subset \R^d$ is closed, then also the optimality of (3) still holds, {\it i.e.,}~$\Theta$ cannot be null permeable if $\dim_T(\Theta)=d-1$:
Assume $\dim_T(\Theta)=d-1$. Then by \cite{Frankl:1930ab} (for $d=3$) and \cite{Frankl:1930aa} (for $d>3$), 
$\Theta$ does not contain a region, but separates a region. 
Let $G$ be such a region separated by $\Theta$ and let $x\in \Theta\cap G$. Then there exists $r>0$ such that $\Theta$ separates $B_r(x)$, i.e., $B_r(x)$ is the disjoint union of $A_1, A_2,  \Theta\cap B_r(x)$, where $A_1,A_2$ are open sets. Choose $x_1\in A_1$, $x_2\in A_2$. Then 
every path  $\gamma\in B_r(x)$ connecting $x_1$ and $x_2$ has nonempty intersection with $\Theta$. 
Every path  $\gamma$ connecting $x_1$ and $x_2$ and with $\gamma\cap (\R^d\setminus B_r(x))\ne \emptyset$ satisfies $\ell(\gamma)>\vertii{x_2-x_1}+c$ for some constant $c$. Thus $\Theta$ is not null permeable.

For $d\geq 3$ it is unclear if there exists a non-closed null permeable set $\Theta\subset\R^d$ with $\dim_T(\Theta)=d-1$. The example in Sitnikov~\cite[Section~3.3]{Sitnikov} indicates that this might be a difficult question.

\subsection{Synopsis}\label{sec:synopsis}
The interplay between the various notions of dimension we treated in this section and their relation to permeability is summarized concisely in \Cref{fig:dimensions}.
\begin{figure}[h]
\begin{center}
\includegraphics[width=0.7\textwidth]{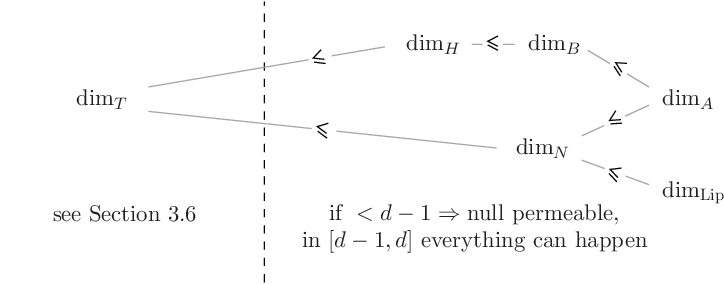}
\end{center}
\caption{Relations between different notions of dimension and permeability}\label{fig:dimensions}
\end{figure}

\section{Permeability and self-similar sets}\label{sec:selfsimilarsets}
We now turn to permeability results for self-similar (and, in the last part of this section, self-affine) sets. It turns out that the well-known {\em finite type condition} ({\em cf.}~\cite[Definition~8]{BandtRao:07} and see \cite{NgaiWang01} for a predecessor of this condition) is crucial for permeability of self-similar sets. This section is organized as follows. In Section~\ref{sec:ifsdef} we recall the terminology on iterated function systems that is relevant for us. Moreover, we prove that in the case of self-similar sets, the finite type condition is equivalent to the seemingly stronger condition of {\em finite types of neighborhoods}. This result is needed in the subsequent sections. 
\Cref{sec:selfsimilar2d} gives a criterion for a two-dimensional self-similar  set to be permeable. In dimensions greater than 2, we can use Theorem~\ref{th:Nagata} in order to give a criterion for a self-similar  set to be null permeable. This will be established in \Cref{sec:3difs}. Finally, \Cref{sec:bedford-mcmullen} is devoted to the impermeability of certain self-affine Bedford-McMullen carpets. Self-affine sets show a completely different behavior w.r.t.\ their permeability properties. This becomes apparent from the examples exhibited in~\cite{BuLeSt22}.

\subsection{Iterated function systems and self-similar sets}\label{sec:ifsdef}
Recall that $\{f_1,\ldots, f_m\}$ is an {\em iterated function system} (IFS, for short) in $\R^d$, if, for each $j\in\{1,\ldots,m\}$, $f_j\colon \R^d\to \R^d$ is a contraction having {\em contraction ratio} $r_j\in (0,1)$. It is known from Hutchinson~\cite{Hutchinson:81} that there exists a unique nonempty compact set $K$ satisfying
 \[
 K=\bigcup_{j=1}^m f_{j}(K).
 \]  
This set $K$ is called the {\em attractor} of the iterated function system $\{f_1,\ldots, f_m\}$. If the mappings $f_1,\ldots,f_m$ are similarity transformations, we call the iterated function system $\{f_1,\ldots,f_m\}$ {\em self-similar}. In this case $K$ is called a {\em self-similar set}. The definition of a {\em self-affine set} is analogous if the mappings $f_1,\ldots,f_m$ are affine transformations.

We introduce some notation. For a nonempty set $M$ we denote by $M^*$ and $M^\N$ the set of finite and infinite sequences of elements of $M$, respectively. 
For $k\in\N$ and $\mathbf{i}=i_1\cdots i_k\in M^*$ we write $|\mathbf{i}|:=k$ for the {\em length} of $\mathbf{i}$ and we write $M^{\le k}=\{\mathbf{i}\in M^*\colon |\mathbf{i}|\le k\}$. Let $\mathcal{F}=\{f_1,\ldots,f_m\}$ be a self-similar IFS in $\mathbb{R}^d$ with similarity ratios $r_i:=\operatorname{sr}(f_i)$ for $i\in\{1,\ldots,m\}$.   Set $r_{\min}:=\min\{r_1,\ldots, r_m\}$ and $r_{\max}:=\max\{r_1,\ldots, r_m\}$.
For $\mathbf{i}=i_1\cdots i_k\in\{1,\ldots, m\}^*$ let $f_{\mathbf{i}}=f_{i_1} \circ \cdots \circ f_{i_k}$. For $\mathbf{i}=\epsilon$, the empty word, this means that $f_\epsilon=\operatorname{id}$. 
In the sequel, we will deal with attractors $K$ satisfying  $\#f_i(K) \cap f_j(K) < \infty$ for $1\le i<j\le m$. 

\begin{definition}[Open set condition]
We say that an iterated function system $\mathcal{F}:=\{f_1,\ldots, f_m\}$ in $\R^d$ satisfies the {\em open set condition} if there exists a nonempty open set $U\subset \R^d$ such that $\bigcup_{i=1}^m f_i(U)\subset U$ and $f_i(U)\cap f_j(U) = \emptyset$ for $i\ne j$.
\end{definition}
 
We now recall the {\em finite type condition} of \cite[Definiton~8]{BandtRao:07} together with a stronger version of it. We use the notation $[Y]_\varepsilon$ from \eqref{eq:Aeps}.

\begin{definition}[Neighbor maps, finite types of neighborhoods, and finite type condition]\label{def:finiteneighborhoods}
Let $K$ be the attractor of a self-similar IFS $\mathcal{F}=\{f_1,\ldots,f_m\}$. 
\begin{enumerate}
\item A {\em neighbor map} for $\mathcal{F}$ is a map $h\colon \R^d\to \R^d$ of the form $h=f_{\mathbf{i}}^{-1} \circ f_{\mathbf{j}}$
with $h\ne \operatorname{id}$. For every $\varepsilon \in [0,\infty)$ denote by  $N_\varepsilon(\mathcal{F})$ the set of neighbor maps $h$ with $[K]_\varepsilon \cap h([K]_\varepsilon) \ne \emptyset$ and $\operatorname{sr}(h) \in [r_{\min},r_{\min}^{-1}]$;
\item $\mathcal{F}$ has {\em finite types of neighborhoods for $\varepsilon >0$} if  $N_\varepsilon(\mathcal{F})$  is finite;
\item $\mathcal{F}$ has {\em finite types of neighborhoods} if there exists $\varepsilon>0$ such that 
$\mathcal{F}$ has finite types of neighborhoods for $\varepsilon$;
\item $\mathcal{F}$ satisfies the {\em finite type condition} if $N_0(\mathcal{F})$ is finite ({\em cf.}~\cite[Definiton~8]{BandtRao:07}).
\end{enumerate}
\end{definition}

Set
\[
H(\mathcal{F}) := \{ x \in \mathbb{R}^d \colon x \in   K \cap h(K) \text{ for some }h\in N_0(\mathcal{F})  \}=\bigcup_{h\in N_0(\mathcal{F})} K \cap h(K),
\]
which we call the set of {\em intersection points} of $\mathcal{F}$. It is immediate that $H(\mathcal{F})=\emptyset$ if and only if $N_0(\mathcal{F})=\emptyset$.
Note that if $0<\varepsilon_1<\varepsilon_2$, then $N_{\varepsilon_1}(\mathcal{F})\subset N_{\varepsilon_2}(\mathcal{F})$, so it makes sense to define
\[
N(\mathcal{F}):=\lim_{\varepsilon\to 0+}N_{\varepsilon}(\mathcal{F})=\bigcap_{\varepsilon\in (0,\infty)}N_{\varepsilon}(\mathcal{F})  .
\]

\begin{theorem}\label{thm:finite-type-finite-types}
Let $\mathcal{F}:=\{f_1,\ldots,f_m\}$ be a self-similar IFS. Then $\mathcal{F}$ has finite types of neighborhoods if and only if it satisfies the finite type condition.
\end{theorem}

\begin{proof} 
For $\#K=1$, the assertion is immediate, so we assume $\#K>1$.
We only prove necessity, sufficiency being trivial. Thus we assume that the finite type condition holds, {\it i.e.}, $N_0(\mathcal{F})$ is finite. For convenience, let $\mathcal{M}:=\{1,\ldots, m\}$ and $C:=\big\lceil\frac{\log (r_{\min})}{\log (r_{\max})}\big\rceil$. Define
\[
\begin{split}
R&:= \{ g= f_{i}^{-1}\circ h \circ f_{\mathbf{j}} \colon
h\in N_0(\mathcal{F})\cup\{\mathrm{id}\},\, i\in \mathcal{M}, \mathbf{j}\in \mathcal{M}^{\le C}, 
\operatorname{sr}(g)\in[r_{\min},r_{\min}^{-1}], K\cap g(K)=\emptyset
\}, \\
R(\mathbf{i})&:=\{
g=f_{\mathbf{i}}^{-1}\circ f_{\mathbf{j}} \colon\mathbf{j}\in \mathcal{M}^*,\operatorname{sr}(g)\in[r_{\min},r_{\min}^{-1}], K\cap g(K)=\emptyset
\} \quad (\mathbf{i}\in \mathcal{M}^*).
\end{split}
\]
These sets are finite, since $N_0(\mathcal{F})$ is finite and since there are only finitely many contractions in~$\mathcal{F}$. Because $\#K>1$, the minimum
\[
n_1:=\min\{n\in \N\colon \text{there exists } \mathbf{i}\in \mathcal{M}^{\le n} \text{ with } R(\mathbf{i})\neq \emptyset\}
\]
is finite. Thus
\[
e_n:=\sup\{\varepsilon>0\colon  [K]_\varepsilon\cap h([K]_\varepsilon)=\emptyset
\text{ for each } \mathbf{i}\in \mathcal{M}^{\le n} \text{ and each } h\in R(\mathbf{i}) \} \quad (n\ge n_1)
\]
are all finite. Moreover, let
\[
e:=\sup\{\varepsilon>0\colon [K]_\varepsilon \cap g([K]_\varepsilon)=\emptyset, \,\text{for all } g\in R\}, 
\]
and note that $e=\infty$ if $R=\emptyset$.
Since $K$ is compact and $R$ as well as $R(\mathbf{i})$ are finite, we have $e_n>0$ for $n\ge n_1$ and 
$e\in (0,\infty]$. Moreover, the sequence $(e_n)_{n\geq n_1}$ is nonincreasing. 

The result will follow if we show that each element of the sequence $(e_n)_{n\geq n_1}$ is greater than or equal to $\min\{e,e_{n_1}\}$. To establish this claim it suffices to show that $e_{n+1}<e_n$ implies $e_{n+1}\ge e$. To prove this, let $n\geq n_1$ with $e_{n+1}<e_n$. Since $K$ is compact and  $\bigcup_{|\mathbf{i}|\le n+1}R(\mathbf{i})$ is finite, there exist $\mathbf{i}\in  \mathcal{M}^{\le n+1}$ and $\mathbf{j}\in  \mathcal{M}^*$ with $\operatorname{sr}(f_{\mathbf{i}}^{-1}\circ f_{\mathbf{j}})\in[r_{\min},r_{\min}^{-1}]$ and
$[K]_{e_{n+1}}\cap (f_{\mathbf{i}}^{-1}\circ f_{\mathbf{j}})([K]_{e_{n+1}})\ne \emptyset$ but $[K]_\varepsilon\cap (f_{\mathbf{i}}^{-1}\circ f_{\mathbf{j}})([K]_\varepsilon)=\emptyset$ for all $0\le\varepsilon<e_{n+1}$. Since $e_{n+1}<e_n$, we need to have $|\mathbf{i}|=n+1$ and $|\mathbf{j}|\ge n+1$ (if $|\mathbf{j}|\le n$, one can swap $\mathbf{i}$ and $\mathbf{j}$ to get a contradiction). 
Let $\mathbf{i}'$ be obtained by cutting off the last index from $ \mathbf{i}$. 
There exists $k\in\{0,\ldots,C\}$ such that if  $\mathbf{j}'$ is obtained from $ \mathbf{j}$ by cutting off the last $k$ indices, then $\operatorname{sr}(f_{\mathbf{i}'}^{-1}\circ f_{\mathbf{j}'})\in [r_{\min},r_{\min}^{-1}]$. 
We show that $K\cap \big(f_{\mathbf{i}'}^{-1}\circ f_{\mathbf{j}'}\big)(K)\ne \emptyset$: 
Otherwise $f_{\mathbf{i}'}^{-1}\circ f_{\mathbf{j}'}\in R(\mathbf{i}')$, and if we let $e_{n+1}<\varepsilon<e_n$,
then
\begin{equation}\label{eq:similarity-necessary}
f_{\mathbf{i}}([K]_\varepsilon)\subset f_{\mathbf{i}'}([K]_\varepsilon)\quad\text{ and }\quad f_{\mathbf{j}}([K]_\varepsilon)\subset f_{\mathbf{j}'}([K]_\varepsilon),
\end{equation}
since all occurring maps are similarity transformations. Hence, 
\[\emptyset\ne f_{\mathbf{i}}([K]_\varepsilon)\cap f_{\mathbf{j}}([K]_\varepsilon) \subset f_{\mathbf{i}'}([K]_\varepsilon)\cap f_{\mathbf{j}'}([K]_\varepsilon). \]
But this is impossible by  the definition of $e_n$.
Now $K\cap \big(f_{\mathbf{i}'}^{-1}\circ f_{\mathbf{j}'}\big)(K)\ne \emptyset$ with $\operatorname{sr}(f_{\mathbf{i}'}^{-1}\circ f_{\mathbf{j}'})\in [r_{\min},r_{\min}^{-1}]$ means that $f_{\mathbf{i}'}^{-1}\circ f_{\mathbf{j}'}\in N_0(\mathcal{F})\cup\{\mathrm{id}\}$. 
 Thus $f_{\mathbf{i}}^{-1}\circ f_{\mathbf{j}}=f_{i_{n+1}}^{-1}\circ h \circ f_{\tilde {\mathbf{j}}}$ for $h\in  N_0(\mathcal{F})\cup\{\mathrm{id}\}$, $i_{n+1}\in \mathcal{M}$ and $\tilde{\mathbf{j}}\in\mathcal{M}^{\le C}$, in other words, $f_{\mathbf{i}}^{-1}\circ f_{\mathbf{j}}\in R$. 
From $f_{\mathbf{i}}^{-1}\circ f_{\mathbf{j}}\in R$ we gain $e_{n+1}\ge e$ and the claim is proved. Thus $\mathcal{F}$ has finite types of neighborhoods for every $\varepsilon \in (0,\min\{e,e_{n_1}\})$.
\end{proof}

\begin{remark}\mbox{}
\begin{enumerate}
\item In the preceding proof we obtained, as an intermediate result, that $e_n>e_{n+1}$ implies $R\ne \emptyset$. 
Conversely, if $R=\emptyset$, then $e_n=e_{n_1}$ for every $n\ge n_1$.

\item The inclusions in \eqref{eq:similarity-necessary} only hold for self-similar IFS. We believe that the conclusion of \Cref{thm:finite-type-finite-types} does not hold if ``self-similar'' is replaced by ``self-affine''.
\item  Note that \Cref{thm:finite-type-finite-types} implies $N_0(\mathcal{F})=N(\mathcal{F})$.
\end{enumerate}
\end{remark}

\begin{lemma}\label{lemma:HK-finite}
Let $K\subset \R^d$ be the attractor of a self-similar IFS $\mathcal{F}:=\{f_1,\ldots, f_m\}$ satisfying
\begin{enumerate}
\item   
$f_i(K) \cap f_j(K) < \infty$ for $1\le i<j\le m$;
\item  $\mathcal{F}$ is of finite type.
\end{enumerate}
Then $N(\mathcal{F})$ and $H(\mathcal{F})$ are both finite sets and there exists $\varepsilon\in (0,\infty)$ with  $N(\mathcal{F})=N_{\varepsilon}(\mathcal{F})$.
\end{lemma}

\begin{proof}
By \Cref{thm:finite-type-finite-types}, $\mathcal{F}$ has finite types of neighborhoods for some $\varepsilon_0 > 0$. 
The set $H(\mathcal{F})$ is finite because $\# N_{0}(\mathcal{F})<\infty$ and $\# K \cap h(K)<\infty$ for each $h\in N_{0}(\mathcal{F})$.

Note further that, since $N_\varepsilon(\mathcal{F})$ is a finite set for every $\varepsilon\in (0,\varepsilon_0)$, there exists $\varepsilon\in (0,\infty)$ such that for all $\varepsilon'\in (0,\varepsilon]$ we have $N_{\varepsilon'}(\mathcal{F})=N_{\varepsilon}(\mathcal{F})$, and thus $N(\mathcal{F})=\lim_{\varepsilon'\to 0+}N_{\varepsilon'}(\mathcal{F})=N_{\varepsilon}(\mathcal{F})$. 
\end{proof}

\subsection{Two-dimensional self-similar sets}\label{sec:selfsimilar2d}
In this section we prove the following theorem, which provides a criterion of permeability of planar self-similar sets in terms of the finite type condition from Definition~\ref{def:finiteneighborhoods}.

\begin{theorem}\label{thm:ftn-null-permeable2d}
Let $K\subset \R^2$ be the attractor of a self-similar IFS $\mathcal{F}:=\{f_1,\ldots, f_m\}$ satisfying $\# f_i(K) \cap f_j(K) < \infty$ for $1\le i<j\le m$. If $K$ is connected and $\mathcal{F}$ satisfies the finite type condition, then $K$ is permeable.
\end{theorem}

The proof of this result will be subdivided into lemmata.  

\begin{lemma}\label{prop:ftc-empty-int}
For $d\ge 2$ let $K\subset \R^d$ be the attractor of a self-similar IFS $\mathcal{F}:=\{f_1,\ldots, f_m\}$ satisfying $\#f_i(K) \cap f_j(K) < \infty$ for all $1\le i<j\le m$. Then $K$ has no interior points.
\end{lemma}

\begin{proof} The conclusion trivially holds if $K$ is a finite set. Thus assume that $K$ is infinite.  Suppose there exists  $x\in K$ and $\varepsilon>0$ such that $B_\varepsilon(x)\subset K$. There exists $\mathbf{i}\in \{1,\ldots,m\}^*$ with $f_{\mathbf{i}}(K)\subsetneq B_\varepsilon(x)$. Let $J=\{1,\ldots,m\}^{|\mathbf{i}|}\setminus\{\mathbf{i}\}$. By assumption, for every $\mathbf{j}\in J$ we have that 
 $f_{\mathbf{i}}(K)\cap f_{\mathbf{j}}(K)$ is finite.  Then
  \begin{align*}
(f_{\mathbf{i}}(K)\cap B_\varepsilon(x)) \cup \bigg(\bigcup_{\mathbf{j}\in J}f_{\mathbf{j}}(K)\cap B_\varepsilon(x)\bigg)&=B_\varepsilon(x),
\end{align*}
but $F:=\big(f_{\mathbf{i}}(K)\cap B_\varepsilon(x)\big) \cap \bigcup_{\mathbf{j}\in J}f_{\mathbf{j}}(K)\cap B_\varepsilon(x)$ is finite,
 thus $B_\varepsilon(x)\setminus F=\big(B_\varepsilon(x)\setminus f_{\mathbf{i}}(K)\big) \cup \Big(B_\varepsilon(x)\setminus  \bigcup_{\mathbf{j}\in J}f_{\mathbf{j}}(K)\Big)$, that is, we have written $B_\varepsilon(x)\setminus F$ as the disjoint union of 
 two nonempty open sets. But if $d\ge2$, $B_\varepsilon(x)\setminus F$ is clearly path connected and therefore connected, yielding a contradiction. 
 \end{proof}

\begin{lemma}\label{lemma:different-indices}
Let $K\subset \R^d$ be a nondegenerate continuum which is the attractor of a self-similar IFS $\mathcal{F}:=\{f_1,\ldots, f_m\}$. If $\# f_i(K) \cap f_j(K) < \infty$ for $1\le i<j\le m$, then for any two distinct $\mathbf{i},\mathbf{j}\in\{1,\ldots,m\}^*$ we have $f_\mathbf{i}\ne f_\mathbf{j}$.
\end{lemma}

\begin{proof}
Assume, on the contrary, that $f_\mathbf{i} = f_\mathbf{j}$. We may assume, w.l.o.g., that $\mathbf{i} =i_1\cdots i_k$ and $\mathbf{j}=j_1\cdots j_n$ with $i_1 \ne  j_1$. Then,
\[
f_{i_1}(K) \cap f_{j_1}(K) \supseteq f_\mathbf{i}(K)\cap f_\mathbf{j}(K) = f_\mathbf{i}(K) 
\]
is uncountable, because $K$ is uncountable, and the similarity transformation $f_\mathbf{i}$ is a bijection. This contradicts the finiteness of the intersections.
\end{proof}

Now the following lemma is an easy consequence of \Cref{lemma:different-indices} and \cite[Theorem~9]{BandtRao:07}.

\begin{lemma}\label{lem:osc}
Let $K\subset \R^d$ be a nondegenerate continuum which is the attractor of a self-similar IFS $\mathcal{F}:=\{f_1,\ldots, f_m\}$. If $\# f_i(K) \cap f_j(K) < \infty$ for $1\le i<j\le m$ and $\mathcal{F}$ satisfies the finite type condition, then $\mathcal{F}$ satisfies the open set condition.
\end{lemma}

Let $\varrho \in(0,1)$ be given. We curtail each infinite word $i_1i_2\cdots\in\{1,\ldots,m\}^\N$ after the first term $i_k$ satisfying
\[
r_{\min} \varrho < \operatorname{sr}(f_{i_1\cdots i_k}) = r_{i_1} r_{i_2} \cdots r_{i_k} \le\varrho
\] 
and let $\mathcal{Q}_\varrho$ be the set of all finite sequences in $ \{1,\ldots,m\}^*$ obtained in this way. 
 Then it is easy to see that (recall that  $\varrho < 1$)
\[
K = \bigcup_{\mathbf{i} \in \mathcal{Q}_\varrho} f_{\mathbf{i}}(K).
\]

\begin{lemma}\label{lem:osc-rho}
Let $K\subset \R^d$ be the attractor of a self-similar IFS $\mathcal{F}:=\{f_1,\ldots, f_m\}$ satisfying the open set condition. Then there exists a constant $c$ such that 
for all $\varrho\in (0,\infty)$ and all balls $B\subset \R^d$ with diameter $\varrho$ there exists  $\mathcal D \subset \mathcal{Q}_\varrho$ such that $\#\mathcal D\le c$ and
\[
\bigcup_{\mathbf{i}\in \mathcal{D}}f_\mathbf{i}(K)  \supseteq K \cap \overline{B}.
\]
\end{lemma}
\begin{proof}
See, {\it e.g.}, the proof of \cite[Theorem~9.3]{Falconer2003}.
\end{proof}

Next, we consider the case of plane IFS: Let $K\subset \R^2$ be a nondegenerate continuum which is the attractor of a self-similar  IFS $\mathcal{F}:=\{f_1,\ldots, f_m\}$ satisfying $\# f_i(K) \cap f_j(K) < \infty$ for $1\le i<j\le m$. Suppose that $\mathcal{F}$ has finite type and hence, by \Cref{thm:finite-type-finite-types}  satisfies the finite types of neighborhoods condition for some $\varepsilon>0$. W.l.o.g.,\ $\varepsilon$ is such that $N_\varepsilon(\mathcal{F})=N(\mathcal{F})$.

For $\eta>0$ let \[
\mathbf{C}(\eta) := \bigg\{ \Big(-\frac\eta2,\frac\eta2\Big)^2 + z \colon \bigg(\Big(-\frac\eta2,\frac\eta2\Big)^2 + z\bigg) \cap K \ne\emptyset \;,\; z \in \frac\eta2\mathbb{Z}^2\bigg \}.
\]
The following rather technical choice of $\delta$ will be needed in the sequel, particularly in the proof of \Cref{Uproperties}. 
Since $H(\mathcal{F})$ is finite by \Cref{lemma:HK-finite} (and $\#H(\mathcal{F}) >2$ because $K$ is a nondegenerate continuum) we may choose $\delta>0$ with 
\begin{equation}\label{eq:deltadef1}
\delta < \frac14\min\big\{1,\varepsilon, \min\{\vertii{ x - y} \colon x,y\in H(\mathcal{F}),\, x\ne y\}\big\} 
\end{equation}
and such that for each $x\in H(\mathcal{F})$ and each $h\in N(\mathcal{F})$
\begin{equation}\label{eq:deltadef2}
h(K) \cap B_{2\delta}(x) \ne\emptyset \quad\Longrightarrow\quad  x\in h(K).
\end{equation}
The latter implication says that a ``neighbor'' $h(K)$ of $K$ that is $2\delta$-close to an intersection point of $K$  ({\it i.e.}, to an element of the finite set $H(\mathcal{F})$) already contains this intersection point. 

The set $K\setminus (H(\mathcal{F}))_{\delta}$ is nonempty (by the choice of $\delta$ and by connectedness of $K$) and compact. Since $N(\mathcal{F})$ is finite, and $h(K)$ is compact for each $k\in N(\mathcal{F})$, there is some $k\ge 1$  such that
\begin{equation}\label{eqn:U2-distance}
\min\Big\{\vertii{x-y}\colon x\in K\setminus (H(\mathcal{F}))_{\delta},\, y\in  \bigcup_{h\in N(\mathcal{F}) } h(K) \Big\}> \frac{\delta}{2^{k-1}}.
\end{equation}

Note that \eqref{eqn:U2-distance} guarantees that, for  every open set $V$ with diameter less than $\tfrac{\delta}{2^{k-1}}$ and every $h\in N(\mathcal{F})$,  
\begin{equation}\label{eqn:U2-distance-cor}
\overline{V}\cap K \ne \emptyset\text{ and } \overline{V}\cap h(K)\ne \emptyset\Longrightarrow \overline{V}\cap (H(\mathcal{F}))_\delta\ne \emptyset.
\end{equation}

Set
\begin{equation}\label{eq:U1}
\mathcal U_1:=\{V\in \mathbf{C}(\delta2^{-k})\colon V\cap (K\setminus (H(\mathcal{F}))_{\delta})\ne \emptyset\}.
\end{equation}
In Figure~\ref{fig:HataCover} we illustrate this cover for the {\it Hata tree}, a self-similar set that will appear in the examples later (see~\Cref{exs:p2}). The shaded trees are (parts of) $h(K)$ for $h\in N(\mathcal{F})$. The set $(H(\mathcal{F}))_{\delta}$ is indicated by shaded disks; the estimate in \eqref{eqn:U2-distance} states that outside these disks the attractor $K$ is separated from its ``neighbors'' $h(K)$ by a distance of at least $\frac{\delta}{2^{k-1}}$. Thus no element of $\mathcal{U}_1$ can intersect such a neighbor.
\begin{figure}[h]
\begin{center}
\includegraphics[width=\textwidth]{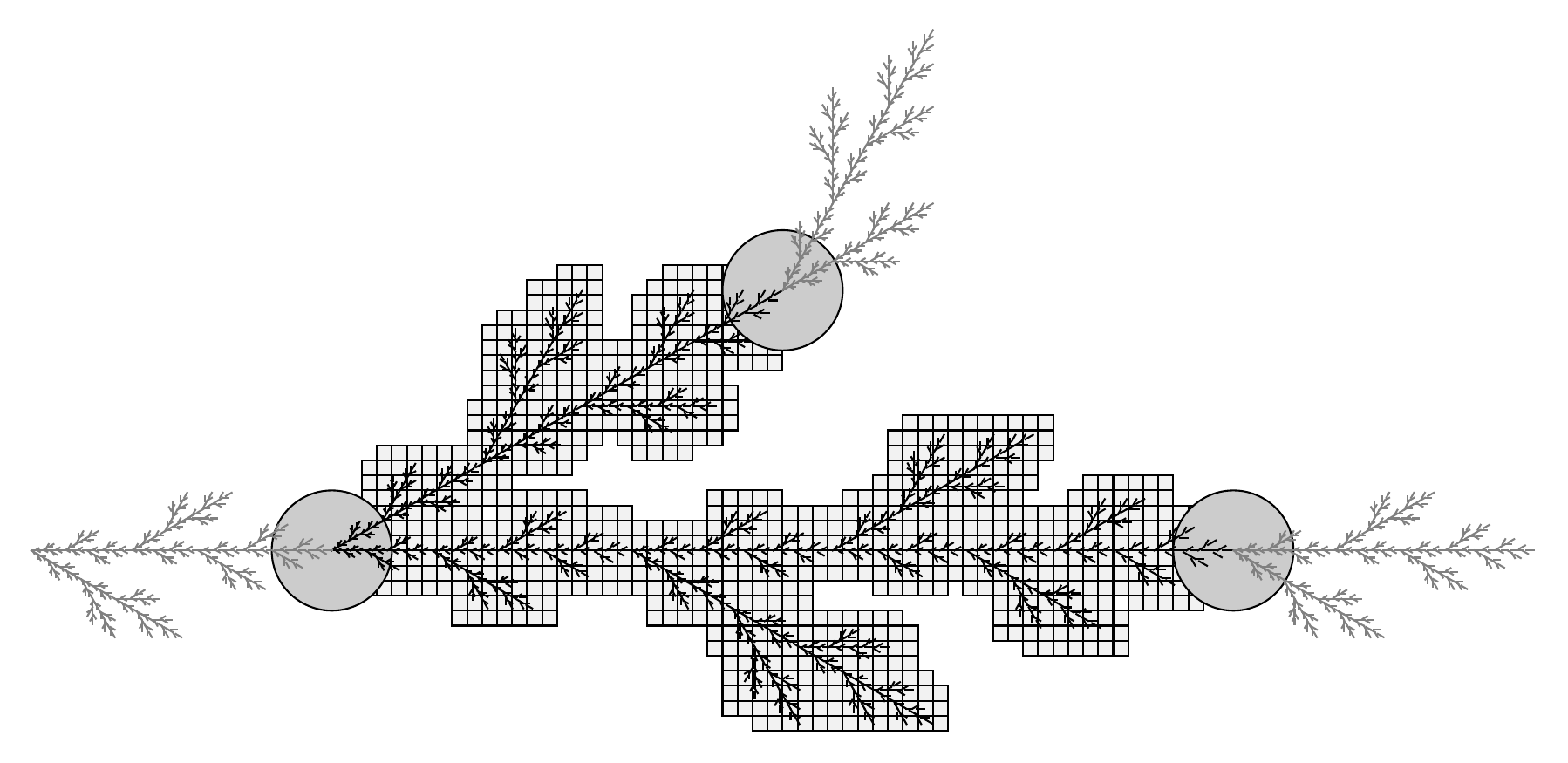}
\end{center}
\caption{The cover $\mathcal{U}_1\cup \{(H(\mathcal{F}))_{\delta}\}$ of the Hata tree $K$.}\label{fig:HataCover}
\end{figure}
We now iteratively define a sequence $(\mathcal{U}_n)_{n\ge 1}$ of covers of $K$  as follows. Suppose that $\mathcal{U}_1,\ldots,\mathcal{U}_{n-1}$ have already been defined. 
Let 
\begin{equation}\label{eq:Pn}
\mathcal{P}_{n}=\{\mathbf{i} \in \mathcal{Q}_{\delta^{n}}\colon
f_{\mathbf{i}}(K) \cap [H(\mathcal{F})]_{\delta^{n-1}} \ne\emptyset
\}.
\end{equation}
For each $\mathbf{i} \in \mathcal{P}_{n}$ we set
\begin{equation}\label{eq:Vbfi}
\mathcal{V}_{\mathbf{i}} = \{V \in  f_\mathbf{i}(\mathbf{C}(\delta2^{-k})) \colon 
V\cap (K\setminus (H(\mathcal{F}))_{\delta^n})\ne \emptyset
\}.
\end{equation}
We set
 \begin{equation}\label{eq:Vndef}
\mathcal{U}_{n} = \mathcal{U}_{n-1}  \cup  \bigcup_{\mathbf{i} \in \mathcal{P}_{n}} \mathcal{V}_{\mathbf{i}} .
\end{equation}

Note that, by Lemma~\ref{lem:osc-rho}, there exists a constant $c$ such that $\#\mathcal{P}_n \le c$ for all $n\in \N$. This uniform boundedness
will be essential later.

\begin{lemma}\label{Uproperties}
Let $(\mathcal{U}_n)_{n\ge 1}$ be defined as above and let $U_n:=\bigcup_{U\in \mathcal U_n}U$. Then the following assertions hold.
\begin{itemize}
\item[(1)] For each $n\ge 1$, the collection $\mathcal{U}_n$ is an open cover of $K\setminus (H(\mathcal{F}))_{\delta^n}$.
\item[(2)] For each $n\ge 1$, each $\varrho>0$, and each $\mathbf{j}\in\mathcal{Q}_\varrho$,we have 
\[
f_{\mathbf{j}}(\overline{U_n}) \cap \bigcup_{\begin{subarray}{c}\mathbf{k}\in\mathcal{Q}_\varrho\\\mathbf{k}\neq \mathbf{j} \end{subarray}} f_{\mathbf{k}}(K)=\emptyset.
\]
\end{itemize} 
\end{lemma}

\begin{proof}
In this proof the choice of $\delta$ satisfying  \eqref{eq:deltadef1} and \eqref{eq:deltadef2} will be of great importance.

Assertion (1) is proved by induction. For $\mathcal{U}_1$ this assertion follows from its definition in \eqref{eq:U1}. Now assume that $(1)$ holds for $\mathcal{U}_{n-1}$.  By the definition of $\mathcal{P}_{n}$ in \eqref{eq:Pn} the union $\bigcup_{\mathbf{i} \in \mathcal{P}_{n}} \mathcal{V}_{\mathbf{i}} $ covers $ [H(\mathcal{F})]_{\delta^{n-1}} \setminus  (H(\mathcal{F}))_{\delta^{n}}$ and, hence, by \eqref{eq:Vndef}, (1) also holds for $\mathcal{U}_{n}$.

To prove assertion (2), we use induction as well. Choose two distinct elements $\mathbf{j},\mathbf{k}\in \mathcal{Q}_\varrho$ arbitrary and set $h:=f^{-1}_\mathbf{j}\circ f_\mathbf{k}$. By the definition of $\mathcal{Q}_\varrho$ we have $\operatorname{sr}(h)\in (r_{\min},r_{\min}^{-1})$,  and by \Cref{lemma:different-indices}, $h\ne \id$. 

To show (2) for $n=1$ it suffices to prove that 
\begin{equation}\label{eq:U1hKempty}
\overline{U_1} \cap h(K) = \emptyset. 
\end{equation}
If $h \not \in N(\mathcal{F})$ then $[K]_\varepsilon \cap h([K]_\varepsilon) = \emptyset$ by our choice of $\varepsilon$. Thus, because $\delta<\frac\varepsilon4$ and $U_1\subset [K]_\delta$ by \eqref{eq:U1}, we get \eqref{eq:U1hKempty}. It remains to deal with $h\in N(\mathcal{F})$. In this case \eqref{eq:U1hKempty} follows immediately from \eqref{eqn:U2-distance} and \eqref{eq:U1} because the diameter of the elements of $\mathbf{C}(\delta 2^{-k})$ is $\delta 2^{-k+1/2} < \delta 2^{-k+1}$.

Now assume that assertion (2) holds for $n-1$. In view of \eqref{eq:Vndef} 
it is sufficient  to show that for each $\mathbf{i} \in \mathcal{P}_n$ and each $V \in \mathcal{V}_\mathbf{i}$ we have 
\begin{equation}\label{eq:VhKempty}
\overline{V} \cap h(K)=\emptyset. 
\end{equation}
If $h \not \in N(\mathcal{F})$ then $[K]_\varepsilon \cap h([K]_\varepsilon) = \emptyset$. By \eqref{eq:Vbfi} we have $V \subset [K]_{\delta 2^{-k+1/2}}\subset [K]_{\varepsilon}$ from which we immediately gain \eqref{eq:VhKempty}. Next  we  assume that $h\in N(\mathcal{F})$. We argue by contradiction. Suppose that \eqref{eq:VhKempty} does not hold. Then there exists $\mathbf{l}\in\{1,\ldots,m\}^*$  such that 
\begin{equation}\label{eq:vhfnot0}
\overline{V} \cap h\circ f_\mathbf{l} (K)\ne\emptyset,
\end{equation}
and such that $g:=f^{-1}_\mathbf{i} \circ h \circ f_\mathbf{l}=f^{-1}_\mathbf{i} \circ f^{-1}_\mathbf{j} \circ f_\mathbf{k} \circ f_\mathbf{l}$ has a similarity ratio $\operatorname{sr}(g) \in [r_{\min},1)$.
Applying $f^{-1}_\mathbf{i}$ to \eqref{eq:vhfnot0} we gain 
\begin{equation}\label{eq:vhfnot0b}
f^{-1}_\mathbf{i}(\overline{V}) \cap g(K) \ne \emptyset. 
\end{equation}
Since $f^{-1}_\mathbf{i}(\overline{V}) \in \mathbf{C}(\delta2^{-k})$ has diameter bounded by $\delta2^{-k+1/2} <\frac{\delta}{2^{k-1}}<\delta< \frac\varepsilon4$ and satisfies $f^{-1}_\mathbf{i}(\overline{V})\cap K\not=\emptyset$, we conclude that $g\in N(\mathcal{F})$. Thus \eqref{eqn:U2-distance-cor} implies that
 \begin{equation}\label{eq:vhfnot0c}
f^{-1}_\mathbf{i}(\overline{V}) \cap (H(\mathcal{F}))_\delta \ne\emptyset.
 \end{equation}
Observe that, by \eqref{eq:deltadef1}, $\delta$ is so small that $f^{-1}_\mathbf{i}(\overline{V})$ can intersect the $\delta$-neighborhood of at most one point $x\in H(\mathcal{F})$. By \eqref{eq:vhfnot0b}, for this point $x$ we have $f^{-1}_\mathbf{i}(\overline{V}) \subset B_{2\delta}(x)$. Thus \eqref{eq:vhfnot0c} yields $B_{2\delta}(x) \cap g(K) \ne\emptyset$, and by the definition of $\delta$ in \eqref{eq:deltadef2} we gain that $x\in K\cap g(K)$. Combining this with \eqref{eq:vhfnot0c} again, yields $f^{-1}_\mathbf{i}(\overline{V}) \cap (K \cap g(K))_\delta \ne\emptyset$. 
Applying the similarity transformation $f_\mathbf{i}$ and observing that  $\operatorname{sr}(f_\mathbf{i}) \le \delta^n$ we derive
\begin{equation}\label{eq:VfiNonEmpty}
\emptyset \ne \overline{V} \cap (f_\mathbf{i}(K) \cap h\circ f_\mathbf{l}(K))_{\delta^{n+1}} 
\subset \overline{V} \cap (K \cap h(K))_{\delta^{n+1}} 
\subset \overline{V} \cap (H(\mathcal{F}))_{\delta^{n+1}}. 
\end{equation}
On the other hand, by the definition of $\mathcal{V}_\mathbf{i}$ we have that $ \overline{V} \setminus (H(\mathcal{F}))_{\delta^n} \ne\emptyset$. And because $\mathrm{diam}(\overline{V})<\delta^{n+1}<\delta^n-\delta^{n+1}$, since $\delta<\frac{1}{4}$, we have $\overline{V} \cap (H(\mathcal{F}))_{\delta^{n+1}}=\emptyset$, a contradiction to \eqref{eq:VfiNonEmpty}. This concludes the proof.
\end{proof}
 
We are now in a position to exhibit a path $\alpha$ that surrounds small copies of $K$ by intersecting $K$ in finitely many points.

\begin{lemma}\label{prop:2difsaux}
Let $(U_n)_{n\ge 1}$ be as in \Cref{Uproperties},  and $O:=\bigcup_{n\ge 1} U_n$. The boundary of the unbounded component\footnote{Note that in general, $\R^2\setminus O$ is not connected.} of $\R^2\setminus O$ is a path $\alpha$ of finite length satisfying $f_\mathbf{i}(\alpha) \cap K= f_\mathbf{i}(H(\mathcal{F}))$ for each $\mathbf{i}\in\{1,\ldots,m\}^*$.
\end{lemma}

\begin{proof}
Each element of $\mathcal{U}_n$ is a square. We want to calculate the sum $\ell_n$ of the lengths $\ell(\partial U)$ of all loops\footnote{Recall that a loop in $\R^d$ is a continuous injection $\alpha\colon\mathbb{S}^1\to \R^d$. Since  $\alpha$ can be considered as a path in $\R^d$, its length is well defined. } in the set $L_n:=\{\partial U\colon U \in \mathcal{U}_n\}$. Because $\# \mathcal{P}_n \le c$ uniformly in $n$, we gain from  \eqref{eq:Vbfi}, \eqref{eq:Vndef}, and \eqref{eq:U1}
\begin{align}
\ell_n &:= \sum_{U\in \mathcal{U}_n} \ell(\partial U) = \sum_{U\in \mathcal{U}_1} \ell(\partial U) + \sum_{j=2}^n\sum_{\mathbf{i}\in\mathcal{P}_j} \sum_{U\in \mathcal{V}_\mathbf{i}} \ell(\partial U) \nonumber\\
 &\le  4 \delta 2^{-k} \#\mathbf{C}(\delta 2^{-k}) +  \sum_{j=2}^n c 4 \delta^{j}2^{-k}\#\mathbf{C}(\delta 2^{-k}) \label{sumDeltaUAlign} \\
&\le 2^{-k+2}c\#\mathbf{C}(\delta 2^{-k}) \sum_{j=1}^n \delta^{j} < 2^{-k+2}c\#\mathbf{C}(\delta 2^{-k})
\nonumber
\end{align}
(in the last inequality we used that $\delta<\frac14$). Because the last term in \eqref{sumDeltaUAlign} does not depend on $n$, the sequence $(\ell_n)_{n\ge 1}$ is bounded. 

Set $V_n:= U_n \cup (H(\mathcal{F}))_{\delta^{n-1}}$. Let $\alpha_n$ be the boundary of the unbounded complementary component of $V_n$. By Lemma~\ref{Uproperties}~(1), $\mathcal{U}_n \cup \{B_{\delta^{n-1}}(x)\colon x\in H(\mathcal{F})\}$  is an open cover of the connected set $K$ each of whose elements has nonempty intersection with $K$. Thus its closure $\overline{V_n}$ is a nondegenerate Peano continuum without cut points and, hence, Torhorst's Theorem (see~\cite[\S61, II, Theorem~4, (ii)]{Kuratowski68}) shows that $\alpha_n$ is a loop. Because this loop is contained in $L_n \cup \partial [H(\mathcal{F})]_{\delta^{n-1}}$, also the length of $\alpha_n$, $n\in \N$, is bounded by some constant $\ell_{\max}$ (note that  $\partial [H(\mathcal{F})]_{\delta^{n-1}}$ is the union of finitely many circles of radius $\delta^{n-1}$).

By the definition of $V_n$, for each $\eta>0$ there is $N\in \mathbb{N}$ such that 
$V_n \triangle V_m \subset [H(\mathcal{F})]_\eta$ for all $n,m > N$. 
Thus the sequence $(\alpha_n)_{n\in \N}$ converges uniformly  to a loop $\alpha$, whereby
$\alpha$ coincides with $\alpha_n$ on the complement of  $[H(\mathcal{F})]_{\delta^{n}}$ for every $n$. Hence the length of $\alpha$ is bounded by $\ell_{\max}$.
 By construction, $\alpha$ is the boundary of the unbounded complementary component of $O$ and $\alpha\cap K= H(\mathcal{F})$, which is finite. From this we get
$f_\mathbf{i}(\alpha) \cap K=f_\mathbf{i}(H(\mathcal{F}))$. Because we can regard $\alpha$ as a path, the result follows.
\end{proof}

\begin{proof}[Proof of \Cref{thm:ftn-null-permeable2d}]
Because of the finiteness condition $\#f_i(K) \cap f_j(K) < \infty$, it  follows from \Cref{prop:ftc-empty-int}  that $K$ has empty interior. 
Moreover, by Lemma~\ref{lem:osc}, $K$ satisfies the open set condition. 
However, according to \cite[(2.3) Corollary]{Schief:94}, a self-similar set satisfying the open set condition has empty interior if and only if it has Lebesgue measure zero. Thus $K$ has Lebesgue measure zero. Moreover, \Cref{thm:finite-type-finite-types} implies that $K$ has finite types of neighborhoods. Let $O$ be as in \Cref{prop:2difsaux}. By \Cref{prop:2difsaux}, the boundary of the unbounded complementary component of $O$ is a path $\alpha$ of finite length satisfying $f_\mathbf{i}(\alpha) \cap K= f_\mathbf{i}(H(\mathcal{F}))$ for each $\mathbf{i}\in\{1,\ldots,m\}^*$. 

Let $x$, $y$ be two points in $x,y \in \mathbb{R}^2\setminus K$. Since $K$ is compact, we may choose $\delta>0$
such that $\min\big\{d(x,K),d(y,K)\big\}>\frac{\delta}{4}$. 
By \Cref{lemma:outside-theta-closed}, to prove permeability of $K$, we have to construct a path 
$\gamma\colon [0,1]\to \R^2$ connecting $x$ and $y$ with $\ell(\gamma)<\vertii{x-y} +\delta$, and with $K\cap \gamma$ countable. 

Let $g$ be the line passing through $x$ and $y$. Because $K$ is a Lebesgue nullset, Fubini's theorem implies that there is a line $g'\subset [g]_{\delta/4}$, parallel to $g$, such that $\lambda_{g'}(g'\cap K) = 0$ (here, $\lambda_{g'}$ is the one-dimensional Lebesgue measure on $g'$). Let $x'$ and $y'$ be the orthogonal projections of $x$ and $y$ to $g'$, respectively. Since $\min\big\{d(x,K),d(y,K)\big\}>\frac{\delta}{4}$, the line segments $\overline{xx'}$ and $\overline{y'y}$ have empty intersection with $K$.   Let $c$ be the constant from \Cref{lem:osc-rho} and note that $c$ only depends on $f_1,\ldots,f_m$. Since $\lambda_{g'}(g'\cap K) = 0$, and $g'\cap K$ is closed,  we can cover $g'\cap K$, regarded as a subset of $g'$, with finitely many open intervals with total measure less than $\frac{\delta}{4c\ell(\alpha)}$. That means that $g'\cap K$, regarded as a subset of $\R^2$, can be covered by finitely many disks $B_{r_1}(x_1),\ldots,B_{r_k}(x_k)$ with $\sum_{j=1}^k 2r_j<\frac{\delta}{4c\ell(\alpha)}$. 
By \Cref{lem:osc-rho}, for every $j\in\{1,\ldots,k\}$ there exists $\mathcal D_j \subset \mathcal{Q}_{2r_j}$ with $\#\mathcal D_j\le c$
and \[
\bigcup_{\mathbf{i}\in \mathcal{D}_j}f_\mathbf{i}(K)  \supseteq K \cap \overline{B_{r_j}(x_j)}.
\]
 Set $I:=\mathcal{D}_1\cup\cdots\cup \mathcal{D}_k$. Then
 \begin{align*}
\sum_{\mathbf{i}\in I}\operatorname{sr}(f_\mathbf{i})\le\sum_{j=1}^k2r_j c<2c\frac{\delta}{4c\ell(\alpha)}=\frac{\delta}{2\ell(\alpha)}.
\end{align*}
Consider the unbounded component of $\R^2\setminus\Big(g' \cup \bigcup_{i\in I} f_\mathbf{i}(O)\Big)$. Its boundary is contained in 
$g'\cup \bigcup_{i\in I} f_\mathbf{i}(\alpha)$, which is path-connected and contains $x'$ and $y'$, and so there exists a path
\(
\beta \subset g' \cup \bigcup_{\mathbf{i}\in I} f_\mathbf{i}(\alpha)
\)
which connects $x'$ and $y'$ with 
\[\big\vert \ell(\beta) - \vertii{x'-y'} \big\vert \le \sum _{\mathbf{i}\in I} \ell(f_\mathbf{i}(\alpha))=\ell(\alpha) \sum_{\mathbf{i}\in I}\operatorname{sr}(f_\mathbf{i})<\frac{\delta}{2} ,\] 
and $\beta \cap f_\mathbf{i}(K)$ is finite. Concatenating $\overline{xx'},\beta,\overline{y'y}$ yields the desired path $\gamma$.
\end{proof}

\begin{examples}\label{exs:p2}
\Cref{thm:ftn-null-permeable2d} applies to a wealth of examples. For instance, it  implies permeability of the {\it von Koch Curve} (see \cite[Chapter~3]{Edgar:04}), the {\it Sierpi\'nski triangle} (see~\cite{Sierpinski:1915}), and the {\it Hata tree} (see~\cite{Hata:85}) (for standard references to these famous self-similar sets we refer for instance to~\cite{Barnsley:93,Falconer2003,YamagutiHataKigami1997}).  Besides that, the examples of self-similar sets $K$ in Barnsley~\cite[Chapter~7, Example~1.14]{Barnsley:93}, satisfying $ \# f_i(K) \cap f_j(K) <\infty$ (which can be verified easily), are permeable by \Cref{thm:ftn-null-permeable2d}. In each of these examples the finite type condition is obviously satisfied.

We also refer to \cite{bandt2018elementary,Bandt:22}, where the computer system {\tt ifstile}\footnote{See \url{https://ifstile.com}.} is used to create interesting carpets. This system can be used to create further examples to which \Cref{thm:ftn-null-permeable2d} can be applied.
\end{examples}

 The {\it Sierpi\'nski carpet}\footnote{See {\it e.g.}~\cite[Example~1.3.17]{Kigami:01} for the definition of the Sierpi\'nski carpet.} $K$ mentioned in Example~\ref{exs:basic}~(\ref{ex:carpet1}) is an impermeable self-similar set in the plane.
It is a nondegenerate Peano continuum (this follows immediately, see \cite[Theorem~4.6]{Hata:85}) with $\dim_H(K)=\dim_B(K)=\frac{\log8}{\log3}\approx 1.8928$ (by \cite[Theorem~9.3]{Falconer2003}). Moreover, it is clearly porous  in the sense of Definition~\ref{def:porosity}. 
This example shows that one cannot get rid of the condition $\# f_i(K) \cap f_j(K) <\infty$ in \Cref{thm:ftn-null-permeable2d}.

\subsection{Null permeability of self-similar sets in higher dimensions}\label{sec:3difs}
We now turn to self-similar sets in dimension $d\ge 3$. It turns out that Theorem~\ref{th:Nagata} on the null permeability of sets with small Nagata dimension is useful here. The following theorem, which gives a criterion for null permeability of self-similar sets of dimension $d\ge 3$ is a consequence of this result. Indeed, we will show that the finite type condition from Definition~\ref{def:finiteneighborhoods} provides a criterion for a self-similar set to have Nagata dimension less than or equal to $1$. 

\begin{theorem}\label{thm:ftn-null-permeable3d}
For $d\ge 3$ let $K\subset \R^d$ be the attractor of a self-similar IFS $\mathcal{F}:=\{f_1,\ldots, f_m\}$ satisfying $\#f_i(K) \cap f_j(K) < \infty$ for $1\le i<j\le m$.
Suppose further that $\mathcal{F}$ is of finite type.
Then $K$ is null permeable.
\end{theorem}

To prove this result, we show that the conditions of the theorem imply that the Nagata dimension of $K$ is less than or equal to $1$. This result is of interest in its own right.

\begin{proposition}\label{lem:topdimIFS}
For $d\ge 2$ let $K\subset \R^d$ be the  attractor of a self-similar IFS $\mathcal{F}:=\{f_1,\ldots, f_m\}$ satisfying the finite type condition and $\#f_i(K) \cap f_j(K) < \infty$ for $1\le i<j\le m$. Then $\dim_N(K)\le~\!1$.
\end{proposition}

\begin{proof} 
Since the case where $\# K=1$ is trivial, we assume that $\#K> 1$, so in particular, $\diam(K)>0$. W.l.o.g., we may assume that $\diam(K)=1$.
By \Cref{thm:finite-type-finite-types}, $K$ has finite types of neighborhoods, 
so, by \Cref{lemma:HK-finite}, $H(\mathcal{F})$ is finite and there exists $\varepsilon\in(0,1)$
such that $N_\varepsilon(\mathcal{F})=N(\mathcal{F})$, and $N(\mathcal{F})$ is finite. 
Thus $Z:=H(\mathcal{F})\cup\{h(H(\mathcal{F}))\colon h \in N(\mathcal{F})\}$ is finite as well. W.l.o.g., we assume\footnote{As usual, a minimum over an empty set is infinite.} 
\begin{equation}\label{nagatad1-epsilonZ}
\varepsilon<\min\{\vertii{x-y}\colon x,y \in Z, x\ne y\}.
\end{equation} 
As before, let $r_{\min}$ denote the minimum of the similarity ratios of $f_1,\ldots, f_m$. Set
\begin{align}
\label{nagatad1-c1}c_1&:=\frac{\varepsilon r_{\min}}{8},\\
\label{nagatad1-c2}c_2&:=\frac{r_{\min}}{2}\min\{\vertii{x-y}\colon x\in K \setminus (H(\mathcal{F}))_{c_1}, y\in h(K \setminus (H(\mathcal{F}))_{c_1}), h\in N(\mathcal{F})\},
\end{align}
which are strictly positive by the finiteness of $N(\mathcal{F})$ and the compactness of $K \setminus (H(\mathcal{F}))_{c_1}$.
Further define $c:=\min\{c_1,c_2\}$ and let $s>0$. We define a collection
\begin{align*}
\mathcal{U}_0 &:= \big\{f_{\mathbf{i}}(B_{2 c_1}(x)) \colon x\in H(\mathcal{F}),\,\mathbf{i} \in \mathcal{Q}_{s/2}   \big\}.
\end{align*}
Note that $\mathcal{U}_0$ is not $cs$-separated, because it may contain different balls with the same center. Let 
$\mathcal{U}_1$ be the collection of sets obtained from $\mathcal{U}_0$, if for every such center only the ball with the largest radius is chosen, {\em i.e.}, 
\[
\mathcal{U}_1 := \{B \in \mathcal{U}_0 : \text{there is no } B'\in \mathcal{U}_0 \text{ with } B\subsetneq B'\}.
\]
Further set 
\begin{align*}
\mathcal{U}_2 &:= \big\{f_{\mathbf{i}}(K \setminus [H(\mathcal{F})]_{c_1}) \colon \mathbf{i} \in \mathcal{Q}_{s/2} \big\}.
\end{align*} 
and consider the union $\mathcal{U}=\mathcal{U}_1\cup \mathcal{U}_2$. 
From our choice \eqref{nagatad1-c1} of $c_1$,  $\mathcal{U}$ is an $s$-cover of $K$. 
Indeed, this is true for each element of $\mathcal{U}_1$ because $ c_1<\tfrac 1 8$ and, hence, $\diam\big(f_{\mathbf{i}}(B_{2 c_1}(x))\big)\le \tfrac s 2 \cdot 4 c_1<s$ for every 
$\mathbf{i}\in \mathcal{Q}_{s/2}$. Moreover, $V=f_{\mathbf{i}}(K \setminus [H(\mathcal{F})]_{c_1})\in \mathcal{U}_2$ has diameter less than or equal to $s$ since $\operatorname{sr}(f_{\mathbf{i}})\le \tfrac s 2$ and $\diam(K)=1$. 

It remains to show that both, $\mathcal{U}_1$ and $\mathcal{U}_2$, are $c s$-separated. 
First consider two distinct elements $B_1,B_2\in \mathcal{U}_1$ with centers $x_1$ and $x_2$, 
respectively. Note that each of their radii is less than or equal to $c_1s$. There exist $\mathbf{i}_1,\mathbf{i}_2 \in \mathcal{Q}_{s/2}$ and $z_1,z_2\in H(\mathcal{F})$ such that 
$x_1=f_{\mathbf{i}_1}(z_1)$ and $x_2=f_{\mathbf{i}_2}(z_2)$. If $f_{\mathbf{i}_1}^{-1}(x_2)\notin Z$, {\em i.e.}, 
$f_{\mathbf{i}_1}^{-1}\circ f_{\mathbf{i}_2}(z_2)\notin Z$ then $f_{\mathbf{i}_1}^{-1}\circ f_{\mathbf{i}_2}\notin N(\mathcal{F})\cup \{\id\}$.
Thus $[H(\mathcal{F})]_\varepsilon\cap \big(f_{\mathbf{i}_1}^{-1}\circ f_{\mathbf{i}_2}\big)([H(\mathcal{F})]_\varepsilon)=\emptyset$,  and, hence, 
\begin{equation}\label{nagata-dist-centers1}
\begin{aligned}
\vertii{x_1-x_2}&= \vertii{f_{\mathbf{i}_1}(z_1)-f_{\mathbf{i}_2}(z_2)}\ge \frac{s}{2} r_{\min}\vertii{z_1-f_{\mathbf{i}_1}^{-1}\circ f_{\mathbf{i}_2}(z_2)}\\
&\ge \frac{s}{2} r_{\min}\min\{\vertii{z_1-y}\colon y\in f_{\mathbf{i}_1}^{-1}\circ f_{\mathbf{i}_2}(H(\mathcal{F})) \}\ge \frac{s}{2} r_{\min}\varepsilon=4c_1s,
\end{aligned}
\end{equation}
where we used \eqref{nagatad1-c1} in the last equality.
Now suppose $f_{\mathbf{i}_1}^{-1}(x_2)\in Z$. Since $f_{\mathbf{i}_1}$ is bijective 
and $x_1\ne x_2$,
this implies $f_{\mathbf{i}_1}^{-1}(x_2)\in Z\setminus\{z_1\}$. Hence, 
$\vertii{z_1-f_{\mathbf{i}_1}^{-1}(x_2)}\ge \frac{8 }{r_{\min}}c_1$ by \eqref{nagatad1-epsilonZ} and \eqref{nagatad1-c1}, and because $\operatorname{sr}(f_{\mathbf{i}_1})\ge \frac{s}{2}r_{\min}$, we gain 
\begin{equation}\label{nagata-dist-centers2}
\vertii{x_1-x_2}\ge 4 c_1s.
\end{equation}
Now, since $B_1,B_2$ are balls with radii less than or equal to $c_1s$, we get from \eqref{nagata-dist-centers1}, \eqref{nagata-dist-centers2},  and  the inequality $c_1\ge c$ that 
\[
\min\{\vertii{x-y}\colon x\in B_1,y\in B_2\}\ge \vertii{x_1-x_2} -2 c_1s\ge 4 c_1s-2 c_1s>  cs,
\]
so $\mathcal{U}_1$ is $cs$-separated.

Finally, let $\mathbf{i},\mathbf{j} \in \mathcal{Q}_{s/2}$, with $\mathbf{i}\ne \mathbf{j}$. If $f^{-1}_\mathbf{i}\circ f_\mathbf{j}\notin N(\mathcal{F})$, then $f_\mathbf{i}([K]_\varepsilon)\cap  f_\mathbf{j}([K]_\varepsilon)=\emptyset$, and therefore 
\begin{align*}
\lefteqn{\min\{\vertii{x-y}\colon x\in f_\mathbf{i}(K \setminus (H(\mathcal{F}))_{c_1}), y\in f_\mathbf{j}(K \setminus (H(\mathcal{F}))_{c_1})\}}\\
&\ge\min\{\vertii{x-y}\colon x\in  f_\mathbf{i}(K), y\in f_\mathbf{j}(K)\}>2 \varepsilon \frac{s}{2} r_{\min}= \varepsilon s r_{\min}\ge cs,
\end{align*} 
where we used \eqref{nagatad1-c1} and the inequality $c_1\ge c$ in the last estimate. If $f^{-1}_\mathbf{i}\circ f_\mathbf{j}\in N(\mathcal{F})$,
then by \eqref{nagatad1-c2} and the inequality $c_2\ge c$ we have 
\begin{align*}
&\min\{\vertii{x-y}\colon x\in f_\mathbf{i}(K \setminus (H(\mathcal{F}))_{c_1}), y\in f_\mathbf{j}(K \setminus (H(\mathcal{F}))_{c_1})\}\\
&= \operatorname{sr}(f_\mathbf{i}) \min\{\vertii{x-y}\colon x\in K \setminus (H(\mathcal{F}))_{c_1}, y\in f^{-1}_\mathbf{i}\circ f_\mathbf{j}(K \setminus (H(\mathcal{F}))_{c_1})\}\\
&\ge \operatorname{sr}(f_\mathbf{i}) \frac{2  c_2}{r_{\min}} \ge \frac{s}{2}r_{\min}    \frac{2c_2}{r_{\min}}\ge cs.
\end{align*}
We have shown  that $\mathcal{U}_2$ is $cs$-separated.

Thus, according to Definition~\ref{def:nagata}~(2), $K$ has Nagata dimension less than or equal to $1$.
\end{proof}

On close examination of the previous proof one sees that, if $H(\mathcal{F})=\emptyset$, then 
already $\mathcal{U}_2$ is a $cs$ separated set, and therefore $\dim_N(K)=0$. 
The proof remains correct also in this case.

\begin{proof}[Proof of \Cref{thm:ftn-null-permeable3d}]
From \Cref{lem:topdimIFS} we gain $\dim_N(K) \le1$. Since $d \ge 3 \; \Leftrightarrow \; 1 \le d-2$ the result follows from \Cref{th:Nagata}.
\end{proof}

It is reasonable to conjecture that, with suitable assumptions on $\mathcal{F}$, we have
$\dim_N(K)\le k+1$ if $\dim_N(f_i(K) \cap f_j(K))\le k$ for $i\neq j$. Although this would lead to a generalization of \Cref{thm:ftn-null-permeable3d}, we do not pursue this here because in general, for a given self-similar set $K$, the inequalities $\dim_N(f_i(K) \cap f_j(K))\le k$ ($i\neq j$) do not seem to be  easier to check than $\dim_N(K)\le k+1$.

\begin{example}[Sierpi\'nski tetrahedron]\label{ex:SierpTetr2}
In Example \ref{ex:SierpinskiTetrahedron} we have already shown that the Sierpi\'nski tetrahedron in $\R^3$ is permeable.
\Cref{thm:ftn-null-permeable3d} gives the stronger result that 
the Sierpi\'nski tetrahedron is null permeable.
\end{example}

\subsection{Final example: A Bedford-McMullen carpet}\label{sec:bedford-mcmullen}

The classical Sierpi\'nski carpet is the attractor of a self-similar IFS, it can be represented explicitly as 
\[
\Big\{\sum_{k=1}^\infty 3^{-k} x_k\colon x_1,x_2,\ldots\in \{0,1,2\}^2\setminus \{(1,1)\}\Big\}\subset [0,1]^2.
\] 
As already stated in \Cref{exs:basic}(\ref{ex:carpet1}), the classical Sierpi\'nski carpet is impermeable. It has Lebesgue measure 0 and topological dimension 1. We want to construct an impermeable subset of $\R^2$ of measure 0 and topological dimension 0. The construction is similar to that of \cite[Section 4.C.1]{hakobyan2008} and can be generalized to subsets of $\R^d$ for $d\ge 1$, {\em cf.}\ \Cref{thm:bedf-mcmullen-d}.

\smallskip

A {\em Bedford-McMullen carpet} (BMC, for short; see, {\em e.g.}, \cite{Fraser:21b} for a recent survey on these objects) is a self-affine subset of $[0,1]^2$ that is defined as follows: Fix  $n,m\in \N$ and consider a fixed subset $R\subset \{0,\ldots,n-1\}\times \{0,\ldots,m-1\}$,  which we will call {\em $n\times m$-pattern}. Set 
\[
K_R:=\Big\{\sum_{k=1}^\infty T^k x_k\colon x_1,x_2,\ldots\in R\Big\}\subset [0,1]^2,
\] 
where $T(x,y):=\big(\frac{x}{n},\frac{y}{m}\big)$, $(x,y)\in \R^2$. The closed set $K_R$ is the attractor of the self-affine IFS $\{f_{(i,j)} \colon (i,j)\in R\}$ consisting of the functions
\begin{equation}\label{eq:bm-functions}
f_{(i,j)}(x,y):=\Big(\frac{x}{n},\frac{y}{m}\Big)+\Big(\frac{i}{n},\frac{j}{m}\Big)\,,\qquad(i,j)\in R. 
\end{equation}

\begin{figure}[h]
     \centering
     \begin{subfigure}[b]{0.416\textwidth}
         \centering
         \includegraphics[width=\textwidth]{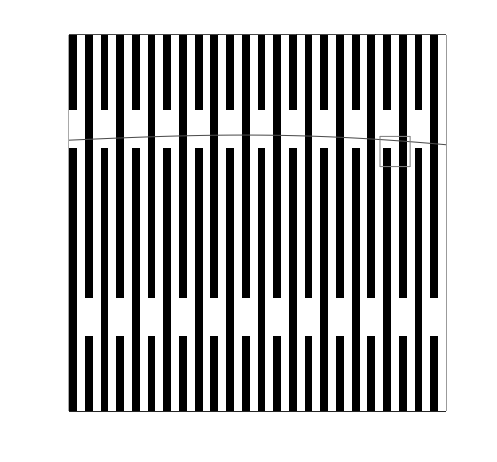}
         \caption{A path crossing the BMC} 
     \end{subfigure}
     \hfil
     \begin{subfigure}[b]{0.55\textwidth}
         \centering
         \includegraphics[width=\textwidth]{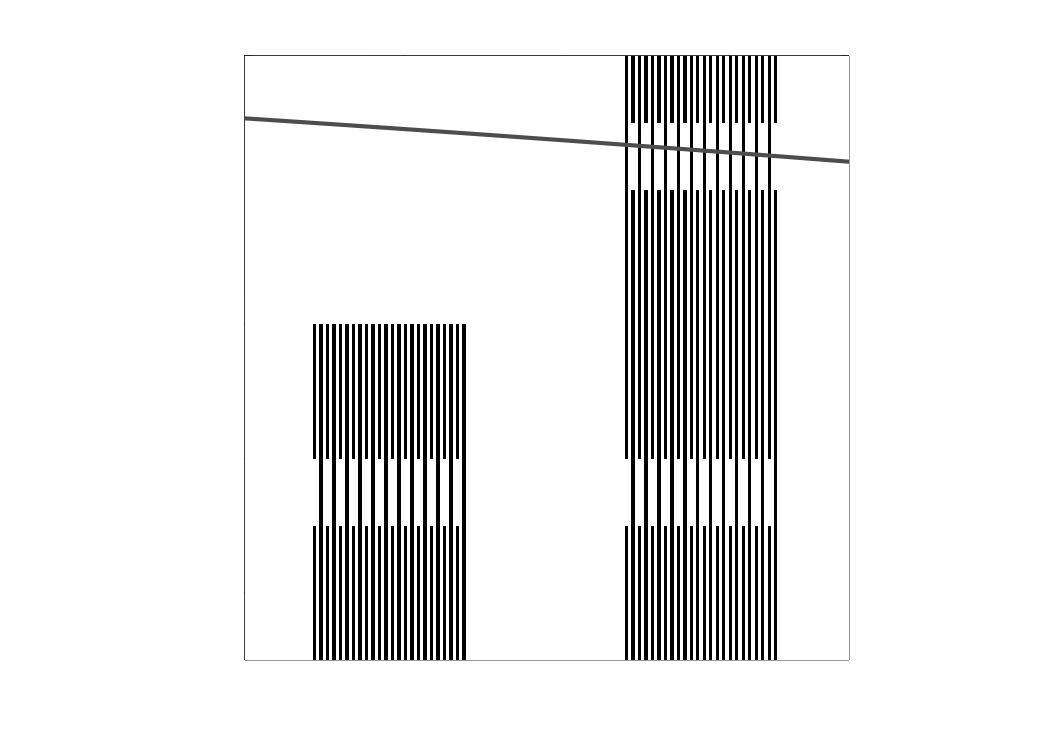}
         \caption{A magnified section of (A)} 
     \end{subfigure}
 	\caption{A Bedford-McMullen carpet}
        \label{fig:BMC-imp}
\end{figure}

For our purposes we consider the $48\times 10$-pattern $R$ depicted in \Cref{fig:BMC-imp}~(A), {\it i.e.}, we set
\[
R:=\{(2 i,j)\colon 0\le i<24, 0\le j< 10, j\not \equiv 5i+7\pmod {10}\}.
\] 
Clearly, $K:=K_R$ is closed and has Lebesgue measure 0. 
We show that $K$ is impermeable. For this purpose we formulate a lemma. In its statement recall that the total variation $V_a^b(f)$ of a function 
$f\colon [a,b]\to \R$ is defined as 
\[
V_a^b(f):=\sup\bigg\{  \sum_{k=1}^N |f(t_k)-f(t_{k-1})|\colon N\in\N,\,a=t_0<\cdots<t_N=b
\bigg\}.
\]

Having $n=48$ and $m=10$, let $R_2 := R \cup (R + \{(0,10)\})$ and let
\begin{equation}\label{eq:bm-functions-neu}
f_{(i,j)}(x,y):=\Big(\frac{x}{n},\frac{y}{m}\Big)+\Big(\frac{i}{n},\frac{j}{m}\Big)\,,\qquad(i,j)\in R_2. 
\end{equation}

In the next lemma and its proof, we say that $\gamma\colon [a,b]\to \R^2$ is a path from $X\subset \R^2$ to $Y\subset \R^2$, if $\gamma(a)\in X$ and $\gamma(b)\in Y$.
We denote the coordinate functions of a path $\gamma$ in $\R^2$ by $(\gamma)_1$ and $(\gamma)_2$, respectively, {\it i.e.}, we write $\gamma(t)=((\gamma)_1(t),(\gamma)_2(t))$. 

\begin{lemma}\label{lemma:zoltan-neu}
Let $\gamma\colon [a,b]\to \R^2$ be a path from $\{0\}\times[0,2]$ to $\{1\}\times[0,2]$, satisfying $\gamma([a,b])\subset [0,1]\times [0,2]$ and $V_a^b((\gamma)_2) < 1$. Then there exist elements $(i_1,j_1),(i_2,j_2) \in R_2$ with $i_1<i_2$ and $(i_1,j_1+1),(i_2,j_2+1) \in R_2$ with the following property:
There exist $a_1<b_1<a_2<b_2$ such that, for $k\in \{1,2\}$,  $\gamma_k:=\gamma|_{[a_k,b_k]}$ is a path from 
$f_{(i_k,j_k)}(\{0\}\times[0,2])$ to $f_{(i_k,j_k)}(\{1\}\times[0,2])$ satisfying $\gamma_k([a_k,b_k])\subset f_{(i_k,j_k)}([0,1]\times[0,2])$ and $V_{a_k}^{b_k}((\gamma_k)_2) < \frac1{10}$.
\end{lemma}

\begin{proof}
For every $i\in \{0,\ldots,47\}$ we define $s_{i}:=\sup\{t\in [a,b]\colon (\gamma)_1(t)=\tfrac{i}{48}\}$ and
 $t_{i}:=\inf\{t\in [s_i,b]\colon (\gamma)_1(t)=\tfrac{i+1}{48}\}$.  Since $\gamma$ is a path from $\{0\}\times[0,2]$ to $\{1\}\times[0,2]$, $s_i$ and $t_i$ are well-defined
 and $s_i,t_i\in [a,b]$ with $s_i<t_i$.

Since $a\le s_{i}<t_{i}\le s_{i+1}\le b$ for all $i\in  \{0,\ldots,47\}$, the intervals from the collection $\{[s_{4\nu},t_{4\nu+3}]:\nu\in\{0,\ldots,11\}\} $ have mutually disjoint interiors. Moreover, 
\[
\sum_{\nu=0}^{11}V_{s_{4\nu}}^{t_{4\nu+3}}\big((\gamma)_2\big)\le V_a^b\big((\gamma)_2\big)<1,
\]
so there are
at least 2 indices $\nu_1,\nu_2\in\{0,\ldots,11\}$, $\nu_1<\nu_2$, with $V_{s_{4\nu_k}}^{t_{4\nu_k+3}}\big((\gamma)_2\big)<10^{-1}$ for $k \in \{1,2\}$. Thus for each $k \in \{1,2\}$ there exist $j_k \in \{0,\ldots,18\}$ with
\[
\gamma([s_{4\nu_k},t_{4\nu_k+3}]) \subset \Big[\frac{4\nu_k}{48},\frac{4\nu_k+4}{48} \Big] \times
 \Big[\frac{j_k}{10},\frac{j_k+2}{10} \Big].
\]
By the definition of $R$ (resp.\ $R_2$) for every $k\in \{1,2\}$ there is at least one index
\(
i_k \in 
\{4\nu_k,4\nu_k+2\}
\)
with $\{(i_k,j_k),(i_k,j_k+1)\} \subseteq R_2$ (see Figure~\ref{fig:R2} for an illustration). 
\begin{figure}[h]
\begingroup%
  \makeatletter%
  \providecommand\color[2][]{%
    \errmessage{(Inkscape) Color is used for the text in Inkscape, but the package 'color.sty' is not loaded}%
    \renewcommand\color[2][]{}%
  }%
  \providecommand\transparent[1]{%
    \errmessage{(Inkscape) Transparency is used (non-zero) for the text in Inkscape, but the package 'transparent.sty' is not loaded}%
    \renewcommand\transparent[1]{}%
  }%
  \providecommand\rotatebox[2]{#2}%
  \newcommand*\fsize{\dimexpr\f@size pt\relax}%
  \newcommand*\lineheight[1]{\fontsize{\fsize}{#1\fsize}\selectfont}%
  \ifx\svgwidth\undefined%
    \setlength{\unitlength}{397.57699765bp}%
    \ifx\svgscale\undefined%
      \relax%
    \else%
      \setlength{\unitlength}{\unitlength * \real{\svgscale}}%
    \fi%
  \else%
    \setlength{\unitlength}{\svgwidth}%
  \fi%
  \global\let\svgwidth\undefined%
  \global\let\svgscale\undefined%
  \makeatother%
  \begin{picture}(1,0.40654764)%
    \lineheight{1}%
    \setlength\tabcolsep{0pt}%
    \put(0,0){\includegraphics[width=\unitlength,page=1]{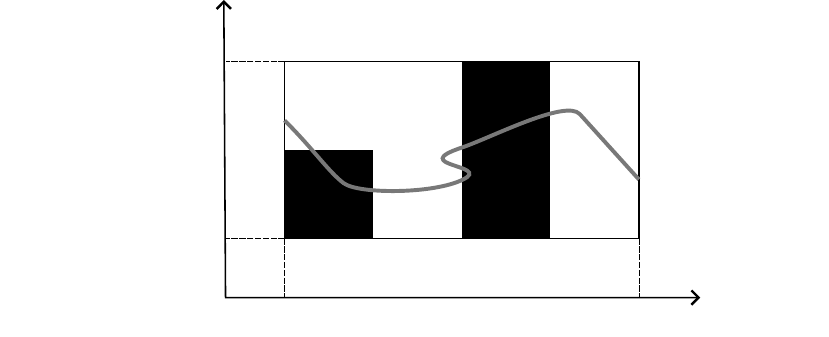}}%
    \put(0.34359481,0.00616036){\makebox(0,0)[t]{\lineheight{1.25}\smash{\begin{tabular}[t]{c}$\frac{4\nu_k}{48}$\end{tabular}}}}%
    \put(0.771383,0.00616036){\makebox(0,0)[t]{\lineheight{1.25}\smash{\begin{tabular}[t]{c}$\frac{4\nu_k+4}{48}$\end{tabular}}}}%
    \put(0.23822017,0.11480119){\makebox(0,0)[t]{\lineheight{1.25}\smash{\begin{tabular}[t]{c}$\frac{j_k}{10}$\end{tabular}}}}%
    \put(0.23822017,0.32869528){\makebox(0,0)[t]{\lineheight{1.25}\smash{\begin{tabular}[t]{c}$\frac{j_k+2}{10}$\end{tabular}}}}%
    \put(0.39004079,0.26953954){\makebox(0,0)[t]{\lineheight{1.25}\smash{\begin{tabular}[t]{c}$\gamma(s_{4\nu_k})$\end{tabular}}}}%
    \put(0.52183988,0.23853291){\makebox(0,0)[t]{\lineheight{1.25}\smash{\begin{tabular}[t]{c}$\gamma(s_{i_k})$\end{tabular}}}}%
    \put(0.70008496,0.28927335){\makebox(0,0)[t]{\lineheight{1.25}\smash{\begin{tabular}[t]{c}$\gamma(t_{i_k})$\end{tabular}}}}%
    \put(0.83461353,0.18315013){\makebox(0,0)[t]{\lineheight{1.25}\smash{\begin{tabular}[t]{c}$\gamma(t_{4\nu_k+3})$\end{tabular}}}}%
    \put(0,0){\includegraphics[width=\unitlength,page=2]{bedf-mcmull-part.pdf}}%
  \end{picture}%
\endgroup%
 \caption{An illustration of a $4\times 2$ subblock of the pattern in Figure~\ref{fig:BMC-imp}. The black squares correspond to elements $(i,j)\in R$. Each such subblock contains a vertical column of two black squares. \label{fig:R2}}
\end{figure}
With this choice we have $\gamma|_{[s_{i_k},t_{i_k}]} \subseteq f_{(i_k,j_k)}([0,1]\times[0,2])$. Thus, setting $a_k=s_{i_k}$ and $b_k=t_{i_k}$, the path $\gamma_k:=\gamma|_{[a_k,b_k]}$ satisfies the assertions of the lemma (note that $i_1<i_2$ by construction).
\end{proof}

Now let $\gamma\colon [0,1]\to \R^2$ be a path connecting the points 
$(0,\tfrac{1}{2})$ and $(1,\tfrac{1}{2})$ with  $V_0^1\big((\gamma)_2\big)<\tfrac{1}{2}$.
Let $a=a_{0,1}:=\sup\{t\in [0,1]\colon (\gamma)_1(t)=0\}$ and $b=b_{0,1}:=\inf\{t\in [a,1]\colon (\gamma)_1(t)=1\}$. Then $\gamma|_{[a,b]}$ satisfies the assumptions of \Cref{lemma:zoltan-neu} and thus, 
by iteratively applying this lemma, there exist double sequences $\big((i_{l,k},j_{l,k})\big)_{l \ge 0,1\le k \le 2^l}$ in $R_2$ and 
$\big([a_{l,k},b_{l,k}]\big)_{l\in \N,1\le k \le 2^l}$ 
of intervals, such that $\big([a_{l,k},b_{l,k}]\big)_{1\le k \le 2^l}$ are disjoint for every $l$, $\bigcup_{k=1}^{2^l}[a_{l,k},b_{l,k}]\subset \bigcup_{k=1}^{2^{l-1}}[a_{l-1,k},b_{l-1,k}]$ and 
\[
\gamma\Big(\bigcup_{k=1}^{2^l}[a_{l,k},b_{l,k}]\Big)\subset \bigcup_{k=1}^{2^l}f_{(i_{1,k},j_{1,k})}\circ\cdots\circ f_{(i_{l,k},j_{l,k})} ([0,1]\times[0,2])
\subset \bigcup_{\mathbf{j}\in R^l}f_{\mathbf{j}}([0,1]^2),
\]
where the $f_{(i_{l,k},j_{l,k})}$ map into $[0,1]^2$ since $\gamma$ is contained in $[0,1]^2$ (otherwise its variation would exceed $\tfrac{1}{2}$).
Note that  $(\gamma)_1\Big(\bigcup_{k=1}^{2^l}[a_{l,k},b_{l,k}]\Big)$ is the disjoint union of $2^l$ closed intervals of length $\frac1{48^l}$. Thus 
\[
\bigcap_{l\ge 0}(\gamma)_1\Big(\bigcup_{k=1}^{2^l}[a_{l,k},b_{l,k}]\Big)
\]
is a Cantor set, and, hence, 
\[
\bigcap_{l\ge 0}\gamma\Big(\bigcup_{k=1}^{2^l}[a_{l,k},b_{l,k}]\Big)\subset \bigcap_{l\ge 0}\bigcup_{\mathbf{j}\in R^l}f_{\mathbf{j}}([0,1]^2)= K.
\]
is uncountable.
Note that the $\|\cdot\|_1$-length of $\gamma$ is $V_0^1((\gamma)_1)+V_0^1((\gamma)_2)$. We have shown that if  $\gamma\colon [0,1]\to \R^2$ is a path from 
$(0,\tfrac{1}{2})$ to $(1,\tfrac{1}{2})$ with $V_0^1(\gamma_2)<\tfrac{1}{2}$, then $\gamma\cap K$ is uncountable, and thus $K$ 
is impermeable with respect to the $1$-norm on $\R^2$. 
From 
impermeability of $K$ with respect to 
the 1-norm follows impermeability with respect to the Euclidean norm, {\em cf.}\ \Cref{rem:non-convex-norm}.

It is not hard to see that $K$ is totally disconnected. Since it is also compact, by \cite[II. Section 3.2, Theorem 5]{o2012general} we have $ \dim_T(K)=0$. We thus have shown the following proposition.

\begin{proposition}
There exists an impermeable set in $\R^2$ which is compact, has Lebesgue measure~0 and topological dimension~0.
\end{proposition}

The example is easily generalized to several dimensions: using the same $n\times m$-pattern $R$ as above, consider
\[
\tilde K_d:=K\times [0,1]^{d-2}= \Big\{\sum_{k=1}^\infty  \{T^k x_k\}\times [0,1]^{d-2}\colon x \in R^\N\Big\}\subset [0,1]^d,
\] 
which provides us with an example of a set of topological dimension $d-2$ and measure 0, which is impermeable.
Alternatively, we may write 
\[
\tilde K_d= \Big\{\sum_{k=1}^\infty  \tilde T^k x_k\colon x \in (R\times \{0,\ldots,9\}^{d-2})^\N\Big\}\subset [0,1]^d,
\]
where $\tilde Ty = (48^{-1}y_1,10^{-1}y_2,\ldots,10^{-1}y_2)$ for $y\in \R^d$. Now define 
\[
P_k\colon \R^d\to \R^d\colon\; (y_1,\ldots, y_d)\mapsto(y_1,y_{2 + (k \pmod{d-1})},\ldots,y_{2 + (k+d-2 \pmod{d-1})}). 
\]
The purpose of $P_k$ is to cyclically shift the last $d-1$ coordinates of an element of $\R^d$ by $k$ places modulo $d-1$. Define  
\[
K_d:= \Big\{\sum_{k=1}^\infty  \tilde T^k P_k x_k\colon x \in (R\times \{0,\ldots,9\}^{d-2})^\N\Big\}\subset [0,1]^d.
\] 
Then $K_d$ is compact and totally disconnected, therefore has topological dimension 0 (by \cite[II. Section 3.2, Theorem 5]{o2012general}), has measure 0 and, along the same lines as for the 2-dimensional case, one can show that it  is impermeable. We end by stating this result as a theorem, adding the trivial case $d=1$, where the middle third Cantor set is a suitable example.

\begin{theorem}\label{thm:bedf-mcmullen-d}
For each $d\ge 1$ there exists an impermeable set in $\R^d$ which is compact, has Lebesgue measure~0 and topological dimension~0.
\end{theorem}

\bigskip

\subsection*{Acknowledgement.} We thank the anonymous referee for her/his careful reading of the manuscript.

\bibliographystyle{amsplain}  
\providecommand{\bysame}{\leavevmode\hbox to3em{\hrulefill}\thinspace}
\providecommand{\MR}{\relax\ifhmode\unskip\space\fi MR }
\providecommand{\MRhref}[2]{%
  \href{http://www.ams.org/mathscinet-getitem?mr=#1}{#2}
}
\providecommand{\href}[2]{#2}

\end{document}